\documentclass{article}
\usepackage{enumerate}
\usepackage{bbm}

\usepackage{amsthm}
\usepackage{amsmath}
\usepackage{color}

\usepackage[colorlinks]{hyperref}
\usepackage{colonequals}
\usepackage{mathrsfs}
\usepackage[utf8]{inputenc}
\usepackage{dsfont}
\usepackage{cleveref}
\usepackage{amssymb}
\usepackage{bbold}

\usepackage[style=alphabetic,maxnames=200,maxalphanames=6,giveninits]{biblatex}

\usepackage[shortcuts]{extdash}

\numberwithin{equation}{section}

\def\R{\mathbb{R}}
\def\N{\mathbb{N}}
\def\T{\mathbb{T}}
\def\cP{\mathbb{P}}

\def\M{\mathcal{M}}

\def\cD{\mathcal{D}}
\def\cS{\mathcal{S}}

\def\cH{\mathcal{H}}
\def\cU{\mathcal{U}}

\def\cF{\mathcal{F}}

\def\TGRE{\mathcal{TGRE}}
\def\TCRE{\mathcal{TCRE}}

\def\aL{\mathsf L}
\def\aM{\mathsf M}
\def\aaE{\mathsf E}

\def\aS{\mathsf S}
\def\aZ{\mathsf Z}
\def\az{\mathsf z}

\def\aJ{\mathsf J}
\def\aB{\mathsf B}
\def\QL{Q_{\sf L}}
\def\QB{Q_{\sf B}}
\def\DL{D_{\sf L}}
\def\aR{\mathsf R}
\def\av{\mathsf v}
\def\aw{\mathsf w}
\def\aT{\mathsf T}
\def\ax{\mathsf x}

\def\hf{\mathbb{f}}

\def\cI{\mathcal{I}}

\def\DL{\mathcal{D}_{\aL}}
\def\DB{\mathcal{D}_{\aB}}
\def\Ccinfty{C^\infty_c}

\def\deta{\mathrm{d}\eta}

\def\Bgrad{\overline\nabla}

\newcommand{\mtt}[1]{\mathtt{#1}}
\def\grad{\nabla}
\newcommand{\norm}[1]{\left\lVert#1\right\rVert}
\newcommand{\Lp}[2]{L_{#1}^{#2}}
\def\Lgrad{\widetilde\nabla}

\def\dsigma{\mathrm{d}\sigma}
\newcommand{\abs}[1]{\left\lvert#1\right\rvert}
\newcommand{\eps}{\varepsilon}
\newcommand{\da}{\downarrow}

\def\Z{{\mathbb Z}}
\def\G{\R^{4d}}
\def\GL{\R^{4d}}
\def\Doa{\Omega\times\R^d}

\def\GB{\G\times S^{d-1}}

\def\cJ{\mathcal{J}}
\def\cA{\mathcal{A}}

\def\supp{\operatorname{supp}}

\def\loc{{\text{loc}}}

\def\Lip{\operatorname{Lip}}

\newcommand{\defeq }{\mathop{=}\limits^{\textrm{def}}}
\def\d{\partial}
\def\Do{\R^{2d}}
\def\Dot{\T^d\times\R^d}
\newcommand{\dd}{{\,\rm d}}
\newcommand{\limit}[1]{\xrightarrow{#1}}

\theoremstyle{plain}
        \newtheorem{theorem}{Theorem}[section]

        \newtheorem{example}[theorem]{Example}

        \newtheorem{assumption}[theorem]{Assumption}
         
        \newtheorem{proposition}[theorem]{Proposition}
        \newtheorem{lemma}[theorem]{Lemma}
        \newtheorem{corollary}[theorem]{Corollary} 
        \newtheorem{definition}[theorem]{Definition} 
        \newtheorem{remark}[theorem]{Remark}  

\usepackage{todonotes}

\title{On a fuzzy Landau Equation: Part III. The grazing collision limit}

\author{Manh Hong Duong \and Boris Golubkov
  \and  Zihui He
}

\newcommand{\Addresses}{{
  \bigskip
  \footnotesize

  H.~Duong, \textsc{School of Mathematics, University of Birmingham, UK}\par\nopagebreak
  \textit{E-mail}: \texttt{h.duong@bham.ac.uk}

   \medskip

  B.~Golubkov, \textsc{Fakult\"at f\"ur Mathematik, Universit\"at Bielefeld, Postfach 100131, 33501 Bielefeld, Germany}\par\nopagebreak
  \textit{E-mail}: \texttt{boris.golubkov@math.uni-bielefeld.de}

  \medskip

  Z.~He, \textsc{Fakult\"at f\"ur Mathematik, Universit\"at Bielefeld, Postfach 100131, 33501 Bielefeld, Germany}\par\nopagebreak
  \textit{E-mail}: \texttt{zihui.he@uni-bielefeld.de}

}}
\begin{document}

\maketitle

\begingroup
\renewcommand{\thefootnote}{}
\endgroup

\begin{abstract}

In this paper, we study the grazing limit from the non-cutoff fuzzy Boltzmann equations to the fuzzy Landau equation, where particles interact through delocalised collisions. We show the grazing limit through variational formulations that correspond to the GENERIC (General Equations for Non-Equilibrium Reversible–Irreversible Coupling) structure of the respective equations.  
 We show that the variational formulation associated with a non-quadratic dual dissipation pair for the fuzzy Boltzmann equations converges to a variational formulation of the fuzzy Landau equation corresponding to a quadratic dissipation pair.

\end{abstract}

\tableofcontents


\section{Introduction}

In this article, we study the grazing collision limit from the {\it fuzzy} Boltzmann equations to the {\it fuzzy} Landau equation. The {\it fuzzy} Boltzmann equation can be written as
\begin{equation}
  \label{FBE}  
     \left\{ 
     \begin{aligned}
      &\d_t f+v\cdot \nabla_x f = Q_{\aB}(f,f)\\
         & Q_{\aB}(f,f) = \int_{\Do\times S^{d-1}}
\kappa(x-x_*)B(|v-v_*|,\sigma)\big(f'f'_*-ff_*\big)\dd \sigma\dd x_*\dd v_*.
     \end{aligned}
     \right.
\end{equation}
The unknown $f_t(x,v):[0,T]\times\Do\to\R_+$ with $d\ge2$ corresponds to the density of particles at time $t\in[0,T]$ and position $x\in\R^d$ and velocity $v\in \R^d$. The linear transport term $v\cdot\nabla_x f$ describes the advection of the density by the velocity of the particles. The collision term $Q_{\aB}(f,f)$ describes the {\it delocalised} elastic collision of the particles with the pre-collision status $(x,v)$ and $(x_*,v_*)$. The post-collision velocities $v'$ and $v_*'$ are given by 
\begin{align*}
    v'=\frac{v+v_*}{2}+\frac{|v-v_*|}{2}\sigma,\quad v_*'=\frac{v+v_*}{2}-\frac{|v-v_*|}{2}\sigma,\quad   \sigma\in S^{d-1}.
\end{align*}
We use the notations
\begin{equation*}
f=f(x,v),\quad f_*=f(x_*,v_*),\quad f'=f(x,v'),\quad f'_*=f(x_*,v'_*).
\end{equation*}
The collision kernel $B:\R^d\times S^{d-1}\to\R_+$ is given by
\begin{equation*}
\label{kernel:B}
   B(|v-v_*|,\sigma)=A_0(|v-v_*|)b(\theta),
\end{equation*}
where the kinetic kernel takes the form of  
\begin{equation}
\label{kernel:A0}
     A_0(|v-v_*|)=|v-v_*|^{\gamma}, \quad \gamma\in[-2,1].
\end{equation}
When $d=2$, we consider $\gamma\in(-2,1]$. In the hard and Maxwellian potential cases, $A_0$ can also take the form $A_0(|v-v_*|)\sim \langle v-v_*\rangle^\gamma$, $\gamma\in(-\infty,1]$.
The deviation angle $\theta\in [0,\pi/2]$ is given by 
\begin{equation*}
\label{kernel:B:theta}
  \theta=\arccos \frac{v-v_*}{|v-v_*|}\cdot \sigma.
\end{equation*}

The {\it delocalised} collision is characterised by a spatial interaction kernel $\kappa$. For example, one choice of $\kappa$ could be a renormalised Gaussian kernel proportional to $\exp (-\langle x\rangle )$, where $\langle z\rangle:=\sqrt{1+|z|^2}$ denotes the Japanese bracket. The fuzzy equation \eqref{FBE} is introduced in \cite{EH25} as an approximation of the classical inhomogeneous Boltzmann equation 
\begin{equation}
  \label{BE-cla}  
     \left\{ 
     \begin{aligned}
      &\d_t f+v\cdot \nabla_x f = Q_{\sf class}(f,f)\\
         & Q_{\sf class}(f,f) = \int_{\R^d\times S^{d-1}}
B(|v-v_*|,\sigma)\times\\
&\hspace{2.5cm}\big(f(x,v')f(x,v'_*)-f(x,v)f(x,v_*)\big)\dd \sigma\dd v_*.
     \end{aligned}
     \right.
\end{equation}
One major difficulty in studying the classical Boltzmann equations is that the natural a priori bound—the conservation of mass, given by the 
$L^1$-norm of $f$, is fundamentally incompatible with the $L^2$-integrability in the spatial variable 
$x$ required by the quadratic collision term $Q_{\sf class}(f,f)$. The so-called renormalised solutions are studied in, for example, \cite{DL89,AV02}. In \cite{EH25b}, Erbar and the third author showed that as $\kappa(x-x_*)\to \delta_0(x-x_*)$ a Dirac measure, the fuzzy equation \eqref{FBE} indeed converges to the classical equation \eqref{BE-cla}. Here, we do not take the singular limit in $\kappa$, further assumptions of $\kappa$ can be found in Assumption \ref{ASS:kappa}.

We consider the singular 
 angle function $\beta(\theta)\defeq\sin\theta ^{d-2}b(\theta)$ such that $\supp(\beta)\subset[0,\pi/2]$,
\begin{equation}
\label{beta}
    \beta(\theta) = \sin\theta^{d-2}b(\theta) \gtrsim \theta^{-1-\nu}\quad\text{and}\quad \int_0^{\frac{\pi}{2}}\beta(\theta)\theta^2\dd\theta<+\infty,
\end{equation}
for some $\nu\in(0,2)$.
Further assumptions of $B$ and $\beta$ can be found in Section \ref{sec:main-thm}.
We take the following scaling for  $\eps \in (0,1)$
\begin{equation}
\label{beta-eps}
\beta^\varepsilon(\theta)={\pi^3}/{\varepsilon^3}\beta\Big(\frac{\pi\theta}{\varepsilon}\Big).
\end{equation}
The scaling induces a new collision operator $Q_{\aB}^\eps$ corresponding to the collision kernel $B^\varepsilon=A_0(|v-v_*|) b^\varepsilon(\theta)$ and formally the following \textit{grazing collision limit}
\begin{equation}
\label{intro:GL}
\partial_tf^\eps + v\cdot\grad_xf^\eps = Q_\aB^\eps(f^\eps,f^\eps) \limit{\eps \to 0} \partial_tf + v\cdot\grad_xf = Q_{\aL}(f,f),
\end{equation}
that the {\it fuzzy} Boltzmann equation converges to the {\it fuzzy} Landau equation
\begin{equation}
    \label{FLE}
    \left\{
    \begin{aligned}
    &\d_t f+v\cdot\nabla_x f=Q_{\sf L}(f,f)\\
    &Q_{\sf L}(f,f)=\nabla_v \cdot\int_{\Do}\kappa(x-x_*)A(|v-v_*|)\Pi_{(v-v_*)^\perp}\big(f_*\nabla_v f-f\nabla_{v_*} f_*\big)\dd x_*\dd v_*,
    \end{aligned}
    \right.
\end{equation}
where the kinetic kernel $A(|v-v_*|)=A_0(|v-v_*|)|v-v_*|^2$, and the orthogonal projection $\Pi_{(v-v_*)^\perp}$ is given by 
\begin{equation*}
\label{projection}
  \Pi_{(v-v_*)^\perp}=\operatorname{Id}-\frac{(v-v_*)\otimes (v-v_*)}{|v-v_*|^2}.
\end{equation*}


Similar to the difficulty encountered in studying the classical inhomogeneous Boltzmann equation \eqref{BE-cla}, the $L^2$-integrability in $x$ is incompatible with mass conservation when analysing the classical inhomogeneous Landau equation, which corresponds to the fuzzy equation \eqref{FLE} with $\kappa$ taken as a Dirac measure.
The fuzzy Landau equation \eqref{FLE} was introduced in \cite{DH25} as an approximation to the classical inhomogeneous Landau equations and also has been studied further in \cite{DH25b,gualdani2025fuzzy}. The grazing limit \eqref{intro:GL} of the classical inhomogeneous Boltzmann equation was shown on the formal level by  \textcite{Des92}. 
Simplified models of the inhomogeneous Boltzmann equation \eqref{BE-cla} and the Landau equation are the spatial {\it homogeneous} equations obtained by removing the spatial variable $x$. For $f=f(v)$, the homogeneous Boltzmann and Landau equations write as follows 
\begin{equation}
    \label{homo}
\begin{gathered}
    \d_t f (v)= \int_{\R^d\times S^{d-1}}
B(|v-v_*|,\sigma)\big(f(v')f(v'_*)-f(v)f(v_*)\big)\dd \sigma\dd v_*,\\
\d_t f(v)=\nabla_v \cdot\int_{\R^d}A(|v-v_*|)\Pi_{(v-v_*)^\perp}\big(f(v_*)\nabla f(v)-f(v)\nabla_{*} f(v_*)\big)\dd v_*.
\end{gathered}
\end{equation}
As a consequence, the linear transport term disappears, and the quadratic dependence on $x$ also vanishes in the homogeneous collision operators.
Concerning the rigorous proof of the grazing limit in the {\it homogeneous} case, the hard potential cases were studied in \cite{AB90}, the soft potential cases were shown in \cite{Vil98b}, where the framework of $\cH$-solutions was introduced. The weak convergence was improved in \cite{ADVW00} by taking advantage of the angular singularity. Combining with the so-called velocity average techniques in \cite{GLPS88,DLM91}, the grazing limit in the {\it inhomogeneous} case was shown in \cite{AV02}.
{Recently, the grazing limit for the quantum Boltzmann operator was studied in \cite{GPTW25}, and for wave kinetic equations in \cite{DH25C}.}


The spatially homogeneous Boltzmann and Landau equations \eqref{homo} can be formulated as gradient flows, or steepest descents, of the dissipative Boltzmann entropy 
$\cH(f)=\int f\log f$. A rigorous variational characterisation of these gradient-flow structures was established in \cite{erbar2023gradient} and \cite{carrillo2024landau}, respectively. Using these gradient flow structures and the \textit{$\Gamma-$convergence} framework for gradient flows developed in \cite{SASE11}, \textcite{carrillo2022boltzmann} showed the grazing limit in the homogeneous case \eqref{homo}. 

The {\it fuzzy} Boltzmann equation \eqref{FBE} and the Landau equations \eqref{FLE} have been rigorously cast into the GENERIC (General Equations for Non-Equilibrium Reversible Irreversible Coupling) framework, which is a combination of both conservative and dissipative dynamics, in \cite{EH25,DH25} using the variational formulation for a general GENERIC system formally proposed in \cite{DPZ13}. Motivated by the aforementioned results on the grazing limit of the classical Boltzmann equations, in this paper, by extending the method in \cite{carrillo2022boltzmann} for the {\it homogeneous} equations to the {\it fuzzy} models and combining with velocity-averaging techniques to obtain the necessary compactness, we rigorously establish the grazing limit in the fuzzy case using a variational framework that captures the GENERIC structures of the fuzzy equations developed in \cite{EH25,DH25}. In particular, we show that the grazing limit of the fuzzy Boltzmann equations—whose dynamics are governed by a pair of non-quadratic ($\cosh$)-dissipation potentials, 
converges to the fuzzy Landau equations, which admit a pair of quadratic dissipation potentials. 






\subsection{GENERIC structure}\label{subsec:GENERIC}

The GENERIC framework provides a systematic method to derive thermodynamically consistent evolution equations describing complex systems consisting of both Hamiltonian (reversible) dynamics and gradient (irreversible) flows. It was first introduced in the context of complex fluids  \cite{GO97a,GO97b}. 
A  GENERIC system describes the evolution of an unknown $\mathsf z$ in a state space $\mathsf Z$ via the equation
\begin{equation}\label{eq:generic}
\partial_t \az = \aL \dd\aaE+\aM \dd\aS.
\end{equation}
The building block $\{\aaE,\aS,\aL,\aM\}$ satisfy the following conditions:
\begin{itemize}
\item The functionals $\aaE,\aS:\aZ\to \R$ are interpreted as, energy and entropy functionals  respectively; and $\dd\aaE$, $\dd\aS$ are their differentials.
\item $\aL(\az)$, $\az\in \aZ$ are antisymmetric operators mapping cotangent to tangent vectors and satisfying the Jacobi identity.

\item $\aM(\az)$, $\az\in \aZ$ is a symmetric and positive semi-definite operator.
\item The following degeneracy conditions are satisfied:
\begin{equation*}\label{eq:degen}
\aL(\az) \dd \aS(\az)=0\quad\text{and}\quad \aM(\az) \dd\aaE(\az) = 0\quad\text{for all }\az.
\end{equation*}
\end{itemize}
The conditions satisfied by the building blocks $\{\aaE, \aS, \aL,\aM\}$ ensure that in any solution to \eqref{eq:generic}, the energy $\aaE$ is conserved and the entropy $\aS$ is non-decreasing
\begin{align*}
\frac{\dd}{\dd t} \aaE(\mathsf z_t) =0\quad\text{and}\quad \frac{\dd}{\dd t} \aS(\az_t) \ge0.
\end{align*}
Thus, the first and second laws of thermodynamics are automatically satisfied for GENERIC systems. The first author, Peletier and Zimmer \cite{DPZ13} propose a variational characterisation: a curve $\az:[0,T]\to\aZ$ is a solution to \eqref{eq:generic} if and only if $J(\mathsf z)=0$, where 
\begin{equation}\label{intro:eq:J}
J(\mathsf z) = \aS(\az_0)-\aS(\az_T) +\frac12 \int_0^T \|\partial_t \mathsf z -\aL\dd\aaE\|_{\aM^{-1}}^2\dd t +\frac12 \int_0^T\|\dd\aS\|_{\aM}^2 \dd t,
\end{equation}
where the weighted (pseudo) norms are associated with the weighted $L^2$-inner products
\begin{align*}
 \langle \av,\aw\rangle_{\aM^{\pm 1}}=\int \av\cdot \aM^{\pm1} \aw.
\end{align*}
When $\aaE=0$, the system \eqref{eq:generic} reduces to a gradient flow and the variational characterisation \eqref{intro:eq:J} corresponds to De Giorgi’s characterisation of gradient flows as curves of maximal slope, see for example \cite{AGS08}. 

The above (quadratic) framework can be generalised by inducing a pair of dual dissipation potentials $\aR:\aZ\times\aT\aZ\to\R_+$ and  $\aR^*:\aZ\times\aT^*\aZ\to\R_+$ defined on (co-)tangential spaces that are convex conjugate
\begin{equation}\label{intro:eq:J-R}
\aJ(\mathsf z) = \aS(\az_0)-\aS(\az_T) + \int_0^T \aR(\partial_t \mathsf z -\aL\dd\aaE)\dd t +\int_0^T\aR^*(\dd\aS)\dd t.
\end{equation}
Similar to the quadratic case, $\aJ=0$ leads to the generalised GENERIC system \cite{mielke2011formulation}
\begin{equation}
\label{GENERIC-R}
 \d_t \az=\aL\dd\aaE+\d \aR^*(\dd\aS), 
\end{equation}
where the irreversible part  $\aM\dd \aS$ in \eqref{eq:generic} is replaced by 
\begin{align*}
\partial_{\av}\aR^*(\av)\big\vert_{\av=\dd\aS}.
\end{align*}
The positivity of $\aJ$ is ensured by Fenchel's inequality
\begin{equation}
\label{Rxv}
    \ax\cdot\av\le \aR(\ax)+\aR^*(\av)\quad\text{with equality iff}\quad \ax=\d\aR^*(\av).
\end{equation}
\subsection{Variational formulations}

The fuzzy Boltzmann equation \eqref{FBE} and the fuzzy Landau equation \eqref{FLE} are associated with, at least formally, the conservative energy and the dissipative Boltzmann entropy 
\begin{equation}
\label{intro-EH}
  \aaE(f)=\frac12\int_{\Do} |v|^2f\dd x\dd v \quad\text{and}\quad \cH(f)=\int_{\Do} f\log f\dd x\dd v.
\end{equation}
The dissipation of the entropy is characterised by the entropy inequalities
\begin{align}
 & \cH(f_T)-\cH(f_0) +\int_0^T\cD(f)\dd t \le 0.\label{H:B}
\end{align}
The Boltzmann and Landau entropy dissipations are given by 
\begin{gather}
\cD_{\aB}(f)=\frac14\int_{\G\times S^{d-1}}\kappa B(f'f_*'-ff_*)\overline\nabla \log f\dd x_*\dd v_*\dd x\dd v \dd \sigma\geq 0,\label{DB}\\
\cD_{\sf L}(f)=\frac12\int_{\G}\kappa ff_*\big|\widetilde\nabla \log f\big|^2\dd x_*\dd v_*\dd x\dd v\geq 0,\label{DL}
\end{gather}
where the Boltzmann and Landau gradients are given by $\overline\nabla f=f_*'+f'-f_*-f$ and $\widetilde\nabla f= \sqrt{A}\Pi_{(v-v_*)^\perp}\big(\nabla_vf-(\nabla_{v}f)_*\big)$. The non-negativity of $\cD_B$ follows from the elementary inequality $(\log x-\log y)(x-y)\geq 0$.

The fuzzy Boltzmann equations can be 
cast as a generalised quadratic GENERIC system with building blocks $\{\aaE,\aS,\aL,\d\aR^*\}$, where $\aaE$ is given by \eqref{intro-EH} and $\aS=-\cH$, and  the operator $\aL$ is given by 
\begin{equation}
\label{intro:L}
\begin{aligned}
\aL(f)g=-\nabla\cdot(f \aJ\nabla g),\quad \aJ=\begin{pmatrix}
    0& \mathsf{id}_d\\
    -\mathsf{id}_d&0
\end{pmatrix},
\end{aligned}
\end{equation}
for all $g\in C^\infty_c(\Do)$, where $\nabla=(\nabla_x,\nabla_v)^T$. The dissipation potential pair $\aR^*,\aR$ is determined by a pair of even, lower semi-continuous convex conjugated functions $\Psi,\Psi^*:\R\to\R_+$ and a $1$-homogeneous concave function $\Theta:\R_+\times \R_+\to\R_+$ satisfying Assumption \ref{ass-pair}, in particular, the following compatibility condition holds
\begin{align*}
 (\Psi^*)' (\log s -\log t) \Theta (s,t) =s-t,\quad\forall s,t>0.  
\end{align*}
Notice that $\dd\aS=-\log f$. We set
\begin{gather*}
\aR(U)=\frac14\int_{\G\times S^{d-1}}\Psi\Big(\frac{U}{\Theta(f)B\kappa}\Big)\Theta(f)B\kappa ,\\ 
\cD_{\Psi^*}(f)\defeq\aR^*(- \log f)=\frac14\int_{\G\times S^{d-1}}\Psi^*(-\overline\nabla \log f)\Theta(f)B\kappa .
\end{gather*}
In the above, we write $\Theta(f)=\Theta(f'f_*',ff_*)$. The corresponding variational functional \eqref{intro:eq:J-R} writes
\begin{equation}
\label{J-Rnonquadratic}
\aJ_{\aB}(f):=\cH(f_T)-\cH(f_0) +\int_0^T\cD_{\Psi^*}(f)\dd t+\int_0^T\aR(\d_t f-v\cdot\nabla_x f)\dd t.
\end{equation}
In particular, one can take the important class of hyperbolic cosine functions
\begin{equation}
\label{intro-cosh}
    \begin{aligned}
\Psi^*(r)=4(\cosh(r/2)-1)\quad\text{and}\quad \quad\Theta(s,t)=\sqrt{st}.
    \end{aligned}
    \end{equation}
    
The rigorous proof of the generalised GENERIC structure characterised by the variational formulation \eqref{J-Rnonquadratic} was shown in \cite{EH25} that $\aJ_{\aB}=0$ if and only if $f$ solves the fuzzy Boltzmann equation \eqref{FBE}, where the cut-off (bounded angular kernel) hard potential case that $A_0\sim \langle v-v_*\rangle^\gamma$, $\gamma\in(-\infty,1]$ were considered. Although we lack the rigorous proof of the GENERIC structures in the non-cutoff and soft potential cases, the variational inequality $\aJ_{\aB}\le 0$ is ensured by the entropy inequality \eqref{H:B} and the duality condition \eqref{Rxv}.

The fuzzy Landau equation \eqref{FLE} can be cast as a GENERIC system with the building block  $\{\aaE,\aS,\aL,\widetilde\aM\}$, where $\aaE$ and $\aL$ are given by \eqref{intro-EH} and \eqref{intro:L}, $\aS=-\cH$, and $\widetilde\aM$ is given by 
\begin{equation*}
\begin{aligned}
\widetilde\aM(f)g=-\frac{1}{2}\widetilde{\nabla}\cdot\Big(\kappa ff_*\widetilde{\nabla} g\Big),
\end{aligned}
\end{equation*}
where $\widetilde{\nabla}\cdot$ is the $L^2$-dual of $\widetilde\nabla$ given by \eqref{landau-div}. In \cite{DH25}, the first and third authors rigorously showed this GENERIC structure via a variational characterisation corresponding to \eqref{intro:eq:J} under appropriate curve assumptions.

Taking the scaling \eqref{beta-eps}, 
inspired by the homogeneous argument in \cite{carrillo2022boltzmann}, to show the grazing limit \eqref{intro:GL}, we show that as $\varepsilon\to0$, the fuzzy Boltzmann variational inequality 
\begin{equation}
\label{J-B}
\cH(f^\varepsilon_T)-\cH(f^\varepsilon_0) +\int_0^T\cD_{\Psi^*}^\varepsilon(f^\varepsilon)\dd t+\int_0^T\aR^\varepsilon(\d_t f^\varepsilon-v\cdot\nabla_x f^\varepsilon)\dd t\le0
\end{equation}
leads to the fuzzy Landau variational inequality
\begin{equation}
\label{J-L}
\cH(f_T)-\cH(f_0) +\frac12\int_0^T\cD_{\aL}(f_t)\dd t+\frac12\int_0^T\|\d_t f-v\cdot\nabla_x f\|_{\tilde\aM^{-1}}^2\dd t \le 0.
\end{equation}

We show that the grazing limit of the fuzzy Boltzmann equations—whose dynamics are governed by a pair of non-quadratic dissipation potentials, converges to the fuzzy Landau equations, which admit a pair of quadratic dissipation potentials.

The grazing limit for the homogenous case in \cite{carrillo2022boltzmann} only deals with the quadratic admissible pair in the Boltzmann equation, which corresponds to $\Psi^*(r)=\Psi(r)=\frac{r^2}{2}$ and $\Theta(s,t)=\frac{s-t}{\log s-\log t}$ (the logarithm mean). Our results also apply to the homogeneous cases, where one can take non-quadratic pairs in the variational inequality corresponding to the gradient flows of homogeneous Boltzmann equations.  In particular, in \eqref{J-B}, we can take the admissible $\cosh$-pair given in \eqref{intro-cosh}.
The (generalised) GENERIC formulation \eqref{GENERIC-R} using the $\cosh$-pair dissipation functionals \eqref{intro-cosh} for the Boltzmann equation was shown formally in \cite{grmela1993weakly,grmela2002reciprocity,grmela2010multiscale,grmela2018generic}. This specific non-quadratic variational formulation for dissipative systems has recently received considerable attention in the PDE and probability communities due to its intimate connections to large deviation principle of stochastic processes \cite{mielke2014relation,mielke2017non,peletier2022jump, peletier2023cosh,kraaij2020fluctuation,duong2023non}. We do not derive a large deviation principle for the fuzzy Boltzmann and Landau equations in this paper, which would be interesting topics for future research, but refer the reader to the papers \cite{leonard1995large,rezakhanlou1998large,bouchet2020boltzmann,bodineau2023statistical,feliachi2024dynamical,feliachi2021dynamical,bodineau2023statistical} for relevant discussions on the classical models.



\medskip

The rest of the paper is organised as follows. In Section \ref{sec:main-results}, we present the fuzzy Landau and fuzzy Boltzmann equations in detail and state the main result, Theorem \ref{thm:main}. In Section \ref{sec:pre}, we introduce a polar-coordinate representation of the Boltzmann collision operator and provide the necessary pointwise limits and uniform bounds. In Section \ref{sec:conv}, we establish the compactness of solutions to the fuzzy Boltzmann variational inequality,  $f^\varepsilon$, in \eqref{J-B}. Finally, in Section \ref{sec:dissipation} and Section \ref{sec:action}, we pass to the limit in the variational inequalities \eqref{J-B}.  

This article is the third in a series of works on the fuzzy Landau equation. We refer to the first two papers on the variational characterisation and the solvability results of the fuzzy Landau equations in \cite{DH25} and \cite{DH25b} respectively.

\subsection*{Acknowledgements}
M. H. D is funded by an EPSRC Standard Grant EP/Y008561/1. Z.~H. is funded by the Deutsche Forschungsgemeinschaft (DFG, German Research Foundation) – Project-ID 317210226 – SFB 1283.

The authors are grateful to Matthias Erbar and Jeremy Wu for their helpful discussions and comments.

\section{The equations and main results}\label{sec:main-results}

Throughout this work, we consider the kernels $A_0=A_0(|v-v_*|)$ taking the following forms
\begin{equation}
    \label{A-1}
     A_0(|v-v_*|)=|v-v_*|^{\gamma},\quad\gamma\in[-2,1].
\end{equation}
When $d=2$, we consider $\gamma\in(-2,1]$.
In the hard and Maxwellian potential case, the kernel $A_0$ can also take the following form
\begin{equation}
    \label{A-3}
    0<C^{-1}\langle v-v_*\rangle^{\gamma}\le A_0(|v-v_*|)\le C\langle v-v_*\rangle^{\gamma},\quad \gamma\in(-\infty,1].
\end{equation}

In this article, the domain can be taken $\Do$ or $\Dot$. We take $T>0$ arbitrarily fixed.

In this section, we present the fuzzy Landau and fuzzy Boltzmann equations in Section \ref{sec:FL} and  Section \ref{sec:FBE}, respectively, where the GENERIC formulation and variational characterisation are also discussed. In particular, Section \ref{sec:non-quadratic} addresses the entropy inequality of fuzzy Boltzmann equations associated with non-quadratic dissipation potential pairs. Finally, in Section \ref{sec:main-thm}, we present the main Theorem \ref{thm:main} and the strategy of the proof, where the grazing-collision limit of the fuzzy Boltzmann equations corresponding to a non-quadratic entropy–dissipation pair converges to the fuzzy Landau equation characterised by a quadratic entropy–dissipation pair.
\subsection{Fuzzy Landau equations}\label{sec:FL}
We recall the fuzzy Landau equation \eqref{FLE}
\begin{equation*}
    \left\{
    \begin{aligned}
    &\d_t f+v\cdot\nabla_x f=Q_{\sf L}(f,f)\\
    &Q_{\sf L}(f,f)=\nabla_v \cdot\int_{\Do}\kappa(x-x_*)A(|v-v_*|)\Pi_{(v-v_*)^\perp}\big(f_*\nabla_v f-f\nabla_{v_*} f_*\big)\dd x_*\dd v_*,
    \end{aligned}
    \right.
\end{equation*}
and the kernel $A(v-v_*)=A_0(|v-v_*|)|v-v_*|^2$. We define the Landau gradient $\widetilde\nabla$ as follows
\begin{align*}
    \widetilde\nabla f= \sqrt{A}\Pi_{(v-v_*)^\perp}\big(\nabla_vf-(\nabla_{v}f)_*\big).
\end{align*}
For $G:\G\to\R^d$, the divergence $\widetilde\nabla\cdot$ is given by 
\begin{equation}
\label{landau-div}
    \widetilde\nabla\cdot G=\nabla_v\cdot \int_{\Do}\sqrt{A}\Pi_{(v-v_*)^\perp} (G-G_*)\dd x_*\dd v_*
\end{equation}
via the integration by parts formula
\begin{align*}
\int_{\G}G\cdot \widetilde\nabla f\dd\eta=- \int_{\Do}\widetilde\nabla \cdot G  f\dd x\dd v, 
\end{align*}
where we use the shortened notation $\dd\eta=\dd x_*\dd v_*\dd x\dd v$ to denote the Lebesgue measure on $\R^{4d}$.

Thus, the fuzzy Landau equation \eqref{FLE} can be written as
\begin{equation*}
    \d_t f+v\cdot\nabla_x f=\frac12\widetilde \nabla\cdot \big(\kappa ff_* \widetilde \nabla \log f\big).
\end{equation*}
The following mass, momentum and energy conservation laws at least formally hold
\begin{equation}
\label{consv-law}
\begin{aligned}
\int_{\Doa} (1,v,|v|^2)f_t\dd x\dd v=\int_{\Doa}(1,v,|v|^2) f_0\dd x\dd v\quad\forall t\in[0,T].
\end{aligned}
\end{equation}
The Boltzmann entropy of the fuzzy Landau equation is non-increasing over time, and the following inequality holds, at least formally,
\begin{align*}
\cH(f_T)-\cH(f_0)=-\int_0^T\cD_{\sf L}(f_t)\dd t\leq 0,
\end{align*}
where the entropy dissipation is given by \eqref{DL}.

We follow the homogeneous cases in \cite{Vil98} to define $\cH$-solutions to the fuzzy Landau equations. The existence of the $\cH$-solutions, the conservation laws and entropy inequalities have been discussed in \cite{DH25b}. Let $\cP(\Do)$ denote the space of probability measures on $\Do$.
\begin{definition}[$\cH$-solution]\label{def:introH}
    Let $\gamma\in[-2,1]$. Let $(f_t)_{t\in[0,T]}\subset\cP(\Do)$ by weakly continuous in time. Let $f\in L^1([0,T];L^1_{2,2+\gamma_+}(\Do))\cap L^\infty([0,T];L\log L(\Do))$. We say $f$ is an $\cH$-solution of \eqref{FLE} if the following entropy dissipation is time-integrable
\begin{align*}
    \int_0^T\DL(f_t)\dd t<+\infty,
\end{align*}
and the following weak formulation holds 
    \begin{equation}
    \label{Landau:weak}
     \begin{aligned}
&\int_{\Do}f_0\varphi(0)\dd x\dd v+\int_{0}^T\int_{\Do}f(\d_t\varphi+v\cdot\nabla_x\varphi)\dd x\dd v\dd t\\
&=-\frac12\int_0^T\int_{\G}\kappa ff_* \widetilde\nabla \varphi\cdot \widetilde\nabla\log f\dd\eta\dd t
\end{aligned}
 \end{equation}
 for all $\varphi\in C^\infty_c([0,T)\times\Do)$.
\end{definition}
In the above, we use the notation $\gamma_+=\max(\gamma,0)$. For $a,\,b\in\R$, $ L^1_{a,b}(\Do)$ denotes the functional space
\begin{align*}
    L^1_{a,b}(\Do)&=\{f\in L^1(\Do)\mid \|f\|_{L^1_{a,b}(\Do)}<+\infty\},
    \end{align*}
with the norms
\begin{align*}
\|f\|_{L^1_{a,b}(\Do)}\defeq\int_{\Do}\big(\langle x\rangle^a+\langle v\rangle^b \big)f(x,v)\dd x\dd v.
\end{align*}

\subsubsection{A variational characterisation}
In \cite{DH25}, the first and third authors show a variational characterisation that recasts the fuzzy Landau equation \eqref{FLE} within the framework of GENERIC systems with the building blocks $\{\widetilde\aM,\,\aL,\,\aaE,\,\aS\}$, where the energy and entropy are given by
\begin{equation}
\label{ES}
\aaE(f)=\int_{\Do} \frac{|v|^2}{2} f\dd x\dd v\quad\text{and}\quad \aS(f)=-\cH(f),
\end{equation}
and the operators $\aL$ and $\widetilde\aM$ at $f\in\aZ$ are given by
\begin{equation}
\label{ML}
\begin{aligned}
&\widetilde\aM(f)g=-\frac{1}{2}\widetilde{\nabla}\cdot\Big(\kappa ff_*\widetilde{\nabla} g\Big)\quad\text{and}\quad \aL(f)g=-\nabla\cdot(f \aJ\nabla g)
\end{aligned}
\end{equation}
for all $g\in\aZ$ and $\aJ$ is given in \eqref{intro:L}.

To characterise the speed of curves alone the characteristic line $\|\d_t f(x+tv)\|_{\tilde\aM^{-1}}^2$ in \eqref{J-L}, we define the transport grazing rate equation
\begin{equation}
\label{intro:TGRE}
    \d_t f+v\cdot \nabla_x f+\frac12 \widetilde \nabla \cdot U_{\aL}=0,
\end{equation}
where $U_{\aL}:[0,T]\times \GL\to \R^d$ denotes the grazing rate. Notice that, when $U_{\aL}=-\kappa ff_* \widetilde\nabla \log f$ equation \eqref{intro:TGRE} recovers the fuzzy Landau equation \eqref{FLE}. We define an admissible pair for \eqref{intro:TGRE}.
\begin{definition}\label{def:TGRE} We say $(f_t,U_t)$ is a pair of solutions to the transport grazing rate equation \eqref{intro:TGRE} if 
\begin{enumerate}
    \item $f_t:[0,T]\to \cP(\Do)$ is weakly continuous; 
    \item $(U_t)_{t\in[0,T]}$ is a family of Borel measures in $\M(\R^{4d};\R^d)$;
    \item $\int_0^T\int_{\G} |U_t|\dd \eta \dd t<+\infty$;
    \item For all $\varphi\in C^\infty_c(\Do)$ the following equality holds
    \begin{equation*}
        \frac{d}{dt}\int_{\Do}\varphi f\dd x\dd v-\int_{\Do}v\cdot \nabla\varphi f\dd x\dd v=\frac12\int_{\G} \widetilde\nabla \varphi\cdot U\dd\eta.
    \end{equation*}
    We define the set of such pairs $(f_t,U_t)$ by $\TGRE_T$.
\end{enumerate}
\end{definition}
We define the following curve action associated to \eqref{intro:TGRE}
\begin{align*}
\cA_{\aL}(f,U_{\aL})=\frac12\int_{\G}\frac{|U_{\aL}|^2}{ff_*\kappa} \dd\eta.
\end{align*}
We have the following variational characterisation of the $\cH$-solutions to the fuzzy Landau equation \eqref{FLE}.
\begin{theorem}[\cite{DH25}]\label{thm:FLE:we}
Let $\gamma\in[-2,1]$. Let $(f,U)\in \TGRE_T$ such that $\cH(f_0)<+\infty$.
We have
\begin{equation}
\label{V-Landau}
\cJ_{\aB}(f,U)\defeq\cH(f_T)-\cH(f_0) +\frac12\int_0^T\cD_{\aL}(f_t)\dd t+\frac12\int_0^T\cA_{\aL}(f_t,U_t)\dd t \ge 0.
\end{equation}
Moreover, we have $\cJ_{\aB}(f,U)=0$ if and only if $f$ is an $\cH$-solution for the fuzzy Landau equation \eqref{FLE}.
\end{theorem}

\begin{remark}
In \cite{DH25}, Theorem \ref{thm:FLE:we} was proved only for the kernel $A_0$ of the form \eqref{A-1}, namely $A_0(|v-v_*|)=|v-v_*|^\gamma$ and the curves $f$ satisfying certain $L^p$-bounds. Here, we do not specify the kernel and curve assumptions in the statement of the theorem. On the one hand, such assumptions are unnecessary for showing the grazing limit in the main Theorem \ref{thm:main}.
On the other hand, we expect that the kernel and curve restrictions in \cite{DH25} can be removed in future work. 
\end{remark}

\subsection{Fuzzy Boltzmann equation}\label{sec:FBE}
We recall the \textit{fuzzy}  Boltzmann equation \eqref{FBE}
\begin{equation*}
\left\{
\begin{aligned}
&\d_t f+v\cdot \nabla_x f=Q_{\aB}(f,f)\\
&\QB(f,f)=\int_{\Do\times S^{d-1}}
\big(f'f'_*-ff_*\big)B(v-v_*,\sigma)\kappa(x-x_*)\dd \sigma\dd x_*\dd v_*,
\end{aligned}
\right.
\end{equation*}
and the collision kernel 
\begin{equation}
\label{boltzmann:kernel}
   B(|v-v_*|,\sigma)=A_0(|v-v_*|)b(\theta).
\end{equation}
In this article, we do not distinguish the notations of $B(|v-v_*|,\sigma)$ and $B(|v-v_*|,\theta)$, $b(\theta)$ and $b(\sigma)$. The deviation angle is given by $\theta=\arccos\frac{v-v_*}{|v-v_*|}\cdot\sigma\in [0,\pi/2]$.
We note that one can restrict the deviation angle on $[0,\pi/2]$ by symmetrising the collision angle \begin{align*}
B=\big[ B(|v-v_*|,\sigma)+ B(|v-v_*|,-\sigma)\big]\mathbb{1}_{\theta\in[0,\pi/2]}.    
\end{align*}

We define the Boltzmann gradient corresponding to $\sigma\in S^{d-1}$ by
\begin{equation}
\label{boltzmann-grad}
    \overline\nabla f= f'+f_*'-f-f_*.
\end{equation}
For $G:\G\times S^{d-1}\to \R$, the divergence $\overline\nabla\cdot$ is given by
\begin{align*}
\overline\nabla\cdot G=\int_{\Do\times S^{d-1}} \big(G+G_*-G'-G_*'\big)\dd \sigma\dd x_*\dd v_*   
\end{align*}
via the integration by parts formula
\begin{equation}
\label{Boltzmann-IP}
\int_{\G\times S^{d-1}}\overline\nabla f\cdot G \dd\eta\dd\sigma=-\int_{\Do} f\overline\nabla\cdot G\dd x\dd v.
\end{equation}
In the above, we write
\begin{align*}
    G_*=G(x_*,x,v_*,v),\quad G'=G(x,x_*,v',v_*'),\quad G_*'=G(x_*,x,v',v_*').
\end{align*}
The fuzzy Boltzmann equation \eqref{FBE} can be written as
\begin{align*}
    \d_t f+v\cdot\nabla_x f=\frac14 \overline\nabla\cdot\big(\kappa B \Lambda(f)\overline\nabla \log f\big).
\end{align*}
 In the above, we define $\Lambda(f)=\Lambda(f'f'_*,ff_*)$ and for all $s,t>0$
 \begin{equation}
     \label{def:log-mean}
     \Lambda(s,t):=(s-t)/(\log s-\log t)\quad\text{denotes the logarithmic mean}.
 \end{equation}

Similar to the fuzzy Landau equations, at least on the formal level, we have  the mass, momentum and energy conservation laws \eqref{consv-law}, and the following entropy identity 
\begin{equation}
\label{B:entropy}
\cH(f_T)-\cH(f_0)=-\int_0^T\cD_{\aB}(f_t)\dd t\leq 0,
\end{equation}
where $\cD_{\aB}$ is given by \eqref{DB}.
We define an $\cH$-solution for the fuzzy Boltzmann equations analogues to  Definition \ref{def:introH} with the weak formulation \eqref{Landau:weak} replaced by 
    \begin{equation}
    \label{weak:FBE}
     \begin{aligned}
&\int_{\Do}f_0\varphi(0)\dd x\dd v+\int_{0}^T\int_{\Do}f(\d_t\varphi+v\cdot\nabla_x\varphi)\dd x\dd v\dd t\\
&=-\frac14\int_0^T\int_{\G\times S^{d-1}}\kappa B \Lambda(f) \overline\nabla \varphi\cdot \overline\nabla\log f\dd\eta\dd\sigma\dd t.
\end{aligned}
 \end{equation}

Based on the existence results of the fuzzy Boltzmann in the cut-off hard potential case in \cite{EH25}, we have the following existence results.
\begin{theorem}[Existence of $\cH$-solutions]\label{thm:existence}
Let $f_0\in\cP(\Do)\cap L^1_{2,2+\gamma_+}(\Do)$ such that $\cH(f_0)<+\infty$. There exists at least one $\cH$-solution to the fuzzy Boltzmann equation \eqref{FBE} such that 
\begin{equation*}
\|f_t\|_{L^1_{2,2+\gamma_+}(\Do)}\le C(T, \|f_0\|_{L^1_{2,2+\gamma_+}(\Do)})\quad\text{for all }t\in[0,T],
\end{equation*}
and the following entropy inequality holds
\begin{equation}
\label{B:in-entropy}
\cH(f_T)-\cH(f_0)+\int_0^T\cD_{\aB}(f_t)\dd t\le 0. 
\end{equation}
\end{theorem}
\begin{proof}
We combine the existence results in \cite{Vil98b} and \cite{EH25}, and sketch the argument below.

We approximate the collision kernel $B$ by a sequence of bounded, monotone increasing kernels $B^n\ge0$, and consider the corresponding approximation equation of \eqref{FBE} by replacing $B$ with $B^n$
\begin{equation}
\label{app:FBE}
    \left\{
    \begin{aligned}
        &\d_tf^n+v\cdot\nabla_x f^n=Q_{\aB}^n(f^n,f^n)\\
        &f^n|_{t=0}=f_0.
    \end{aligned}\right.
\end{equation}
The existence of $\cH$-solutions to \eqref{app:FBE} in the Maxwellian potential and bounded angular kernel setting was obtained in \cite[Theorem 3.2]{EH25}. Moreover, the following uniform bounds hold
\begin{equation}
    \label{exits:bound}
\begin{gathered}
    \|f^n_t\|_{L^1_{0,2}(\Do)}\le \|f_0\|_{L^1_{0,2}(\Do)},\quad \|f^n_t\|_{L^1_{2,0}(\Do)}\le C(T)\|f_0\|_{L^1_{2,2}(\Do)}\\
    \cH(f^n_T)-\cH(f_0)+\int_0^T\cD^n_{\aB}(f^n_t)\dd t\le 0,\\
    \text{and}\quad \|f^n_t\|_{L^1_{0,2+\gamma_+}(\Do)}\le C(T, \|f_0\|_{L^1_{0,2}(\Do)})\|f_0\|_{L^1_{0,2+\gamma_+}(\Do)}.
\end{gathered}
\end{equation}
Here, the constants $C\ge0$ are independent of $n$. The last inequality follows from the Povzner inequality \cite[Theorem 4]{Elm83},
\begin{align*}
\overline\nabla |v|^{s}\lesssim \cos\theta(|v|^{s-1}|v_*|+|v||v_*|^{s-1}),\quad s>0,    
\end{align*}
and the assumption \eqref{beta} $\int_0^{\frac{\pi}{2}}\beta(\theta)\theta^2\dd\theta<+\infty$.

The uniform estimates \eqref{exits:bound} also imply a uniform bound $\|f^n\log f^n\|_{L^1(\Do)}\le C$, see, for example \cite{EH25b}. For any $t\in[0,T]$, the sequence $\{f^n_t\}$ is therefore precompact in $L^1(\Do)$ by the Dunford--Pettis theorem. Using a diagonal extraction argument in time, we obtain (up to a subsequence) $f^n\rightharpoonup f$ in $L^1(\Do)$ for all $t\in[0,T]$. Finally, the entropy inequality \eqref{B:in-entropy} holds for the limit since 
$\cH(f)$ and $\cD(f)$ are lower semicontinuous in $f$, and by the monotonicity of $B^n$.

 The weak formulation holds in the soft-potential case $\gamma\in[-2,0)$, following the ideas of $\cH$-solutions:
\begin{align*}
    \Big|\int_{\Do} Q_{\aB}(f,f)\phi\dd x\dd v\Big|\lesssim \Big(\int_{\G\times S^{d-1}} |\overline\nabla \phi|^2|v-v_*|^{\gamma}b(\sigma)\dd \eta\dd\sigma\Big)^{\frac12}\cD_{\aB}(f)^{\frac12},
\end{align*}
which is bounded as a consequence of the polar coordinate representation in Section \ref{sec:SC} and Lemma \ref{lem:phi:conv} that $|\overline\nabla \phi| \lesssim |\phi-\phi'| \lesssim \theta |v-v_*|$. For the same reason as in Remark \ref{rmk:compare}, we cannot treat $f-f_*$ as $|v-v_*|$, as is done in the homogeneous case in \cite{Vil98}. Consequently, this existence result cannot hold for the very soft-potential regime $\gamma\in(-4,-2)$.
\end{proof}

\subsubsection{Dual dissipation potentials}\label{sec:non-quadratic}

In this subsection, we discuss the (formal) non-quadratic GENERIC formulation and the corresponding entropy inequalities for the fuzzy Boltzmann equations \eqref{FBE}.

For all $V$ and $U:\G\to\R$, by Cauchy--Schwarz inequality, we have
\begin{equation*}
    \label{C-S}
\begin{aligned}
    |VU|&\le\Big|V \frac{U}{\Lambda(f)B\kappa}\Big|\Lambda(f)B\kappa\\
    &\le \frac12 |V|^2\Lambda(f)B\kappa+\frac12 \Big|\frac{U}{\Lambda(f)B\kappa}\Big|^2\Lambda(f)B\kappa,
\end{aligned}
\end{equation*}
and the inequality holds if and only if $U=V \Lambda(f)B\kappa$. Let
\begin{equation}
\label{def:U-B}
    U_{\aB}=-\Lambda(f)B\kappa \overline\nabla \log f.
\end{equation}
Applying the above argument to the entropy inequality \eqref{B:in-entropy}, the following quadratic entropy inequality holds
\begin{equation}
    \label{in-entropy:quadratic}
\begin{gathered}
\cH(f_T)-\cH(f_0)+ \frac18\int_0^T\int_{\G\times S^{d-1}}|-\overline\nabla \log f|^2\Lambda(f)B\kappa\dd\eta\dd \sigma\dd t\\
+\frac18\int_0^T\int_{\G\times S^{d-1}}\Big|\frac{U_{\aB}}{\Lambda(f)B\kappa}\Big|^2\Lambda(f)B\kappa\dd\eta\dd \sigma\dd t\le 0.   
\end{gathered}
\end{equation}
The quadratic pair in \eqref{in-entropy:quadratic} can be generalised to non-quadratic pairs. Let $\Psi$ and $\Psi^*:\R\to\R_+$  be a dual pair of functions, and $\Theta:\R_+\times \R_+\to\R_+$ be a symmetric function satisfying Assumption \ref{ass-pair} below.
For $V$ and $U:\G\to\R$, we define the positive dual pair
\begin{align*}
\aR(U)&=\frac14\int_{\G\times S^{d-1}}\Psi\Big(\frac{U}{\Theta(f)B\kappa}\Big)\Theta(f)B\kappa \dd\eta\dd \sigma,\\ 
\aR^*(V)&=\frac14\int_{\G\times S^{d-1}}\Psi^*\big(\overline\nabla V\big)\Theta(f)B\kappa \dd\eta\dd \sigma.
\end{align*}
In the above, we write $\Theta(f)=\Theta(f'f_*',ff_*)$. To clarify the dependence of $\aR$ and $\aR^*$ of $f$, in the following, we also write $\aR(f,U)$ and $\aR^*(f,V)$.

Correspondingly,  we define the dissipation associated with $\Psi^*$
\begin{equation}
\label{def:D-Psi}
\cD_{\Psi^*}(f)\defeq\aR^*(- \log f)=\frac14\int_{\G\times S^{d-1}}\Psi^*(-\overline\nabla \log f)\Theta(f)B\kappa \dd\eta\dd \sigma.
\end{equation}
We replace the quadratic pair in \eqref{in-entropy:quadratic} by non-quadratic pairs. 
\begin{proposition}[Non-quadratic entropy dissipation inequality]
    \label{lem:non-quadratic}
Let $(\Psi^*,\Theta)$ satisfy Assumption \ref{ass-pair}. Let $f$ be an $\cH$-solution to the fuzzy Boltzmann equation \eqref{FBE} such that the entropy inequality \eqref{B:in-entropy} holds. The following non-quadratic entropy dissipation inequality
holds
\begin{equation}
\label{VC:R}
\cH(f_T)-\cH(f_0) +\int_0^T\cD_{\Psi^*}(f_t)\dd t+\int_0^T\aR(f_t,U_{\aB})\dd t \le 0,
\end{equation}
where $U_{\aB}=-B\kappa (f'f_*'-ff_*)$.
\end{proposition}
\begin{proof}
The entropy inequality \eqref{VC:R} holds as a direct consequence of Assumption \ref{ass-pair} and the optimality condition, such that the duality inequality holds as equality $|ab|=\Psi(a)+\Psi^*(b)$, as stated in Remark \ref{rmk:opt} below.    
\end{proof}

The entropy inequality \eqref{VC:R} corresponds to the (generalised) GENERIC formulation of the fuzzy Boltzmann equation \eqref{FBE} with the building block $\{\aL,\aaE,\aS,\d\aR^*\}$ in the sense of \cite{mielke2011formulation} that
\begin{equation}
\label{GENERIC-boltzmann}
 \d_t f=\aL\dd\aaE+\d \aR^*(f,\dd\aS),  
\end{equation}
where $\aL,\,\aaE,\,\aS$ are given as in \eqref{ES} and \eqref{ML}, and the irreversible part is given by 
\begin{align*}
    \partial_V\aR^*(f,V)\big\vert_{V=\dd\aS}.
\end{align*}
The GENERIC formulation of the fuzzy Boltzmann equations \eqref{FBE} with kernels of the form \eqref{A-3} in the angular cut-off case was rigorously derived in \cite{EH25} via a variational characterisation. From \eqref{intro:eq:J-R}, it follows that the solution curves of \eqref{GENERIC-boltzmann} satisfy the variational inequality
\begin{equation}
    \begin{aligned}
    \cH(f_T)-\cH(f_0) +\int_0^T\aR^*(f,\dd\aS)\dd t+\int_0^T\aR(f,\d_t f-\aL\dd\aaE)\dd t \le 0.    
    \end{aligned}
\end{equation}
To characterise the dual dissipation potential $\aR(f,\d_t f-\aL\dd\aaE)=\aR(f,\d_t f+v\cdot\nabla_xf)$, we define the transport collision rate equation as follows
\begin{equation}
\label{intro:TCRE}
    \d_t f+v\cdot \nabla_x f+\frac14 \overline \nabla \cdot U_{\aB}=0,
\end{equation}
where $U_{\aB}:[0,T]\times \G\times S^{d-1}\to \R$ denotes the collision rate. Notice that, when $U_{\aB}=-B\kappa (f'f_*'-ff_*)$ equation \eqref{intro:TCRE} recovers the fuzzy Boltzmann equation \eqref{FBE}. The set of admissible pairs $(f,U)\in\TCRE_T$ is defined analogously to Definition \ref{def:TGRE}
as those satisfying the transport collision rate equation \eqref{intro:TCRE} in the weak sense.

\begin{theorem}
    [\cite{EH25}]\label{thm:VC:Boltzmann}
    Let $B(|v-v_*|)\sim\langle v-v_*\rangle^\gamma$ and $\gamma\in(-\infty,1]$. Let $(f,U)\in \TCRE_T$ such that $\int(\langle x\rangle^2+\langle v\rangle^{2+\gamma_+})f<+\infty$ and $\cH(f_0)<+\infty$. 
Then we have
\begin{equation}
\label{V-Boltzmann}
\cJ_{\aB}(f,U):=\cH(f_T)-\cH(f_0) +\int_0^T\cD_{\Psi^*}(f_t)\dd t+\int_0^T\aR(f_t,U_t)\dd t \ge 0,
\end{equation}
Moreover, we have $\cJ_{\aB}(f,U)=0$ if and only if $f$ is an $\cH$-solution for the fuzzy Boltzmann equation \eqref{FBE}.
\end{theorem}
The proof of Theorem \ref{thm:VC:Boltzmann} does not directly apply to the kernels of the form $B(|v-v_*|,\theta)=|v-v_*|^\gamma b(\theta)$ and non-cutoff $b(\theta)$, due to technical difficulties. Nevertheless, we expect that a variational characterisation of the fuzzy Boltzmann equations can also be established for other types of kernels.

\medskip

In the following, we state the assumptions on $\Psi^*$ and $\Theta$.
\begin{assumption}\label{ass-pair}
\begin{enumerate}[(1)]
    \item  Let $\Psi^*\in C^\infty(\R;\R_+)$ be a convex, superlinear and even function such that
 $\Psi^*(0)=0$ and $(\Psi^*)''(0)=1$.
There exist $C_1,C_2>0$ such that  $\Psi^*(r)\le C_1\exp(C_2 |r|)$.
 

 \item Let $\Theta\in C(\R_+\times\R_+;\R_+)$ be a joint concave function satisfying:
 \begin{itemize}
     \item symmetry $\Theta(s,t)=\Theta(t,s)$ for all $s,t\in\R_+$;
     \item positivity and normalisation $\Theta(s,t)>0$ if  $s\neq t$ and $\Theta(1,1)=1$;
     \item positive $1$-homogeneity $\Theta(\lambda s,\lambda t)=\lambda \Theta(s,t)$ for all $s,t\in\R_+$ and $\lambda\ge 0$;
     \item monotonicity $\Theta(r,s)\le \Theta(r,t)$ for all $r\in\R_+$ and $0\le s\le t$;
     \item behaviour at $0$, $\Theta(0,t)=0$ for all $t\in\R_+$;
 \end{itemize}
 \item Let $(\Psi^*,\Theta)$ satisfy the compatibility conditions:
 \begin{itemize}
     \item For all $s,t>0$, we have
     \begin{equation}
         \label{con-cpt}
         (\Psi^*)'(\log s-\log t)\Theta(s,t)=s-t;
     \end{equation}
     \item The function
     \begin{align*}
         G_{\Psi^*}(s,t)\defeq\frac14 \Psi^*(\log s-\log t)\Theta(s,t)\quad \forall s,t>0
     \end{align*}
    is convex and lower semi-continuous, and $ G_{\Psi^*}(s,t)=0$ if and only if $s=t$. 
 \end{itemize}
\end{enumerate}
Let $\Psi$ be the convex conjugate of $\Psi^*$, and $\Psi$ is strictly convex, strictly increasing, superlinear, even and $\Psi(0)=0$. As a direct consequence of convexity, we have
\begin{align*}
    r\mapsto \frac{\Psi(r)}{r}\quad\text{is non-decreasing}.
\end{align*}
\end{assumption}


\begin{remark}[The dissipation $\cD_{\Psi^*}(f)$]
\label{rmk:opt}

\begin{itemize}
\item We recall the definition \eqref{def:D-Psi} that
\begin{align*}
\cD_{\Psi^*}(f)=\frac14\int_{\G\times S^{d-1}}\Psi^*(-\overline\nabla \log f)\Theta(f)B\kappa \dd\eta\dd \sigma.
\end{align*}
The convexity of $\Psi^*$ and $\Psi^*(0)=0$ imply that, for all $s>t>0$,
\begin{align*}
0\le \Psi^*(\log s-\log t)\le (\Psi^*)'(\log s-\log t)(\log s-\log t).
\end{align*}
Combining with the compatibility condition  \eqref{con-cpt}, we have 
\begin{align*}
\Psi^*(\log s-\log t)\Theta(s,t)\le (s-t)(\log s-\log t).
\end{align*}
Hence, we have 
\begin{equation}
\label{bdd:D-Psi}
\cD_{\Psi^*}(f)\le \frac14\int_{\G\times S^{d-1}}|\overline\nabla \log f|^2\Lambda(f)B\kappa \dd\eta\dd \sigma=\cD_{\aB}(f).
\end{equation}
\item 
    By using the convex duality, we have
\begin{align*}
    \frac14|(\log s-\log t) U|&=\frac14\Big|(\log s-\log t) \frac{U}{\Theta(s,t)}\Big|\Theta(s,t)\\
    &\le \frac14\Psi\Big(\frac{U}{\Theta(s,t)}\Big)\Theta(s,t)+\frac14\Psi^*(\log s-\log t)\Theta(s,t).
\end{align*}
As a consequence of the compatibility condition \eqref{con-cpt}, the equality holds if and only if $U=(\Psi^*)'(\log s-\log t)\Theta(s,t)=s-t$.
As a direct consequence, we have
\begin{align*}
\cD_{\aB}(f)&=\frac14\int_{\G\times S^{d-1}}|-U_{\aB}\overline\nabla \log f|\dd\eta\dd \sigma\\
&=\frac14\int_{\G\times S^{d-1}}\Big[\Psi\Big(\frac{U_{\aB}}{\Theta(f)B\kappa}\Big)+\Psi^*\big(-\overline\nabla\log f\big)\Big]\Theta(f)B\kappa \dd\eta\dd \sigma\\
&=\aR(f,U_{\aB})+\aR^*(f,-\log f),
\end{align*}
where the optimal $U_{\aB}$ is given by \eqref{def:U-B}
\begin{equation*}
    U_{\aB}=(\Psi^*)'(-\overline\nabla \log f)\Theta(f)B \kappa=-(ff_*-f'f_*')B\kappa.
\end{equation*}
Hence, the entropy inequality \eqref{VC:R} holds.

\end{itemize}

\end{remark}

\begin{example}
\label{example}
    We have the following examples of the typical choice of $\Psi, \Psi^*$ and $\Theta$:
\begin{itemize}
    \item (Quadratic pair). We have
    \begin{equation}
    \label{p-1}
    \Psi(r)=\Psi^*(r)=\frac12 |r|^2\quad\text{and}\quad \Theta(s,t)=\Lambda(s,t).
    \end{equation}
    In this case, we have 
    \begin{align*}
\cD_{\Psi^*}(f)=\frac12\cD_{\aB}(f),\quad \aR(f,U)=\frac18\int_{\G\times S^{d-1}}\frac{|U|^2}{\Lambda(f)B\kappa}\dd\eta\dd \sigma.
\end{align*}

In this case, the entropy inequality \eqref{VC:R} is coincident with \eqref{in-entropy:quadratic}.

    \item (Cosh pair). We have 
    \begin{equation}
    \label{p-2}
    \begin{aligned}
    &\Psi^*(r)=4(\cosh(r/2)-1),\quad\Theta(s,t)=\sqrt{st},\\
&\Psi(r)=2r\log\big(\frac{r+\sqrt{r^2+4}}{2}\big)-\sqrt{r^2+4}+4.
    \end{aligned}
    \end{equation}
    In this case, we have 
    \begin{equation}
    \label{D:cosh}
\cD_{\Psi^*}(f)=\frac12\int_{\G\times S^{d-1}}|\sqrt{f'f_*'}-\sqrt{ff_*}|^2B\kappa\dd\eta\dd \sigma\le \frac12 \cD_{\aB}(f),
\end{equation}
where we use the elementary inequality
\begin{equation}
\label{ele-ineq}
 |\sqrt{s}-\sqrt{t}|^2\le \frac14 (s-t)(\log s-\log t),\quad s,t>0. 
\end{equation}
\end{itemize}
For many jump processes, the $\cosh$-pair arises naturally from the study of large deviation principle for the associated empirical measure where the variational functional $\cJ_{\aB}(f,U)$  is (rescaling) of the rate functional, see for instance \cite{mielke2014relation,mielke2017non,peletier2022jump, peletier2023cosh,kraaij2020fluctuation,duong2023non} and \cite{leonard1995large,rezakhanlou1998large,bouchet2020boltzmann,BBBO21, bodineau2023statistical,feliachi2024dynamical,feliachi2021dynamical,bodineau2023statistical} for the classical Boltzmann equation. In particular, in \cite{feliachi2021dynamical} the authors formally derive the grazing limit of the Boltzmann equation to the Landau equation by passing to the limit in the corresponding rate functionals. 

We refer the reader to \cite{peletier2023cosh} for further properties of the cosh-pair.
\end{example}



In the following, we show an upper bound of $\Theta(s,t)$ that is used in Section \ref{sec:action}
\begin{lemma}\label{lem:mean}
For all $\Psi^*$ and  $\Theta$ satisfying Assumption \ref{ass-pair}, there exists a constant $C_{\Psi^*}>0$ such that 
\begin{equation}
\label{mean}
 \Theta(s,t)\le C_{\Psi *}(s+t)\quad\forall s,\,t>0.  
\end{equation}
\end{lemma}
\begin{proof}
In the case of $s=t$, the  positivity and normalisation of $\Theta$ implies that $\Theta(s,t)=\frac{s+t}{2}$. Without loss of generality, we assume $s>t>0$. By Compatibility condition \eqref{con-cpt}, to show \eqref{mean}, we only need to show that
\begin{equation*}
    (\Psi^*)'(r)\ge C_{\Psi^*}^{-1}\tanh(r/2)\quad\forall r\ge 0.
\end{equation*}
By assumption $(\Psi^*)''(0)=1$, there exists $\bar r>0$ such that
\begin{align*}
    (\Psi^*)'(r)\ge r/2\ge \tanh(r/2)\quad\forall r\in[0,\bar r]. 
\end{align*}
 Since $\Psi^*$ is an even function, $(\Psi^*)'$ is an odd function, which implies that $(\Psi^*)'(0)=0$. Together with the convexity of $\Psi^*$ we deduce that $(\Psi^*)'(r)\ge 0$ and monotonically increasing. combining with $0\le \tanh(r)\le 1$, we have 
\begin{align*}
    (\Psi^*)'(r)\ge (\Psi^*)'(\bar r)\tanh(r/2)\quad\forall r\in(\bar r,\+\infty). 
\end{align*}
Hence, \eqref{mean} holds by taking  $C_{\Psi^*}=\min(1/2,(\Psi^*)'(\bar r))$.

\end{proof}

\begin{remark}[$\Theta$ is determined by $\Psi^*$]
{In Assumption \ref{ass-pair} we list both the assumptions on the $\Psi^*$ and $\Theta$ for highlighting the required conditions; however, most of the assumptions of $\Theta$ actually follow from the compatibility condition \eqref{con-cpt} and assumptions on $\Psi^*$. }

In fact, from \eqref{con-cpt}, $\Theta$ can be found from $\Psi^*$ by
\[
\Theta(s,t)=\frac{s-t}{(\Psi^*)'(\log(s/t))}.
\]
From this formula, we can deduce
\begin{itemize}
    \item The symmetry of $\Theta$ is ensured since $\Psi^*$ is even and $(\Psi^*)'$ is odd
    \begin{align*}
     \Theta(t,s)=\frac{t-s}{(\Psi^*)'(\log(t/s))}=-\frac{s-t}{(\Psi^*)'(-\log(s/t))}= \Theta(s,t).  
    \end{align*}

\item  The positivity of $\Theta(s,t)$ is ensured since $\Psi^*$ is convex and even. Indeed, for all $s>t>0$, we have
\begin{align*}
(\Psi^*)'(\log(s/t))>(\Psi^*)'(0)=0,\quad\text{thus},\quad   \Theta(s,t)>0. 
\end{align*}

In addition, by letting $r=\log s-\log t$, $\Theta(s,t)$ can be written as
\begin{equation}
\label{Theta:e-x}
\Theta(s,t)=\frac{t(s/t-1)}{(\Psi^*)'(\log(s/t))}=t\frac{(e^r-1)}{(\Psi^*)'(r)}. 
\end{equation}
By assumption $(\Psi^*)''(0)=1$, as $s\rightarrow t$, we have $x\rightarrow 0$ and 
\begin{equation*}
\label{Theta:t-t}
\Theta(t,t)=\lim_{s\rightarrow t}\Theta(s,t)=t \lim_{x\rightarrow 0}\frac{e^x-1}{(\Psi^*)'(x)}=t \lim_{x\rightarrow 0}\frac{e^x}{(\Psi^*)''(x)}=t.    
\end{equation*}
In particular, we have $\Theta(1,1)=1$.

\item The homogeneity of $\Theta$ is given by 
\begin{align*}
\Theta(\lambda s,\lambda t)=\frac{\lambda  s-\lambda  t}{(\Psi^*)'(\log(\lambda  s/\lambda t))}=\frac{\lambda(s- t)}{(\Psi^*)'(\log(s/t))}=\lambda \Theta(s,t).    
\end{align*}

\item The vanishing of $\Theta(s,t)$ at $0$ is ensured by the super-linear property of $\Psi^*$ that $\lim_{r\rightarrow+\infty}(\Psi^*)'(r)=+\infty$. Indeed, we have 
\begin{align*}
\Theta(s,0)=\lim_{t\rightarrow 0}\Theta(s,t)=\frac{s}{\lim_{t\rightarrow 0}(\Psi^*)'(\log(s/t))}=0.    
\end{align*}

By using the representation \eqref{Theta:e-x}, the monotonicity of $\Theta(s,t)$ is ensured if, in addition, we assume that
\begin{align*}
    \frac{e^{r}-1}{(\Psi^*)'(r)}\quad\text{non-decreasing for all $r>0$}.
\end{align*}



\end{itemize}

\end{remark}



\subsection{Main results}\label{sec:main-thm}

We recall the Boltzmann collision kernel $B$ and the angle function $\beta$
\begin{equation*}
   B=A_0(|v-v_*|)b(\theta)\quad\text{and}\quad\beta(\theta)\defeq\sin\theta ^{d-2}b(\theta).
\end{equation*}
We consider $\beta$ satisfies the following assumption.
\begin{assumption}[Angular kernels]\label{ass:ang}
Let $\beta:[0,\pi/2]\to\R_+$. 
    \begin{itemize}
    \item For all $\delta>0$, we have 
    \begin{equation*}
        \sup_{\theta\in[\delta,\pi / 2]}\beta(\theta)<+\infty\quad \text{and}\quad \supp(\beta)\subset [0,\pi / 2].
    \end{equation*}
    \item There exists $\nu\in(0,2)$ and $C_0>0$ such that
    \begin{equation}
    \label{nu}
        C_0\theta^{-1-\nu}\le\beta(\theta)\quad \text{for all}\quad \theta\in[0,\pi/2].
\end{equation}
\item The angular momentum is finite
\begin{equation*}
\int_0^{\pi/2}\theta^2\beta(\theta)\dd \theta=8(d-1)/|S^{d-2}|,
\end{equation*}
where the constant on the right-hand side is chosen to normalise $\beta$.

For $d=2$, we take $|S^0|=2$. 
\end{itemize}
\end{assumption}

For $\varepsilon\in(0,1)$, we take the scaling 
\begin{equation}
\label{def:beta-epsilon}
\beta^\varepsilon(\theta)={\pi^3}/{\varepsilon^3}\beta\Big(\frac{\pi\theta}{\varepsilon}\Big).
\end{equation}
Notice that $\beta^\varepsilon(\theta)$ satisfies
\begin{equation}
\label{def:beta-int}
\supp (\beta^\varepsilon)\subset[0,\varepsilon/2]\quad\text{and}\quad  \int_0^{\pi/2}\theta^2\beta^\varepsilon(\theta)\dd \theta=8(d-1)/|S^{d-2}|.   
\end{equation}

We consider the kernels $B^\varepsilon$ and $\kappa^\varepsilon$ satisfying the following assumptions.
\begin{assumption}[Collision and spatial kernels]\label{ASS:kappa}
Let $\varepsilon\in(0,1)$. 
\begin{itemize}
    \item Let
\begin{equation}
\label{def:B}
   B^\varepsilon=A_0(|v-v_*|)b^\varepsilon(\theta)\quad\text{and}\quad\beta^\varepsilon(\theta)=\sin\theta ^{d-2}b^\varepsilon(\theta),
\end{equation}
where $\beta^\varepsilon$ is defined as in \eqref{def:beta-epsilon} and $\beta$ satisfy Assumption \ref{ass:ang}. In hard and Maxwellian potential cases, $A_0$ is given by 
\begin{equation}
    \label{A-hard}
    \begin{gathered}
    A_0(|v-v_*|)=|v-v_*|^{\gamma}\text{ with }\gamma\in[0,1]\quad\text{or}\\
    A_0(|v-v_*|)\sim \langle v-v_*\rangle^{\gamma}\text{ with }\gamma\in(-\infty,1].    \end{gathered}
\end{equation}
In the case of $ A_0\sim \langle v-v_*\rangle^{\gamma}$, $A_0$ also satisfying the cross-section assumption \eqref{cross-bd} below.
In soft potential cases, $A_0$ is given by 
\begin{equation}
    \label{A-soft}
     A_0(|v-v_*|)=|v-v_*|^{\gamma}\text{ with }\gamma\in[-2,0).
\end{equation}
When $d=2$, we take $\gamma\in(-2,0)$.

\item 
Let the spatial kernels 
\begin{equation*}
    \label{kappa}
    0\le \kappa^\varepsilon\le C_\kappa\quad\text{and}\quad \kappa^\varepsilon\to \kappa\quad\text{pointwisely}
\end{equation*}
for some constant $C_\kappa>0$ which is independent of $\varepsilon$.
Moreover, for any $B^R_x\subset \R^d$, we have 
\begin{equation}
    \label{kappa-2}
    0<C_0\le \kappa^\varepsilon,
\end{equation}
where $C_0=C_0(R)$ is independent of $\varepsilon$.
\end{itemize}
\end{assumption}
In principle, we regard 
$\kappa$ as a fixed kernel while allowing small perturbations.
For example, we may consider a sequence $\kappa^\varepsilon \sim \exp(-C\langle x\rangle)$ or $\kappa^\varepsilon \sim \langle x\rangle^{-\alpha}$ for some $\alpha>0$. In particular, the restriction \eqref{kappa-2} prevents $\kappa^\varepsilon$ from converging to a Dirac measure.

For $\varepsilon\in (0,1)$, we define the scaling fuzzy Boltzmann equation
\begin{equation}
\label{def:scaling}
\left\{
\begin{aligned}
&\d_t f^\varepsilon+v\cdot\nabla_x f^\varepsilon=\QB^\varepsilon(f^\varepsilon,f^\varepsilon),\\
&f^\varepsilon|_{t=0}=f^\varepsilon_0,
\end{aligned}
\right.
\end{equation}
where $\QB^\varepsilon$ is defined as
\begin{align*}
\QB^\varepsilon(f,f)=\int_{\Do\times S^{d-1}}\kappa^\varepsilon     B^\varepsilon\big(f'f_*'-ff_*\big)\dd x_*\dd v_*\dd\sigma.
\end{align*}
The corresponding entropy dissipation of \eqref{def:scaling} is given by 
    \begin{equation*}
\cD_{\aB}^\varepsilon(f^\varepsilon)=\frac14\int_{\G\times S^{d-1}}\kappa^\varepsilon B^\varepsilon \Lambda(f^\varepsilon)|\overline\nabla \log f^\varepsilon|^2\dd\eta\dd \sigma.
\end{equation*}

We consider the initial value $f^\varepsilon_0$ of \eqref{def:scaling} satisfies the following assumptions.
\begin{assumption}[Initial value assumption]
\label{ass:curve}
Let $\varepsilon\in(0,1)$. Let $\{f^\varepsilon_0\}\subset\cP(\Do)$ be a sequence of probability measures converging to $f_0\in \cP(\Do)$ in the weak*-topology, i.e., in the sense of duality with continuous functions vanishing at infinity, such that
\begin{equation}
\label{H-0}
\lim_{\varepsilon\to0}\cH(f^\varepsilon_0)= \cH(f_0)<+\infty.
\end{equation}
We assume the following  uniform bound
\begin{equation}
\label{moments}
\sup_{\varepsilon\in(0,1)}\|f^\varepsilon_0\|_{L^1_{2,2+\gamma_+}(\Do)}\le C
\end{equation}
for some constant $C>0$ independent of $\varepsilon$.
\end{assumption}

\begin{remark}[Uniform bounds on solution curves]\label{rmk:ass}
Let $(f^\varepsilon_t)\subset\cP(\Do) $ be $\cH$-solutions of \eqref{def:scaling} with initial values $f^\varepsilon_0$ satisfying Assumption \ref{ass:curve}, as provided by Theorem \ref{thm:existence}.
We recall the following entropy inequality
\begin{equation}
    \label{in-entropy-uni}
    \cH(f^\varepsilon_t)-\cH(f^\varepsilon_0)+\int_0^t\DB^\varepsilon(f^\varepsilon_s)\dd s\le0\quad\forall t\in[0,T].
\end{equation}
By Proposition \ref{lem:non-quadratic}, the entropy inequality \eqref{in-entropy-uni} implies that 
\begin{equation*}
\label{VC:R-uni}
\cH(f^\varepsilon_T)-\cH(f^\varepsilon_0) +\int_0^T\cD_{\Psi^*}^\varepsilon(f^\varepsilon_t)\dd t+\int_0^T\aR^\varepsilon(f^\varepsilon_t,U^\varepsilon_{\aB})\dd t \le 0,
\end{equation*}
where $U^\varepsilon_{\aB}=-\overline\nabla \log f^\varepsilon \Lambda(f^\varepsilon) B^\varepsilon\kappa^\varepsilon$, and $\cD_{\Psi^*}^\varepsilon$ and $\aR^\varepsilon$ are given by
\begin{gather*}
\aR^\varepsilon(f^\varepsilon, U)=\frac14\int_{\G\times S^{d-1}}\Psi\Big(\frac{U}{\Theta(f^\varepsilon)B^\varepsilon\kappa^\varepsilon}\Big)\Theta(f^\varepsilon)B^\varepsilon\kappa^\varepsilon \dd\eta\dd \sigma,\\ 
\cD_{\Psi^*}^\varepsilon(f^\varepsilon)=\frac14\int_{\G\times S^{d-1}}\Psi^*(-\overline\nabla \log f^\varepsilon)\Theta(f^\varepsilon)B^\varepsilon\kappa^\varepsilon \dd\eta\dd \sigma.
\end{gather*}  

As a consequence of Theorem \ref{thm:existence} and Assumption \ref{ass:curve}, the following uniform bounds hold
    \begin{equation}
    \label{uni-bdd:D}  
    \begin{gathered}
\sup_{\varepsilon\in(0,1)}\sup_{t\in[0,T]}\|f^\varepsilon_t\|_{L^1_{2,2+\gamma_+}(\Do)}\le C,\\
\sup_{\varepsilon\in(0,1)}\sup_{t\in[0,T]}\big|\cH(f^\varepsilon_t)\big|\le C\quad\text{and}\quad \sup_{\varepsilon\in(0,1)}\int_0^T\DB^\varepsilon(f^\varepsilon_t)\dd t\le C.     \end{gathered}
    \end{equation}
    The uniform bound on the entropy dissipation follows from \eqref{in-entropy-uni} together with the uniform lower bound $\cH(f^\varepsilon_T)\ge -C$, which in turn is ensured by the uniform bound of $\|f^\varepsilon\|_{L^1_{2,2}(\Do)}$, see, for example \cite{JKO98}.

\end{remark}

\medskip

{We show the following main theorem by showing the variational formulation associated with a non-quadratic dual dissipation pair for the fuzzy Boltzmann equations converges to a variational formulation of the fuzzy Landau equation corresponding to a quadratic dissipation pair.}
\begin{theorem}
\label{thm:main}
Let $\varepsilon\in(0,1)$. Let $B^\varepsilon$ and $\kappa^\varepsilon$ satisfy Assumption \ref{ASS:kappa}.
Let $ f^\varepsilon$ be $\cH$-solutions of the scaling fuzzy Boltzmann equation \eqref{def:scaling} with initial values satisfying  Assumption \ref{ass:curve}.
Then, up to a subsequence, for any $t\in[0,T]$, we have
\begin{equation}
\label{intro-conv-p.w.}
    \sqrt {f^\varepsilon} \to \sqrt{f}\quad\text{in}\quad L^2_\loc(\Do).
\end{equation}
 Moreover, $f$ is an $\cH$-solution to the fuzzy Landau equation \eqref{FLE} with initial value $f_0$.
\end{theorem}

\noindent\textbf{Strategy of the proof:}
We use the variational characterisation in Theorem \ref{thm:FLE:we} to verify that the grazing limit of the fuzzy Boltzmann equation is indeed the fuzzy Landau equation \eqref{FLE}.

Let $(\Psi^*,\Theta)$ by a pair satisfying Assumption \ref{ass-pair}. In addition, for technical reasons (see Remark \ref{rmk:non-qua} below), we take the pair $(\Psi^*,\Theta)$ satisfying also Assumption \ref{ass:Psi}. By Lemma \ref{lem:non-quadratic},  the following entropy inequality for the fuzzy Boltzmann equations holds 
\begin{equation}
\label{start:Boltzmann}
\cH(f^\varepsilon_T)-\cH(f^\varepsilon_0) +\int_0^T\cD_{\Psi^*}^\varepsilon(f^\varepsilon_t)\dd t+\int_0^T\aR^\varepsilon(f^\varepsilon_t,U^\varepsilon_{\aB})\dd t \le 0\quad \forall \varepsilon\in(0,1).
\end{equation}

In Section \ref{sec:conv}, we show the convergence \eqref{intro-conv-p.w.} that up to a subsequence,  $\sqrt{f^\varepsilon}\to \sqrt{f}$ in $L^2_\loc(\Do)$. Our goal is to search for appropriate $U_{\aL}$ such that $(f,U_{\aL})\in \TGRE_T$ and the variational inequality \eqref{V-Landau} holds as (in)equality
\begin{equation}
\label{goal:landau}
\cH(f_T)-\cH(f_0) +\frac12\int_0^T\cD_{\aL}(f_t)\dd t+\frac12\int_0^T\cA_{\aL}(f_t,\cU_{\aL})\dd t \le 0.
\end{equation}
Then Theorem \ref{thm:FLE:we} ensures that $f$ is an $\cH$-solution to the fuzzy Landau equation \eqref{FLE}.

The variational inequality \eqref{goal:landau} holds, since
\begin{itemize}
    \item the assumption \eqref{H-0} and the lower semi-continuity of the Boltzmann entropy ensure that 
    \begin{equation*}
\lim_{\varepsilon\to0}\cH(f^\varepsilon_0)=\cH(f_0)\quad\text{and}\quad \cH(f)\le \liminf_{\varepsilon\to0}  \cH(f^\varepsilon).  
     \end{equation*}

     \item in Section \ref{sec:dissipation}, we show that  
      \begin{equation}
      \label{sec5}
\frac12\cD_{\aL}(f)\le \liminf_{\varepsilon\to0}  \cD^\varepsilon_{\Psi^*}(f^\varepsilon).
    \end{equation} 
    \item in Section \ref{sec:action}, we show the existence of $U_{\aL}$ such that $(f,U_{\aL})\in \TGRE_T$
    and
    \begin{equation}
    \label{sec6}
\frac12\cA_{\aL}(U)\le \liminf_{\varepsilon\to0}  \aR^\varepsilon(f^\varepsilon,U^\varepsilon_{\aB}).
    \end{equation} 
\end{itemize}
In the following remarks, we provide further discussion on the main result.
{\begin{remark}
[Compare with homogeneous case]\label{rmk:compare}
Compared with the homogeneous case, due to the independent spatial variables $x$ and $x_*$, we cannot treat $f-f_*$ as vanishing like $|v-v_*|$ to handle the kinetic singularity of the collision kernels in the (very) soft potential regime. See also Remark \ref{rmk:diff} below. Technically, we always treat $f$ and $f_*$ separately.  See also Remark \ref{rmk:gradient} below.
\end{remark}}
\begin{remark}
\label{rmk:non-qua}
\begin{enumerate}[(i)]
    \item \textit{(A non-quadratic form converges to a quadratic form).}
The entropy dissipation $\cD_{\aL}$ and curve action $\cA_{\aL}$ in the Landau variational inequality \eqref{goal:landau} are given by 
\begin{align*}
\cD_{\aL}(f)= \frac12\int_{\G}\kappa ff_*|\widetilde \nabla\log f|^2\dd\eta\quad\text{and}\quad  \cA_{\aL}(f,U)=\frac12\int_{\G}\frac{|U|^2}{\kappa ff_*}\dd\eta,
\end{align*}
which can be seen as induced by a quadratic pair \eqref{p-1} with
\begin{align*}
\Psi(r)=\Psi^*(r)=\frac{r^2}{2}.    
\end{align*} 
In comparison, the entropy dissipation $\cD_{\Psi^*}$ and curve action $\aR$ in the Boltzmann variational inequality \eqref {start:Boltzmann} are given by 
\begin{gather*}
\cD_{\Psi^*}(f)=\frac14\int_{\G\times S^{d-1}}\Psi^*(-\overline\nabla \log f)\Theta(f)B\kappa \dd\eta\dd \sigma,\\
 \aR(f,U)=\frac14\int_{\G\times S^{d-1}}\Psi\Big(\frac{U}{\Theta(f)B\kappa}\Big)\Theta(f)B\kappa \dd\eta\dd \sigma   
\end{gather*}
where $\Psi^*$ satisfying the Assumption \ref{ass-pair} is not necessarily of a quadratic form. In particular, one can take the $\cosh$-pair as \eqref{p-2} that
 \begin{align*}
    &\Psi^*(r)=4(\cosh(r/2)-1)\quad \text{and}\quad\Theta(s,t)=\sqrt{st}.
    \end{align*}

In other words, a fuzzy Boltzmann equation governed by a non-quadratic pair converges to a fuzzy Landau equation characterised by a quadratic pair in a small angle limit. This can be explained intuitively in terms of the entropy dissipation, and the argument for the cure action follows analogously by a duality argument (affine representation as in Remark \ref{rmk:affine}). The compatibility condition \eqref{con-cpt} of $(\Psi^*,\Theta)$ ensures that 
\begin{align*}
\Psi^*(-\overline\nabla \log f)\Theta(f)=\frac{\Psi^*(-\overline\nabla \log f)}{(\Psi^*)'(-\overline\nabla \log f)}\big( f'f'_*- f f_*\big). 
\end{align*}
On the one hand, as $\varepsilon\to0$, we have 
\begin{align*}
    (f^\varepsilon)'(f^\varepsilon)'_*,\,  f^\varepsilon f^\varepsilon_*\to ff_*\quad\text{and}\quad -\overline\nabla \log f^\varepsilon\to0.  
\end{align*}
On the other hand, the Assumption \ref{ass-pair} that $\Psi^*(0)=(\Psi^*)'(0)=0$ and $(\Psi^*)''(0)=1$ imply that
\begin{equation}
\label{Psi-0}
\frac{\Psi^*(r)}{(\Psi^*)'(r)}=\frac{r}{2}+o(r^2).    
\end{equation}
Hence, at least on the formal level, we have 
\begin{equation*}
\begin{aligned}
\Psi^*(-\overline\nabla \log f^\varepsilon)\Theta(f^\varepsilon)&\sim -\frac{\overline\nabla\log f^\varepsilon}{2}\big( (f^\varepsilon)'(f^\varepsilon)'_*- f^\varepsilon f^\varepsilon_*\big)\\
&=-\frac12 |\overline\nabla\log f^\varepsilon|^2\Lambda(f^\varepsilon).
\end{aligned}
\end{equation*}
where $\Lambda(f)$ denotes the logarithm mean as in \eqref{def:log-mean}.
Notice that the right-hand side is exactly given by the quadratic pair \eqref{p-1} that
\begin{align*}
    \Psi^*(r)=\frac{r^2}{2}\quad\text{and}\quad \Theta(s,t)=\Lambda(s,t).
\end{align*}
In contrast, we  refer the reader to \cite{liero2017microscopic,peletier2023cosh} for examples where a quadratic gradient flow system converges to a non-quadratic one. 
\item \textit{(Homogeneous cases).}
\textcite{carrillo2022boltzmann} showed Theorem \ref{thm:main} for homogeneous equations by taking the  quadratic form that $\Psi^*(r)=\frac{r^2}{2}$ in the variational inequality \eqref{start:Boltzmann}. In the homogeneous case, one can also consider a non-quadratic pair (for instance, the $\cosh$-pair) in \eqref{start:Boltzmann}, following the argument in the fuzzy case analogously.

\item \textit{(A technical difficulty).}
Concerning the dissipation inequality \eqref{sec5}, if we want to directly treat all the non-quadratic pairs as if they were a quadratic pair and use the homogeneous strategy described in \eqref{homo:D}, i.e., use \eqref{Psi-0} ($\Psi^*(r)=\frac{r^2}{2}+o(r^2)$),
we need the uniform boundedness of $|\overline\nabla\log f^\varepsilon|$, or at least the uniform vanishing of $\cD_{\Psi^*}(f^\varepsilon)$ on the set $S_n:=\{|\overline\nabla\log f^\varepsilon|\ge n\}$ as $n\to+\infty$. 
The vanishing of $|S_n|$ is ensured by the convergence of $f^\varepsilon$ and the cancellation Lemma \ref{app-lem:cancel}. However, we do not have the equi-integrability of the integrand of $\cD_{\Psi^*}(f^\varepsilon)$. 

For the homogeneous case, \textcite{carrillo2022boltzmann} proved \eqref{sec5} for a quadratic pair by using the elementary inequality $(\log s-\log t)(s-t)\ge 4|\sqrt s-\sqrt t|^2$ for all $s,t>0$, and by showing that
\begin{equation}
\label{homo:D}
\cD_{\aL}(f)\le \liminf_{\varepsilon\to0}  \int\big|\sqrt{(f^\varepsilon)'(f^\varepsilon)'_*}-\sqrt{f^\varepsilon f^\varepsilon_*}\big|^2B^\varepsilon.
\end{equation} 
Notice that the right-hand side of \eqref{homo:D} is exactly the $\cosh$-dissipation \eqref{D:cosh}
 \begin{equation*}
\cD_{\cosh}(f)=\frac12\int_{\G\times S^{d-1}}|\sqrt{f'f_*'}-\sqrt{ff_*}|^2B\kappa\dd\eta\dd \sigma.
\end{equation*}
In Section \ref{sec:dissipation}, we follow \cite{carrillo2022boltzmann} to prove that
\begin{equation*}
\frac12\cD_{\aL}(f)\le \liminf_{\varepsilon\to0}  \cD^\varepsilon_{\cosh}(f^\varepsilon).
    \end{equation*} 
Notice that \eqref{sec5} automatically holds for all $\Psi^*$ such that
\begin{equation}
\label{want:cosh}
 \cD_{\Psi^*}(f) \ge \cD_{\cosh}(f), 
\end{equation}
for which we make Assumption \ref{ass:Psi} below.

The same type of technical difficulty does not appear when showing the curve action inequality \eqref{sec6}. Roughly speaking, we use a duality argument (affine representation as in Remark \ref{rmk:affine}), and we only need to control $\Psi^*(\overline\nabla\phi)$, where $\phi \in C^\infty_c(\Do)$ is a test function. The boundedness of $|\overline\nabla\phi|$ allows us to treat $\Psi^*(\overline\nabla\phi)\sim \frac{|\overline\nabla\phi|^2}{2}$.
More details can be found in Section \ref{sec:action}.

\end{enumerate}
\end{remark}


For technical reasons stated in Remark \ref{rmk:non-qua}, we consider $(\Psi^*,\Theta)$ satisfying the following assumption such that \eqref{want:cosh} holds, which will be used in Section \ref{sec:dissipation}.
\begin{assumption}
\label{ass:Psi}
 Let $\Psi^*$ satisfy the Assumption \ref{ass-pair} and for all $r>0$
\begin{equation}
\label{coth}
(\log \Psi^*)'(r)\le \frac{\coth(r/4)}{2}.
\end{equation}
\end{assumption}
\begin{remark}
\label{lem:psi:low}
    \begin{itemize}
        \item The Assumption \ref{ass:Psi} holds for the quadratic and $\cosh$-paris in Example \ref{example}. In the quadratic case, we have, for all $r>0$,
\begin{align*}
    &\frac12\coth(r/4)-(\log \Psi^*)'(r)
    =\frac{2}{r}\Big(\frac{e^{r/2}+1}{2}\frac{{r/2}-0}{e^{r/2}-1}-1\Big)\ge0,
    \end{align*}
    where we use the inequality $\Lambda(s,t)\le \frac{s+t}{2}$ with $s=e^{r/2}$ and $t=1$.
In the $\cosh$-case, \eqref{coth} holds with equality.

        \item The assumption \eqref{coth} implies the dissipation inequality \eqref{want:cosh}. More precisely, we have, for all $s>t>0$,
\begin{equation}
\label{pair-lbd}
 2|\sqrt s -\sqrt t|^2\le \Psi^*(\log s -\log t)\Theta(s,t).  
\end{equation}
Indeed, we take $r=\log s-\log  t$ in \eqref{coth} to derive
\begin{equation*}
\label{goal}
\Psi^*(\log s-\log t)\ge 2(\Psi^*)'(\log s-\log t)\frac{\sqrt s-\sqrt t}{\sqrt s+\sqrt t}.
\end{equation*}
Then the inequality \eqref{pair-lbd} holds, since, by compatibility condition \eqref{con-cpt}, we have 
\begin{align*}
    &|\sqrt s-\sqrt t|^2=(s-t)\frac{\sqrt s-\sqrt t}{\sqrt s+\sqrt t}
    =(\Psi^*)'(\log s-\log t)\Theta(s,t)\frac{\sqrt s-\sqrt t}{\sqrt s+\sqrt t}.
\end{align*}
    \end{itemize}
\end{remark}

\section{Preliminaries}\label{sec:pre}
The spherical representation of the sphere $S^{d-1}$
 is presented in Section \ref{sec:SC}, while the pointwise limits and bounds in the grazing limit are discussed in Section \ref{sec:form-GL}.

\medskip

Throughout this work, $C>0$ denotes a universal constant, which may vary from line to line. 

We state the dominated convergence theorem here, which will be frequently used.
\begin{theorem}[Dominated convergence theorem]
\label{DCT}
If a pointwise convergence sequence $f_n(x)\to f(x)$, for a.e. $x\in X$ is dominated by an integrable sequence $\{g_n\}$ in the sense that $|f_n(x)|\le |g_n(x)|$, for all $x\in X$ and $n\ge0$. If $g_n\to g $ and $\lim_{n\to\infty}\int_X g_n=\int_X g$, then $\lim_{n\to\infty}\int_X f_n=\int_X f$.
\end{theorem}

\subsection{Spherical coordinate}\label{sec:SC}
For any fixed $v,v_*\in\R^d$, we  define the unit vector 
\begin{align*}
     k=\frac{v-v_*}{|v-v_*|}\in S^{d-1}.
\end{align*}
The deviation angle is given by $\theta=\arccos k\cdot\sigma$.
For $d \ge 2$ and a given $k\in S^{d-1}$, we define the sphere 
\begin{align*}
     S^{d-2}_{k^\perp}\defeq\{p\in S^{d-1}\mid p\cdot k=0\}.
\end{align*}
For $d=2$ and $k=(k_1,k_2)\in S^1$, the sphere $S^0_{k^\perp}$ is defined as
\begin{align*}
     S^0_{k^\perp}=\big\{p\mid p=(k_2,-k_1),(-k_2,k_1)\big\},
 \end{align*}
 and the integration is given by 
\begin{align*}
\int_{S^0_{k^\perp}}f\dd p=f(k_2,-k_1)+f(-k_2,k_1).
\end{align*}
Notice that $\sigma\in S^{d-1}$ can be  parametrised by $(\theta,p)\in [0,\pi/2] \times S^{d-2}_{k^\perp}$ that 
\begin{align*}
     \sigma =k \cos\theta\ + p\sin\theta\quad\text{and} \quad p=\Pi_{k^\perp} \sigma.
\end{align*}
Concerning the Boltzmann collision kernel $B=A_0(|v-v_*|)b(\sigma)$, $b(\sigma)$ can be written in spherical coordinate as
\begin{align*}
     \int_{S^{d-1}}b(\sigma)\dd\sigma = \int_0^{\pi/2}\int_{S^{d-2}_{k^\perp}}b(\theta)\sin^{d-2}\theta \dd p \dd\theta.
\end{align*}
In this article, we do not distinguish $b(\sigma)$ and $b(\theta)$.

We refer the reader to the presentation and figures in \cite[Section 3.2]{carrillo2022boltzmann}, where the spherical coordinates in three dimensions are discussed in detail.

By changing of variables, the weak Boltzmann collision term can be written in spherical coordinates as
\begin{equation}
    \label{wqb-polar}
\begin{aligned}
    &\int_{\Do} Q^\varepsilon_{\aB}(f,f)\phi\dd x\dd v\\
    =&{}\frac12\int_{\GL}\kappa ff_*\Big(\int_{S^{d-1}}\overline\nabla\phi B^\varepsilon \dd\sigma\Big)\dd\eta\\    =&{}\frac12\int_{\GL}\kappa A_0  ff_*\int_{0}^{\varepsilon/2} \beta^\varepsilon (\theta)\Big(\int_{S_{k^\perp}^{d-2}}\overline\nabla\phi \dd p\Big) \dd \theta \dd \eta
\end{aligned}
\end{equation}
for all $\phi\in C^\infty_c(\Do)$.\\ \\
In the following, we present two propositions of the spherical representation, which are used to establish the bounds and limits in Section \ref{sec:form-GL}.
\begin{proposition}
\label{lem:pi}
Let $d\ge 2$. The following identity holds
    \begin{align*}
        \int_{S^{d-2}_{k^\perp}}p\otimes p\dd p=\frac{|S^{d-2}|}{d-1}\Pi_{k^\perp},\quad \Pi_{k^\perp}=\operatorname{Id}-k\otimes k.
    \end{align*}
\end{proposition}
\begin{proof}
For any $v\in\R^d$, $v$ can be written as the sum of projections $v=\Pi_{k^\perp}v+(v\cdot k)k$.
  By definition, we have 
    \begin{align*}
   & \Big(\int_{S^{d-2}_{k\perp}}p\otimes p\dd p\Big)v
   = \int_{S^{d-2}_{k\perp}}p(p\cdot v)\dd p=\int_{S^{d-2}_{k\perp}}p(p\cdot \Pi_{k^\perp}v)\dd p.
    \end{align*}
    Let $\{e_i\}_{i=1}^{d-1}$ denotes the orthogonal basis of $\R^{d-1}_{k^\perp}$. We write $\Pi_{k^\perp} v=\sum_{i=1}^{d-1}v_ie_i$ and $p=\sum_{i=1}^{d-1}p_ie_i$. Thus, we have 
    \begin{align*}
  &\int_{S^{d-2}_{k\perp}}(p\cdot \Pi_{k^\perp}v)p\dd p
 =\sum_{i=1}^{d-1}v_i\int_{S^{d-2}_{k\perp}}(p\cdot e_i)p\dd p.
 \end{align*}
For all $i=1,\dots, d-1$, we have 
\begin{align*}
 \int_{S^{d-2}_{k\perp}}(p\cdot e_i)p\dd p=\int_{S^{d-2}_{k\perp}}p_ip\dd p=e_i\int_{S^{d-2}_{k\perp}}p_i^2\dd p,
 \end{align*}
 where, by symmetry, we have 
 \begin{align*}
 \int_{S^{d-2}_{k\perp}}p_i^2\dd p=\frac{1}{d-1}\int_{S^{d-2}_{k^\perp}}1\dd p=\frac{|S^{d-2}|}{d-1},\quad i=1,\dots, d-1.
 \end{align*}
\end{proof}

\begin{proposition}[Size estimates]
\label{rmk:size-est}
The vector $\sigma-k$ can be written as 
\begin{equation*}\label{sigma-k}
    \sigma-k=k(\cos\theta-1)+p\sin\theta,
\end{equation*}
where, for $\theta\sim0$, we have 
\begin{align*}
    \cos\theta-1=-\frac{\theta^2}{2}+o(\theta^2)\quad \text{and}\quad \sin\theta=\theta+o(\theta).
\end{align*}
For all $\theta\in[0,\pi/2]$, we have
\begin{equation}\label{sigma-k:sqbd}
\begin{aligned}
|\sigma-k|\le 2\theta\quad \text{and}\quad   |\sigma -k|^2 = 2(1-\cos\theta) \le \theta^2.   
\end{aligned}
\end{equation}
\end{proposition}

\subsection{Limits of the fuzzy Boltzmann gradient}\label{sec:form-GL}
In this subsection, we first follow the classical results in \cite{Vil98} to discuss the pointwise grazing limit in the weak formulation, i.e. for fixed $f$ and $\kappa$, one has
\begin{equation}
\label{limit:formal}
 \int_{\Do} Q^\varepsilon_{\aB}(f,f)\phi\dd x\dd v\to \int_{\Do} Q_{\aL}(f,f)\phi\dd x\dd v\quad \forall\phi\in C^\infty(\Do;\R)
\end{equation}
as $\varepsilon\to0$. We show upper bounds and pointwise limit of fuzzy Boltzmann gradient $\overline\nabla$ in Lemma \ref{lem:phi:conv}, which are used to show the grazing limit rigorously in Section \ref{sec:dissipation} and \ref{sec:action}. Similar estimates have been derived for the homogeneous Landau equation in \cite[Section 3]{carrillo2022boltzmann}. There is a significant difference between the homogeneous and the fuzzy cases, which is discussed in Remark \ref{rmk:diff} below.  


We first show \eqref{limit:formal} on a form level.
By \eqref{wqb-polar}, the left-hand side of \eqref{limit:formal} can be written as
\begin{align*}
    &\int_{\Do} Q^\varepsilon_{\aB}(f,f)\phi\dd x\dd v
    ={}\frac12\int_{\GL}\kappa A_0  ff_*\int_{0}^{\varepsilon/2} \beta^\varepsilon (\theta)\Big(\int_{S_{k^\perp}^{d-2}}\overline\nabla\phi \dd p\Big) \dd \theta \dd \eta.
\end{align*}
By integration by parts of $\widetilde\nabla\cdot$, the right-hand side of \eqref{limit:formal} can be written as
\begin{align*}
&\int_{\Do} Q_{\aL}(f,f)\phi\dd x\dd v=\frac12\int_{\GL}\kappa ff_*\widetilde\nabla\cdot \widetilde\nabla \phi\dd\eta\\
=&{}\frac12\int_{\GL}\kappa A_0ff_*(\nabla_v-\nabla_{v_*})\cdot\big(|v-v_*|^2\Pi_{(v-v_*)^\perp}(\nabla_v\phi-\nabla_{v_*}\phi_*)\big)\dd\eta.    
\end{align*}
Thus, to show \eqref{limit:formal}, we only need to show that
\begin{equation}
    \label{limit:goal}
\begin{aligned}
 &\int_{0}^{\varepsilon/2} \beta^\varepsilon (\theta)\Big(\int_{S_{k^\perp}^{d-2}}\overline\nabla\phi \dd p\Big) \dd \theta\\
 \limit{\eps \to 0}&{}(\nabla_v-\nabla_{v_*})\cdot\big(|v-v_*|^2\Pi_{(v-v_*)^\perp}(\nabla_v\phi-\nabla_{v_*}\phi_*)\big).
\end{aligned}
\end{equation}
In Lemma \ref{lem:phi:conv} below, we show the pointwise limit
\begin{equation*}\label{formal:goal}
\begin{aligned}
&\lim_{\theta\to0}\frac{1}{\theta^2}\int_{S^{d-2}_{k^\perp}}\overline\nabla\phi\dd p\\ 
=&{}\frac{|S^{d-2}|}{8(d-1)}(\nabla_v-\nabla_{v_*})\cdot|v-v_*|^2\Pi_{(v-v_*)^\perp}\big(\nabla_v\phi-\nabla_{v_*}\phi_*\big).
\end{aligned}
\end{equation*}
Combining that with the finite angular momentum assumption \eqref{def:beta-int}, we obtain the desired convergence \eqref{limit:goal}.
\medskip

In the following, we adapt the homogeneous arguments in \cite[Section 3]{carrillo2022boltzmann} to establish the estimates and limits of $\overline\nabla$ as $\theta\to0$. In Section \ref{sec:dissipation} and \ref{sec:action}, we apply the dominate convergence Theorem \ref{DCT} to pass to the limit by letting $\theta\to0$. The upper bounds in Lemma \ref{lem:phi:conv} take the form of a product of $|v-v_*|$ and $\theta$, where the velocity difference is employed to control the singular kinetic kernel $A_0=|v-v_*|^{\gamma}$ in the soft potential cases $\gamma\in[-2,0)$, and the deviation angle $\theta$ is used to establish the finite angular momentum form \eqref{def:beta-int} $\int_0^{\frac{\varepsilon}{2}}\theta^2\beta^\varepsilon(\theta)\dd\theta<+\infty$. In Remark \ref{rmk:diff}, we highlight a key distinction in the bounds \eqref{bd-1} and \eqref{bd-2} between homogeneous and fuzzy cases, which directly implies that, under this approach, the very soft potential case ($A_0=|v-v_*|^\gamma$, $\gamma\in(-4,-2)$)  cannot be treated in the fuzzy setting.

\begin{lemma}
\label{lem:phi:conv}
Let $\phi\in C^\infty_c(\Do)$. For all $(x,x_*,v,v_*,\sigma)\in\G\times S^{d-1}$:
\begin{itemize}
\item the following pointwise bounds hold
\begin{gather}
    |\overline\nabla\phi|\le C_1\theta |v-v_*|,\label{bd-1}\\
  \left|\frac{1}{|S^{d-2}|}\int_{S^{d-2}_{k^\perp}}\overline\nabla\phi\dd p\right|\le C_2 \theta^2\big(|v-v_*|+|v-v_*|^2\big), \label{bd-2}
\end{gather}
where $C_1=C_1(\|\nabla_v \phi\|_{L^\infty(\Do)})> 0$ and $C_2=C_2(\|\nabla_v \phi\|_{L^\infty(\Do)},\|D^2_v \phi\|_{L^\infty(\Do)})>0$.

    \item the following pointwise limits hold
\begin{equation}
\label{p.w.conv-1}
\lim_{\eps\to 0}\frac1\theta\overline\nabla\phi=\frac{|v-v_*|}{2}p\cdot\big(\nabla_v\phi-(\nabla_v\phi)_*\big), 
\end{equation}
and
\begin{equation}
\label{p.w.conv-2}
\begin{aligned}
 &\lim_{\eps\to 0}\frac{1}{\theta^2}\int_{S^{d-2}_{k^\perp}}\overline\nabla\phi\dd p\\
=&{}\frac{|S^{d-2}|}{8(d-1)}(\nabla_v-\nabla_{v_*})\cdot\Big(|v-v_*|^2\Pi_{(v-v_*)^\perp}\big(\nabla_v\phi-(\nabla_v\phi)_*\big)\Big).
\end{aligned}
\end{equation}
\end{itemize}
\end{lemma}
\begin{proof}
The proof follows the homogeneous case in \cite[Lemma 3.4--3.6]{carrillo2022boltzmann}. The key difference from the homogeneous case is that here the quantities associated with $f$ and $f_*$ are strictly treated as separate quantities, since they depend on the independent spatial variables $x$ and $x_*$.

We recall the definition
\begin{align*}
    v_1 = \frac{v+v_*}{2} \quad \text{and} \quad v_2=\frac{v-v_*}{2}, 
\end{align*}
so that 
\begin{equation*}
\label{v-v'}
    v=v_1+|v_2|k,\quad v_*=v_1-|v_2|k,\quad  v'=v_1+|v_2|\sigma,\quad v_*'=v_1-|v_2|\sigma.
\end{equation*}
By Taylor expansion, we have
\begin{equation}
\label{taylor-1}
\begin{aligned}
&\overline\nabla\phi(x,v)\\
=&{}\phi(x,v_1+|v_2|\sigma)+\phi(x_*,v_1-|v_2|\sigma)
    -\phi(x,v_1+|v_2|k)-\phi(x_*,v_1-|v_2|k)\\
    =&{}\int_0^1\frac{d}{dt}\phi\big(x,v_1+|v_2|(t\sigma+(1-t)k)\big)+\frac{d}{dt}\phi\big(x_*,v_1-|v_2|(t\sigma+(1-t)k)\big)\dd t\\
    =&{}|v_2|(\sigma-k)\cdot\int_0^1\Big(\nabla_v\phi\big(x,v_1+|v_2|(t\sigma+(1-t)k)\big)\\
   & -\nabla_{v_*}\phi\big(x_*,v_1-|v_2|(t\sigma+(1-t)k)\big)\Big)\dd t,
   \end{aligned}
\end{equation}
and 
\begin{align*}
   &\nabla_{v}\phi\big(x,v_1+|v_2|(t\sigma+(1-t)k)\big)\\
   &{}=\nabla_{v}\phi(x,v_1+|v_2|k)\\
   &\quad{}+t|v_2|(\sigma-k)\cdot \int_0^1D^2_{v}\phi\Big(x,v_1+|v_2|\big(s(t\sigma+(1-t)k)+(1-s)k\big)\Big)\dd s,\\
   &\nabla_{v_*}\phi\big(x_*,v_1-|v_2|(t\sigma+(1-t)k)\big)\\
   &{}=\nabla_{v_*}\phi(x_*,v_1-|v_2|k)\\
   &\quad{}-t|v_2|(\sigma-k)\cdot \int_0^1D^2_{v_*}\phi\Big(x,v_1-|v_2|\big(s(t\sigma+(1-t)k)+(1-s)k\big)\Big)\dd s.
\end{align*}
Thus, $\overline\nabla \phi$ can be written as
\begin{equation}
\label{I-12}
\overline\nabla\phi
=I_1+I_2,
\end{equation}
where
\begin{align*}
I_1&:=|v_2|(\sigma-k)\cdot\big(\nabla_v\phi-(\nabla_v\phi)_*\big),\quad
I_2:=|v_2|^2(\sigma-k)\otimes (\sigma-k):T,
\end{align*}
and $T$ is given by
\begin{align*}
T&=\int_0^1t\int_0^1\Big(D^2_v\phi\big(x,v_1+|v_2|(s(t\sigma+(1-t)k)+(1-s)k)\big)\\
&+D^2_{v_*}\phi\big(x_*,v_1-|v_2|(s(t\sigma+(1-t)k)+(1-s)k)\big)\Big)\dd s\dd t.
\end{align*}

\begin{itemize}

\item \textbf{Bounds \eqref{bd-1} and \eqref{bd-2}.}
We recall the size estimate in Proposition \ref{rmk:size-est} that 
\begin{align*}
    |\sigma-k|\le 2\theta\quad\text{and}\quad |\sigma-k|^2\le \theta^2.
\end{align*}
Then \eqref{taylor-1} implies the pointwise bound \eqref{bd-1} that
\begin{align*}
    |\overline\nabla\phi|&\le2|v_2||\sigma-k|\|\nabla_{v}\phi\|_{L^\infty(\R^{2d})}\lesssim C_1\theta|v-v_*|.
    \end{align*}
We use the representation \eqref{I-12} to show the bound \eqref{bd-2}. Concerning the $I_2$ term, we have 
\begin{equation}
\label{pw-I2}
 \Big|\int_{S^{d-2}_{k^\perp}} I_2\Big|\le |S^{d-2}||I_2|
   \lesssim \|D^2_{v}\phi\|_{L^\infty(\R^{2d})}|v_2|^2|\sigma-k|^2.    
\end{equation}
Concerning the $I_1$ term, by definition $\sigma-k=k(\cos\theta-1)+p\sin\theta$, we have  
\begin{equation*}
\label{N:k-p}
\begin{aligned}
&(\sigma-k)\cdot\big(\nabla_{v}\phi-(\nabla_{v}\phi)_*\big)
={}N_k(\cos\theta-1)+N_p\sin\theta,
\end{aligned}
\end{equation*}
where $N_k$ and $N_p$ denote the projections
\begin{align*}
N_k= \big(\nabla_{v}\phi-(\nabla_{v}\phi)_*\big)\cdot k,\quad
N_p= \big(\nabla_{v}\phi-(\nabla_{v}\phi)_*\big)\cdot p.
\end{align*}
Notice that $N_k$ is independent of $p$ and  
\begin{align*}
  \int_{S^{d-2}_{k^\perp}} N_p\dd p=0.
\end{align*}
This implies that
\begin{equation}
\label{int-I1}
\begin{aligned}
   \int_{S^{d-2}_{k^\perp}} I_1\dd p&= |v_2|\int_{S^{d-2}_{k^\perp}}\big((\cos\theta-1)N_k+\sin\theta N_p\big)\dd p\\
   &=|S^{d-2}||v_2|N_k(\cos\theta-1).
       \end{aligned}
\end{equation}
The identity $|\sigma-k|^2=2(1-\cos\theta)$ implies that
\begin{align*}
   &\Big|\int_{S^{d-2}_{k^\perp}} I_1\Big|
   \lesssim |v_2||\sigma-k|^2\|\nabla_{v}\phi\|_{L^\infty(\R^{2d})}.
\end{align*}
Thus, exploiting \eqref{sigma-k:sqbd}, we obtain the second bound \eqref{bd-2}.

    \item \textbf{Limits \eqref{p.w.conv-1} and \eqref{p.w.conv-2}.} We recall Proposition \ref{rmk:size-est} that 
\begin{equation}
\label{sigma-k-2}
\cos\theta-1\sim -\frac{\theta^2}{2}\quad \text{and}\quad \sin\theta\sim \theta.
\end{equation}
Then the representation \eqref{I-12} and the upper bound of $I_2$ \eqref{pw-I2} imply the limit \eqref{p.w.conv-1} that 
\begin{align*}
\overline\nabla\phi=\frac{\theta|v-v_*|}{2} p\cdot\big(\nabla_v\psi-(\nabla_v\psi)_*\big)+o(\theta^2).
\end{align*}
We show (\ref{p.w.conv-2}). Notice that $\int_{S^{d-2}_{k^\perp}} I_1\dd p$ is given by \eqref{int-I1} and the size estimate \eqref{sigma-k-2} that
\begin{equation}
\label{int-1}
    \int_{S^{d-2}_{k^\perp}}I_1\dd p=-\frac{\theta^2}{2}|v_2||S^{d-2}|N_k+o(\theta^2).
\end{equation}
We left to compute the $I_2$ term. By the dominated convergence theorem and the upper bound of $I_2$ \eqref{pw-I2}, we have 
\begin{equation}
\label{int-2}
\begin{aligned}
\lim_{\theta\to 0}\frac{1}{\theta^2}\int_{S^{d-2}_{k^\perp}}I_2\dd p
    &=\frac{ \theta^2|v_2|^2}{2} \int_{S^{d-2}_{k^\perp}}p\otimes p\dd p:\big(D^2_v\phi+(D^2_v\phi)_*\big)\\
    &=\frac{ \theta^2|v_2|^2|S^{d-2}|}{2(d-1)} \Pi_{k^\perp}:\big(D^2_v\phi+(D^2_v\phi)_*\big),
\end{aligned}
\end{equation}
where we use Lemma \ref{lem:pi} to show the last equality.

Combining \eqref{int-1} and \eqref{int-2}, we use the identities  $|v_2|k=v_2$ and $\Pi_{k^\perp}= \Pi_{(v-v_*)^\perp}$ to derive
\begin{equation*}
\label{int-3}
\begin{aligned}
\lim_{\theta\to0}\frac{1}{\theta^2}\int_{S^{d-2}_{k^\perp}}\overline\nabla \phi\dd p
=&{}\frac{\theta^2}{2}|S^{d-2}|\Big(-\frac{v-v_*}{2}\cdot \big(\nabla_v \phi-(\nabla_v \phi)_*\big)\\
&+\frac{ |v-v_*|^2}{4(d-1)} \Pi_{(v-v_*)^\perp}:(D^2_v\phi+(D^2_v\phi)_*)\Big).
\end{aligned}
\end{equation*}
By using of the identity 
\begin{align*}
(\nabla_{v}-\nabla_{v_*})\cdot(|v-v_*|^2\Pi_{(v-v_*)^\perp})=-2(d-1)(v-v_*),    
\end{align*}
we obtain \eqref{p.w.conv-2}
 \begin{equation*}
    \label{int-4}
\begin{gathered}
\lim_{\theta\to0}\frac{1}{\theta^2}\int_{S^{d-2}_{k^\perp}}\overline{\nabla}\phi\dd p=|v-v_*|^2 \Pi_{(v-v_*)^\perp}:(D^2_v\phi+(D^2_v\phi)_*)\Big)\\
+\frac{\theta^2|S^{d-2}|}{8(d-1)}\Big((\nabla_v-\nabla_{v_*})\cdot(|v-v_*|^2\Pi_{(v-v_*)^\perp})\cdot (\nabla_v \phi-(\nabla_v \phi)_*)\\
=\frac{\theta^2|S^{d-2}|}{8(d-1)}(\nabla_v-\nabla_{v_*})\cdot\Big(|v-v_*|^2\Pi_{(v-v_*)^\perp}(\nabla_v \phi-(\nabla_v \phi)_*)\Big).
\end{gathered}
\end{equation*}

\end{itemize}

\end{proof}

\begin{remark}[Better bounds in the homogeneous cases]\label{rmk:diff}
The similar bounds in Lemma \ref{lem:phi:conv} for spatial homogeneous functions have been shown in \cite{carrillo2022boltzmann}. The main difference is that in the homogeneous case, one has a quadratic bound of the velocity difference
\begin{gather}
|\overline\nabla\phi_1|\lesssim \|\nabla_v \phi_1\|_{L^\infty(\Do)} \theta|v-v_*|\quad\text{fuzzy case}, \notag \\
|\overline\nabla\phi_2|\lesssim  \|D^2_v \phi_2\|_{L^\infty(\R^d)}\theta|v-v_*|^2\quad\text{homogeneous case}, \label{bd-3}    
\end{gather}
where $\phi_1=\phi_1(x,v)$ and $\phi_2=\phi_2(v)$.
We recall \eqref{taylor-1} that
\begin{gather*}
\overline\nabla\phi(x,v)=|v_2|(\sigma-k)\cdot\int_0^1\nabla_v\phi(x,\bar v)-\nabla_{v_*}\phi(x_*,\bar v_*)\dd t,
   \end{gather*}
   where $\bar v=v_1+|v_2|(t\sigma+(1-t)k)$ and $\bar v_*=v_1-|v_2|(t\sigma+(1-t)k)$.
In the homogeneous case $\phi_2=\phi_2(v)$, we have
\begin{equation}
\label{homo:v-v*}
\big|\nabla_v\phi_2(\bar v)-\nabla_{v_*}\phi_2(\bar v_*)\big|  \le \|D^2_v\phi_2\|_{L^\infty(\R^d)}|v-v_*|,
\end{equation}
which leads to the quadratic bounds $|v-v_*|^2$ in \eqref{bd-3}. 
However, \eqref{homo:v-v*} does not hold for the fuzzy case due to the appearance of the different spatial  variables $x$ and $x_*$
\begin{equation}
\label{fuzzy:landau:better}
 \big|\nabla_v\phi_1(x,\bar v)-\nabla_{v_*}\phi_1(x_*,\bar v_*)\big| \le \|\phi_1\|_{W^{2,\infty}(\Do)}\big(|v-v_*|  +|x-x_*|\big).
\end{equation}

Intuitively speaking, following the idea of 
$\cH$-solutions, one has 
\begin{align*}
    \int Q_{\aB}(f,f)\phi\le \cD_{\aB}(f)^{\frac12}\Big(\int \int_{S^{d-1}}|\overline\nabla\phi|^2B\Big)^{\frac12}.
\end{align*}
The vanishing of $|v-v_*|$ around $v\sim v_*$ cancels the singularity of the kinetic kernel $|v-v_*|^\gamma$ in the soft potential cases. The improved bound \eqref{bd-3} in the homogeneous setting allows one to handle the very soft potential case $\gamma\in(-4,-2)$
\begin{equation*}
\label{homo:better}
  |v-v_*|^\gamma|\overline\nabla \phi_2|^2\lesssim |v-v_*|^{\gamma+4}.
\end{equation*}
In contrast, in the fuzzy setting, we restrict our analysis to the hard and moderately soft potential cases $\gamma\in[-2,1]$, where
\begin{equation}
\label{fuzzy:better}
|v-v_*|^\gamma|\overline\nabla \phi_1|^2\lesssim |v-v_*|^{\gamma+2}.
\end{equation}
The bound \eqref{fuzzy:better} is, for example, used in the proof of Lemma \ref{lem:6-6}.

We note that another difference has been observed between the homogeneous and fuzzy Landau equations. The solvability of the homogeneous Landau equations in the very soft potential cases was established in \cite{Vil98} by using \eqref{homo:v-v*} to cancel the kinetic singularity
\begin{equation}
\label{homo:landau:better}
 \big|\widetilde\nabla \phi_2(v)\big|^2\lesssim |v-v_*|^{\gamma+4},\quad \gamma\in(-4,-2).  
\end{equation}
However, \eqref{homo:landau:better} does not hold in the fuzzy case, since $|\widetilde\nabla\phi|$ may depend on $|x-x_*|$ as shown in \eqref{fuzzy:landau:better}. More details can be found in  \cite[Remark 3.4]{DH25b}.

\end{remark}

\section{Compactness results}\label{sec:conv}
In this section, we show the $L^1$-weak precompactness of $\{f^\varepsilon\}$ and the $L^2$-precompactness of $\{\sqrt{f^\varepsilon}\}$ in Lemma \ref{weak:conv} and Lemma \ref{lemma:cpt:sqrt}, respectively.
\begin{lemma}
\label{weak:conv}
Let $\varepsilon\in(0,1)$. Let $B^\varepsilon$ and $\kappa^\varepsilon$ satisfy Assumption \ref{ASS:kappa}. Let $f^\varepsilon$ be $\cH$-solutions of \eqref{def:scaling} with initial values satisfying Assumption \ref{ass:curve}. For any fixed $t\in[0,T]$, we have, up to a sequence, 
\begin{equation*}
    \label{L1-cpt}
 f^\varepsilon\rightharpoonup f \quad\text{in}\quad L^1(\Do)\quad\text{as }\varepsilon\to0.
    \end{equation*}
    Moreover, $f$ is weakly continuous in time.
\end{lemma}
\begin{proof}
   By \eqref{uni-bdd:D}, the uniform bounds of $\|f^\varepsilon\|_{L^1_{2,2}(\Do)}$ and $\|f^\varepsilon\log f^\varepsilon\|_{L^1(\Do)}$,  and the Dunford--Pettis theorem imply the $L^1$-weak compactness of $\{f^\varepsilon\}$ in  $L^1(\Do)$.

    Let $0\le s\le t\le T$. Let $\phi\in C^\infty_c(\Do)$. By the weak formulation \eqref{weak:FBE}, we have
    \begin{align*}
        &\Big|\int_{\Do}(f^\varepsilon_t-f^\varepsilon_s)\phi\dd x\dd v\Big|\\
        \le&{} \int_s^t\int_{\Do}|f^\varepsilon v\cdot\nabla_x\phi|\dd x\dd v\dd \tau+\frac14\Big|\int_s^t\int_{\G\times S^{d-1}}\overline\nabla \phi \overline\nabla \log f^\varepsilon \Lambda(f^\varepsilon)B^\varepsilon \kappa^\varepsilon\dd \sigma\dd \eta\dd \tau\Big|.
    \end{align*}
   The uniformly bounded energy ensures that
    \begin{align*}
        \int_s^t\int_{\Do}|f^\varepsilon v\cdot\nabla_x\phi|\dd x\dd v\dd \tau\le |s-t|\|f^\varepsilon \|_{L^1_{0,2}(\Do)}\le C|s-t|
    \end{align*}
    for some constant $C>0$. By Cauchy--Schwarz inequality, we have
    \begin{align*}
        &\Big|\int_s^t\int_{\G\times S^{d-1}}\overline\nabla \phi \overline\nabla \log f^\varepsilon \Lambda(f^\varepsilon)B^\varepsilon \kappa^\varepsilon\dd \sigma\dd \eta\dd \tau\Big|\\
         \le &{}\Big(\int_0^T\cD_{\aB}^\varepsilon(f^\varepsilon)\dd t\Big)^{\frac12}\Big(\int_s^t\int_{\G\times S^{d-1}}f^\varepsilon f^\varepsilon_*B^\varepsilon \kappa^\varepsilon|\overline\nabla \phi|^2\dd \sigma\dd \eta\dd \tau\Big)^{\frac12}.
    \end{align*}
    The uniform bound of the entropy dissipation term is guaranteed by \eqref{uni-bdd:D}. 
    By polar coordinates and the bound of $|\overline\nabla\phi|$ in Lemma \ref{lem:phi:conv}, we have
    \begin{align*}
        &\int_s^t\int_{\G\times S^{d-1}}f^\varepsilon f^\varepsilon_*B^\varepsilon \kappa^\varepsilon|\overline\nabla \phi|^2\dd \sigma\dd \eta\dd \tau\\
        \le  &{}C \int_s^t\int_{\G}f^\varepsilon f^\varepsilon_*A_0|v-v_*|^2 \kappa^\varepsilon\Big(\int_0^{\frac{\varepsilon}{2}}\theta^2 \beta^\varepsilon(\theta)\dd\theta\Big)\dd \eta\dd \tau\\
        \le&{} C\|f^\varepsilon\|_{L^1_{0,2+\gamma}(\Do)}^2|s-t|\le C|s-t|.
    \end{align*}
    Thus, we have the weak continuity of $f$ in time, since
    \begin{align*}
      \Big|\int_{\Do}(f_t-f_s)\phi\dd x\dd v\Big|&\le\liminf_{\varepsilon\to0}  \Big|\int_{\Do}(f^\varepsilon_t-f^\varepsilon_s)\phi\dd x\dd v\Big|\\
      &\le C\big(|s-t|+|s-t|^{\frac12}\big).
    \end{align*}
\end{proof}
The rest of this section is devoted to showing the following strong convergence of $\sqrt {f^\varepsilon}$.
\begin{lemma}
\label{lemma:cpt:sqrt}
Let $\varepsilon\in(0,1)$. Let $f^\varepsilon$ be given as in Lemma \ref{weak:conv}. For any fixed $t\in[0,T]$, we have, up to a sequence, 
\begin{equation*}
\label{conv:cpt}
    \sqrt{f^\varepsilon}\to \sqrt{f}\quad\text{in}\quad L^2_\loc(\Do).
\end{equation*}
\end{lemma}
In this section, we fix $R>0$ arbitrarily. We define the ball $B^R_x:=\{x\in\R^d\mid |x|\le R\}$, and $B^R_v$ is defined analogously.
\medskip

\noindent\textbf{Strategy of the proof:}
We follow the compactness argument for spatial homogeneous and classical Boltzmann equations in  \cite{AV02,ADVW00} to show Lemma \ref{lemma:cpt:sqrt} with appropriate modification for the fuzzy case.

Let $\rho\in C^\infty_c(\R^d;\R_+)$ such that $\int_{\R^d}\rho=1$. For any $\delta\in(0,1)$, we define $\rho^\delta= \delta^{-d}\rho(\cdot/\delta)$. We define 
 \begin{equation}
 \label{def:g-delta}
g^\varepsilon_\delta=\sqrt{f^\varepsilon}*_v\rho^\delta=\int_{\R^d}\sqrt{f^\varepsilon}(x,w)\rho^\delta(v-w)\dd w\in L^2\cap L^\infty(\Do).
 \end{equation}
 By definition, for any fixed $\varepsilon$, we have the convergence $g^\varepsilon_\delta\to \sqrt{f^\varepsilon}$ as $\delta\to0$. \\
 
\noindent We prove Lemma \ref{lemma:cpt:sqrt} in the following two steps:
\begin{enumerate}[Step $1$:]
    \item  For any fixed $\delta\in(0,1)$, we use the so-called velocity averaging lemma to show the compactness of $g^\varepsilon_\delta$. We show the following result:
    \begin{lemma}
      \label{lemma-1}
      Let $\varepsilon\in(0,1)$. Let $f^\varepsilon$ be given as in Lemma \ref{weak:conv}. For any fixed $\delta\in(0,1)$, we have 
    \begin{equation*}
    \{g^\varepsilon_\delta\}    \quad\text{is relatively compact in 
 }L^2([0,T]\times B^R_x\times B^R_v).   
    \end{equation*}
    \end{lemma}
    \item The angular singularity \eqref{nu}, $\beta(\theta)\gtrsim \theta^{-1-\nu}$ for some $\nu\in(0,2)$, allows us to treat the collision term $Q(f,f)$ as fractional diffusion operator in $v$. We show that $\sqrt{f^\varepsilon}$ is uniformly bounded in $ L^2_xH^{\frac{\nu}{2}}_v$, which leads to the 
 convergence of $g^\varepsilon_\delta$ uniformly in $\varepsilon$. We show the following result:
    \begin{lemma}
      \label{lemma-2}  
      Let $\varepsilon\in(0,1)$. Let $f^\varepsilon$ be given as in Lemma \ref{weak:conv}. For any fixed $t\in[0,T]$, we have  
    \begin{equation*}
    \label{conv-2}
g^\varepsilon_\delta\to \sqrt{f^\varepsilon}   \quad\text{as}\quad\delta\to0\quad\text{in}\quad L^2(B^R_x\times B^R_v)
\end{equation*}
uniformly in $\varepsilon\in(0,1)$.
    \end{lemma}
    \end{enumerate}

Lemma \ref{lemma-1} and Lemma \ref{lemma-2} imply the compactness of $\sqrt{f^\varepsilon}$ in $L^2_\loc(\Do)$. Combing with Lemma \ref{weak:conv}, the uniqueness of weak and strong limits implies that 
\begin{equation*}
    \sqrt{f^\varepsilon}\to \sqrt{f}\quad\text{in}\quad L^2(B_x^R\times B^R_v).
\end{equation*}

In the rest of this section,  we show Lemma \ref{lemma-1} and Lemma \ref{lemma-2} in Section \ref{sec:averge} and  Section \ref{sec:smooth}, respectively.


\subsection{Proof of Lemma \ref{lemma-1}}\label{sec:averge}

We define $\alpha(t)=\sqrt t$ for all $t\ge0$. Let $(f^\varepsilon_t)\subset \cP(\Do)$ be solution curves to the fuzzy Boltzmann equation \eqref{def:scaling}. Then the following renormalised equation holds in the distribution sense
\begin{equation}
    \label{FBE-alpha}
    \d_t \alpha(f)+v\cdot \nabla_x\alpha(f)=\alpha'(f)Q^\varepsilon_{\aB}(f,f).
\end{equation}
We recall the definition \eqref{def:g-delta} that
\begin{align*}
g^\varepsilon_\delta=\alpha(f^\varepsilon)*_v\rho^\delta.   
\end{align*}
We apply the following velocity averaging Lemma \ref{thm:ave} to the renormalised equation \eqref{FBE-alpha} to show that, 
for any fixed $\delta\in(0,1)$,
\begin{equation}
    \label{conv-1}
    \{g^\varepsilon_\delta\}    \quad\text{is relatively compact in 
 }L^2([0,T]\times B^R_x\times B^R_v).  
    \end{equation}
       \begin{theorem}[Velocity averaging lemma, \cite{DLM91,DL88}]
\label{thm:ave}
Let $n\in\mathbb{N}_+$. Let $\{h^n\}$ be bounded in $L^p([0,T]\times\Do)$ for some $p\in(1,+\infty)$. 
Let $h^n$ be a weak solution of
\begin{equation*}
    \d_t h^n+v\cdot \nabla_x h^n=H^n+\sum_{i=1}^d \d_{v_i}H^n_i+\sum_{i=1,j}^d \d_{v_iv_j}H^n_{ij},
\end{equation*}
where $\{H^n\}$, $\{H^n_i\}$ and $\{H^n_{ij}\}$ are bounded $L^1([0,T]\times B^R_x\times B^R_v)$. 
Then for all $\varphi\in C^\infty_c(\R^d)$, we have 
\begin{align*}
\{h^n*_v\varphi\}\quad \text{is relatively compact in $L^1([0,T]\times B^R_x\times B^R_v)$}. 
\end{align*}
\end{theorem}
To apply the above velocity averaging Lemma to the renormalised equation, the key point is to justify the integrability of the $Q^\varepsilon_{\aB}(f,f)$. We split the collision operator as
\begin{equation}
\label{split}
  Q_{\aB}(f,f)=\int_{\G\times S^{d-1}} B\kappa f (f_*'-f_*)\dd \sigma\dd\eta+\int_{\G\times S^{d-1}} B\kappa f_* (f'-f)\dd \sigma\dd\eta.
\end{equation}

We first summarise the results in \cite{AV02}: The definition of the cross-section in Definition \ref{def:cs} and the cancellation Lemma \ref{app-lem:cancel}. The cancellation lemma allows us to treat $Q^\varepsilon_{\aB}(f,f)$ as the sum of two singular integrals.
By using the cancellation Lemma \ref{app-lem:cancel}, we show \eqref{conv-1} for the Maxwellian and hard potential cases \eqref{A-hard} in Section \ref{sec:hard}. Concerning the soft potential cases $\gamma\in[-2,0)$, due to the singularity of the kinetic kernel, instead of directly treating the renormalised equation \eqref{FBE-alpha}, we consider a sequence of parametrised renormalised equations. This case will be discussed in Section \ref{sec:rem} in parallel.

\medskip
 
We review the definition of cross-section.
\begin{definition}[\cite{AV02}]\label{def:cs}
Let $B=B(|v-v_*|,\sigma)\ge 0$. 
We define the cross-section for momentum transfer
\begin{equation}
    \label{cross}
\begin{aligned}
\Lambda(|v-v_*|)&=\int_{S^{d-1}}B(|v-v_*|,\theta)(1-k\cdot\sigma)\dd\sigma,\quad k=\frac{v-v_*}{|v-v_*|}.
\end{aligned}
\end{equation}
We define the following quantities concerning the regularities in $v$
\begin{gather*}
B'(|v-v_*|,\sigma)=\sup_{1<\lambda\le\sqrt 2}\frac{\big|B(\lambda |v-v_*|,\sigma)-B(|v-v_*|,\sigma)\big|}{(\lambda-1)|v-v_*|}\\
\text{and}\quad \Lambda'(|v-v_*|)=\int_{S^{d-1}}B'(|v-v_*|,\sigma)(1-k\cdot\sigma)\dd\sigma.
\end{gather*}
\end{definition}

\begin{remark}
\begin{itemize}
    \item 
Notice that $k\cdot\sigma=\cos\theta$ and $1-\cos\theta \lesssim \theta^2$. The finite angular momentum assumption \eqref{def:beta-int}  implies that 
 \begin{align*}
\Lambda(|v-v_*|) 
&=|S^{d-2}|A_0(|v-v_*|)\int_0^{\frac{\pi}{2}}\beta(\theta)(1-\cos\theta)\dd\theta\lesssim A_0(|v-v_*|).
\end{align*}

\item For all $(v,v_*)\in\Do$, we have 
\begin{equation}
    \label{cpt:S}
\begin{aligned}
\int_{S^{d-1}}B(|v-v_*|,\sigma)(v-v')\dd\sigma=\frac{v-v_*}{2}\Lambda(|v-v_*|).
\end{aligned}
\end{equation}
Indeed, by definition, we have 
\begin{align*}
    v-v'=\frac{|v-v_*|}{2}(k-\sigma).
\end{align*}
We repeat the argument for showing \eqref{int-I1}. Notice that $\sigma$ can be written as the sum of projections $\sigma=k(k\cdot\sigma)+p(p\cdot\sigma)$, where, by symmetry, we have $\int_{S^{d-2}_\perp} p(p\cdot\sigma)\dd p=0$. Thus, we have   
\begin{align*}
    \int_{S^{d-1}} (k-\sigma)\dd p=k\int_{S^{d-1}} (1-k\cdot \sigma)\dd p.
\end{align*}
Since $k$ is independent of $\sigma$, we have 
\begin{equation*}
\begin{aligned}
&\int_{S^{d-1}}B(|v-v_*|,\sigma)(v-v')\dd\sigma\\ 
=&{}\frac{|v-v_*|}{2}k\int_{S^{d-1}}B(v-v_*,\sigma)(1-k\cdot \sigma)\dd\sigma\\
=&{}\frac{v-v_*}{2}\Lambda(|v-v_*|).
\end{aligned}
\end{equation*}
\item In the case of $A_0(|v-v_*|)=|v-v_*|^\gamma$ with $\gamma\in[-2,1]$, we have 
\begin{equation}
\label{cross-bd}
\Lambda(|v-v_*|)+|v-v_*|\Lambda'(|v-v_*|)\lesssim|v-v_*|^\gamma. 
\end{equation}
In the case of $A_0(|v-v_*|)\sim\langle v-v_*\rangle^\gamma$ with $\gamma\in(-\infty,1]$, we assume 
\begin{equation}
\label{cross-bd-sim}
\Lambda(|v-v_*|)+|v-v_*|\Lambda'(|v-v_*|)\lesssim\langle v-v_*\rangle. 
\end{equation}
The assumption \eqref{cross-bd-sim} and the cancellation Lemma \ref{app-lem:cancel} ensure the integrability of $Q_{\aB}(f,f)$
in this case.
\end{itemize}
\end{remark}

The following cancellation lemma ensures that \eqref{split} is well-defined.
\begin{lemma}[\cite{AV02}, Proposition 3.1]
\label{app-lem:cancel}
Let $f\in L^1(\Do)$. Let $B=B(|v-v_*|,\sigma)\ge 0$.
For almost all $(x,v)\in\Do$, we have 
\begin{align*}
[f*_{v}S](x,v)=\int_{\R^d\times S^{d-1}}B(v-v_*,\sigma)(f(x,v_*')-f(x,v_*))\dd v_*\dd \sigma, 
\end{align*}
where $S$ is given by
\begin{equation*}
\label{def-S}
S(|z|)=|S^{d-2}|\int_0^{\frac{\pi}{2}}\Big(\frac{B(|z|/\cos(\theta/2),\theta)}{\cos^{d}(\theta/2)}-B(|z|,\theta)\Big)\sin^{d-2}\theta\dd\theta.
\end{equation*}
And the following estimates hold
\begin{equation*}
|S(|z|)|\le 2^{\frac{d-4}{2}} \cos^{-2}(\pi/8)\big(d\Lambda(z)+|z||\Lambda'(z)|\big).    
\end{equation*}
\end{lemma}

\subsubsection{Hard and Maxwellian cases}\label{sec:hard}
In the hard and Maxwellian case \eqref{A-hard}, we show the following lemma.
\begin{lemma}\label{lemma-3}
 Let $f^\varepsilon$ be given as in Lemma \ref{weak:conv}. Let $\alpha(t)=\sqrt{t}$ for $t\ge0$. For any $\varepsilon\in(0,1)$, we  have 
\begin{equation*}
\label{int-alpha-2}
\alpha'(f^\varepsilon)Q^\varepsilon_{\aB}(f^\varepsilon,f^\varepsilon)\in L^1([0,T]\times \Do).
\end{equation*}
\end{lemma}
Then we apply the velocity averaging Lemma \ref{thm:ave} to the renormalised equation \eqref{FBE-alpha} and $\rho^\delta\in C^\infty_c(\R^d)$ for any fixed $\delta\in(0,1)$ to finish the proof of Lemma \ref{lemma-1} that $\{\sqrt{f^\varepsilon}*_v\rho^\delta\}$ is relatively compact in $L^2_\loc$.

\medskip
\begin{proof}[Proof of Lemma \ref{lemma-3}]
For the notation simplicity, we drop the parameter $\varepsilon$.
Following \eqref{split}, we write 
\begin{align*}
    \alpha'(f)Q_{\aB}(f,f)=I_1+I_2,  
\end{align*}
where 
\begin{align*}
    I_1&=\frac{1}{2\sqrt {f}}\int_{\Do\times S^{d-1}} B \kappa f\big(f_*'-f_*\big)\dd \sigma\dd x_*\dd v_*,\\
     I_2&=\frac{1}{2\sqrt {f}}\int_{\Do\times S^{d-1}} B \kappa f_*'\big(f'- f\big)\dd \sigma\dd x_*\dd v_*.
\end{align*}
By Lemma \ref{app-lem:cancel}, $I_1$ can be written as
\begin{align*}
    I_1=\frac{\sqrt {f}}{2}\int_{\R^d}\kappa \big[f_**_{v_*}S](x_*,v)\dd x_*.
\end{align*}
Moreover, the cross-section bound \eqref{cross-bd} and \eqref{cross-bd-sim} ensures
\begin{align*}
    \|I_1\|_{L^2(\Do)}^2&\lesssim \int_{\Do}f\Big|\int_{\R^d} \big[f_**_{v_*}S](x_*,v)\dd x_*\Big|^2\dd x\dd v\\
    &\lesssim \int_{\Do}f\Big|\int_{\Do} f_*\langle v-v_*\rangle\dd x_*\dd v_*\Big|^2\dd x\dd v.
\end{align*}
Since $f\in \mathbb{P}(\Do)$ and $t\mapsto t^2$ is convex, Jensen's inequality implies that
\begin{equation}
    \label{soft:TD}
\begin{aligned}
    \|I_1\|_{L^2(\Do)}^2
    &\lesssim \int_{\G}f\Big|\int_{\Do} f_*\langle v-v_*\rangle\dd x_*\dd v_*\Big|^2\dd x\dd v\\
     &\lesssim \int_{\G}f f_*\langle v-v_*\rangle^{2}\dd\eta\lesssim \|f\|_{L^1_{0,2}(\Do)}\le C
\end{aligned}
\end{equation}
by \eqref{uni-bdd:D}.
The proof for the $I_2$ term follows analogously.
\end{proof}

\subsubsection{Soft potential cases}
\label{sec:rem}

We recall the cross-section bound \eqref{cross-bd} in the soft potential case 
\begin{align*}
    \Lambda(|v-v_*|)+|v-v_*|\Lambda'(|v-v_*|)\lesssim|v-v_*|^\gamma,\quad \gamma\in[-2,0).
\end{align*}
The proof for the hard and Maxwellian potential cases does not directly apply, since we lack of the bound boundedness of \eqref{soft:TD}, i.e.
\begin{align*}
\int_{\G}ff_*|v-v_*|^{2\gamma}\dd \eta.    
\end{align*}
Thus, we approximate $\alpha(t)=\sqrt t$ by a sequence of bounded functions 
\begin{align*}
\alpha^r (t)=\frac{\sqrt{t}}{1+r\sqrt t}\le \alpha(t),\quad  r\in(0,1). \end{align*}
Notice that for any fixed $ r\in(0,1)$, $\alpha^r (t)$ is concave and $\alpha^r (t)\le r^{-1}$ is bounded. 
Let $f^\varepsilon$ be $\cH$-solutions to the fuzzy Boltzmann equation \eqref{def:scaling}.
The following renormalised equation holds in the distribution sense
\begin{equation}
    \label{rem:FBE}
(\d_t+v\cdot\nabla_x)\alpha^r (f^\varepsilon)=(\alpha^r )'(f^\varepsilon) Q^\varepsilon_{\aB}(f^\varepsilon,f^\varepsilon).
\end{equation}
A similar sequence of renormalised solutions was introduced and studied for the spatial-inhomogeneous Boltzmann equations to relax the quadratic form in $x$ that appeared in the inhomogeneous collision operators, see for example \cite{DL89,AV02}. 

We show the following lemma.
\begin{lemma}\label{lemma-4}
    Let $f^\varepsilon$ be given as in Lemma \ref{weak:conv}. For any fixed $r\in(0,1)$ and all $\varepsilon\in(0,1)$, we have 
    $$(\alpha^r )'Q(f^\varepsilon,f^\varepsilon)\in L^1([0,T];L^1(B^R_x\times B^R_v))+L^\infty([0,T];L^1(B^R_x;W^{-2,1}(B^R_v))).$$
\end{lemma}

\medskip

\begin{proof}[Proof of Lemma \ref{lemma-1}]
We recall the mollified $\rho^\delta\in C^\infty_c(\R^d;\R_+)$.
For any fixed $r,\,\delta\in(0,1)$,
Lemma \ref{lemma-4} allows us to apply the velocity averaging theorem \ref{thm:ave}  to the renormalised equations \eqref{rem:FBE} to derive
\begin{align*}
\{\alpha^r (f^\varepsilon)*_v\rho^\delta\}\quad \text{is relatively compact in $L^1([0,T]\times B^R_x\times B^R_v)$}. 
\end{align*}
We pass the limit by letting $r\to0$. To show the relative compactness of $g^\varepsilon_\delta=\alpha(f^\varepsilon)*_v\rho^\delta$, we only need to show that 
\begin{equation}
  \label{conv:var}
\lim_{r\to0}\sup_{t,\varepsilon}\|\sqrt{f^{\varepsilon}}-\alpha^r (f^{\varepsilon})\|_{L^2(\Do)}=0.
  \end{equation}

We follow a standard argument to show \eqref{conv:var}, see for example \cite{DL89}. Since $t\mapsto \frac{t}{1+r t}$ is concave,  for any $R\ge0$,  we have $0\le t-\frac{t}{1+r t}\le C_r(R)+t\mathbb{1}_{\{t\ge R\}}$ for some constant $C_r(R)=R-\frac{R}{1+r R}\ge 0$. Notice that, for any fixed $R\ge 0$,  we have $C_r(R)\to0$ as $r\to0$. Then we have
\begin{align*}
\sup_{t,\varepsilon}\int_{\Do} |\alpha^r (f^\varepsilon)-\sqrt{f^{\varepsilon}}|^2\dd x\dd v&\le C_r(R)^2+\sup_{t,\varepsilon}\int_{\Do}f^\varepsilon\mathbb{1}_{\{\sqrt{f^\varepsilon}\ge R\}}\dd x\dd v\\
&\to0\quad\text{as}\quad r\to0\quad \text{and}\quad R\to+\infty.
  \end{align*}
Indeed, the uniform equi-integrability of $f^\varepsilon$  ensures that
\begin{equation*}
\lim_{R\to\infty}\sup_{t,\varepsilon}\int_{\Do}f^\varepsilon\mathbb{1}_{\{\sqrt{f^\varepsilon}\ge R\}}\dd x\dd v=0.
  \end{equation*}
 
We conclude that, for any fixed $\delta\in(0,1)$, $\{g^\varepsilon_\delta\}$ is relatively compact in $L^1([0,T]\times B^R_x\times B^R_v)$.

\end{proof}

\begin{proof}[Proof of Lemma \ref{lemma-4}]
 The proof follows \cite{AV02} with appropriate adaptation to the fuzzy cases. For the notation simplicity, we drop the parameter $\varepsilon$. We write 
    \begin{equation*}
    \label{rem:Q}
(\alpha^r )'(f) Q_{\aB}(f,f)=I_1+I_2+I_3,
\end{equation*}
where 
\begin{align*}
I_1&=\big(f (\alpha^r )'(f )-\alpha^r (f)\big)\int_{\Do\times S^{d-1}}\kappa   B \big(f'_*-f_*\big)\dd \sigma\dd x_*\dd v_*,\\ 
I_2&=\int_{\Do\times S^{d-1}}\kappa B  \big(f_*'\alpha^r (f')-f_*\alpha^r (f)\big)\dd \sigma\dd x_*\dd v_*,\\ 
I_3&=-\int_{\Do\times S^{d-1}}\kappa  B f'\Gamma(f,f')\dd \sigma\dd x_*\dd v_*.
\end{align*}
In the above, $\Gamma$ is defined as
$$\Gamma(f,f')=\alpha^r (f')-\alpha^r (f)-(\alpha^r )'(f)(f'-f).$$
We show that
\begin{equation}
\label{lem:R123}
I_1,\,I_3\in L^1([0,T]\times B^R_x\times B^R_v)\text{ and } I_2\in L^\infty([0,T];L^1(B^R_x; W^{-2,1}(B^R_v))).
\end{equation}

Notice that we only need to show \eqref{lem:R123} for $I_1$ and $I_2$. Indeed, if \eqref{lem:R123} holds for $I_1$ and $I_2$, then the weak formulation of \eqref{rem:FBE} implies that 
\begin{align*}
0\le -\int_0^T\int_{\Do}\varphi I_3\dd x\dd v\dd t<+\infty,\quad \text{for all}\quad \varphi\in C^\infty_c\big([0,T)\times \Do;\R_+\big).
\end{align*}
Notice that the concavity of $\alpha^r$ implies  $\Gamma(f,f')\le 0$ and hence, $I_3\le0$. Hence, we have $I_3\in L^1([0,T]\times B^R_x\times B^R_v)$.

\medskip 

We show that \eqref{lem:R123} holds for $I_1$ and $I_2$:
\begin{itemize}
    \item \textbf{$I_1$ term:}    By cancellation Lemma \ref{app-lem:cancel}, $I_1$ can be written as
\begin{equation*}
    I_1=\big(f (\alpha^r )'(f )-\alpha^r (f )\big)\int_{\R^d}[f_**_{v_*}S](x_*,v)\dd x_*. 
\end{equation*}
The definition of $\alpha^r$ ensures that $|f(\alpha^r )'(f)-\alpha^r (f)|\lesssim 1$. Hence, to show the integrability of 
$I_1$, we only need to show $[f_**_{v_*}S]\in L^1_\loc$. 
By Lemma \ref{app-lem:cancel}, we have 
\begin{align*}
\int_{\R^d\times B^R_v}[f_**_{v_*}S]\dd x_*\dd v\lesssim \int_{\Do\times B^R_v}f(x_*,v_*)|v-v_*|^\gamma\dd x_*\dd v_*\dd v.
\end{align*}
We change of the variable by letting $z=v-v_*$. For any fixed $M\ge 0$, we have
\begin{equation}
    \label{R1-arg}
\begin{aligned}
    &\int_{\Do\times B^R_v}f_*|v-v_*|^\gamma\dd x_*\dd v_*\dd v
  =\int_{\Do\times \{|z+v_*|\le R\}}f_*|z|^\gamma\dd x_*\dd v_*\dd z\\
     \le &{}\int_{\R^d\times B^M_{v_*}\times \{|z|\le R+M\}}f_*|z|^\gamma\dd z \dd x_*\dd v_*\\
     &+
     \int_{\R^d\times (B^M_{v_*})^c\times \{|z+v_*|\le R\}}
f_*|z|^\gamma\dd z \dd x_*\dd v_*\\
:=&{} A_1+A_2.
\end{aligned}
\end{equation}
The integrability of $|z|^\gamma\in L^1_\loc(\R^d)$, $\gamma\in[-2,0)$ when $d\ge 3$ and $\gamma\in(-2,0)$ in $d=2$, implies that 
\begin{align*}
     A_1&\lesssim(R+M)^{\gamma+d}\|f\|_{L^1(\Do)},\\
     A_2&\lesssim_R\int_{\R^d\times (B^M_{v_*})^c}f_*|v_*|^\gamma\dd x_*\dd v_*\le M^\gamma \|f\|_{L^1(\Do)}.
\end{align*}
Hence, we have $I_1\in L^\infty([0,T]\times B^R_x\times B^R_v)$.
\medskip

    \item \textbf{$I_2$ term:}
    We use the duality argument. Let $\varphi\in C^\infty_c(\R^d_v)$ be fixed. By changing of variables, we have
\begin{align*}
    \int_{\R^d}I_2\varphi\dd v=\int_{\R^{3d}\times S^{d-1}}Bf_*\alpha^r(f)(\varphi'-\varphi)\dd x_*\dd v_* \dd x\dd \sigma.
\end{align*}

By Taylor's expansion as in Lemma \ref{lem:phi:conv}, we have 
\begin{align*}
    \varphi-\varphi'=(v-v')\cdot\nabla_v \varphi(x,v)+O\big(\|D_v^2\varphi\|_{L^{\infty}}|v-v'|^2\big).
\end{align*}
By Proposition \ref{rmk:size-est}, we have
\begin{align*}
|v-v'|^2=\frac12|v-v_*|^2(1-\sigma\cdot k).
\end{align*} The cross-section identity \eqref{cpt:S} and the definition \eqref{cross} imply that, for all $(x,t)\in\R^d\times[0,T]$, 
\begin{align*}
    &\Big|\int_{S^{d-1}}B(\varphi'-\varphi)\dd \sigma\Big|\\
    \lesssim&{} \|\nabla_v\varphi\|_{L^{\infty}}|v-v_*|\Lambda(|v-v_*|)+\|D^2_v\varphi\|_{L^\infty}|v-v_*|^2\Lambda(|v-v_*|)\\
    \lesssim&{} \|\varphi\|_{W^{2,\infty}(\R^d_v)}\max\big(|v-v_*|,|v-v_*|^2\big) \Lambda(|v-v_*|).
\end{align*}
Hence, we have
\begin{align*}
    &\|I_2\|_{W^{-2,1}(B^R_v)}=\sup_{\varphi\in W^{2,\infty}(B^R_v)}\Big\{\|\varphi\|_{W^{2,\infty}}^{-1}\int_{\R^d}I_2\varphi\dd v\Big\}\\
    \le&{}\int_{\R^d\times B^R_v}f_*\alpha^r(f)\max\big(|v-v_*|,|v-v_*|^2\big) \Lambda(|v-v_*|)\dd v_*\dd v\\
    \le&{}\|\alpha^r\|_{L^\infty}\int_{\R^d \times B^R_v}f_*\max\big(|v-v_*|^{\gamma+1},|v-v_*|^{\gamma+2}\big) \dd v_*\dd v,
\end{align*}
where we use the cross-section bound \eqref{cross-bd} to derive the last inequality. We integrate $\|R_2\|_{W^{-2,1}(\Do)}$ over $x\in B^R_x$ to derive 
\begin{align*}
    \|I_2\|_{L^1(B^R_x;W^{-2,1}(B^R_v))}&\lesssim_r\int_{\Do \times B^R_v}f_*\max\big(|v-v_*|^{\gamma+1},1\big) \dd v_*\dd v\\
    &\lesssim_{R}\|f\|_{L^1(\Do)},
\end{align*}
where we repeat the argument as in \eqref{R1-arg} to derive the last inequality. 

Hence, we have $I_2\in L^\infty([0,T];L^1(B^R_x;W^{-2,1}(B^R_v)))$.

\end{itemize}
\end{proof}

\subsection{Proof of Lemma \ref{lemma-2}}\label{sec:smooth}
In this subsection, our goal is to show
    \begin{equation}
    \label{uni-conv}
\sqrt{f^\varepsilon}*_v\rho^\delta\to \sqrt{f^\varepsilon}   \quad\text{in}\quad L^2(B^R_x\times B^R_v)
\end{equation}
as $\delta\to0$ uniformly in $\varepsilon\in(0,1)$.

We take  $\chi_R\in C^\infty_c(\R^d;[0,1])$ such that $\chi_R=1$ on $B^R_v\subset \R^d$ and $\supp(\chi_R)\subset B^{R+1}_v$. We define
\begin{align*}
    f_R^\varepsilon= f^\varepsilon(x,v)\chi_R(v).
\end{align*}
By the triangle inequality, we have
\begin{align*}
&\big\|\big(\sqrt{f^\varepsilon}*_v\rho^\delta- \sqrt{f^\varepsilon}\big)\chi_R\big\|_{L^2}\\
\le&{} \big\|\big(\sqrt{f^\varepsilon}*_v\rho^\delta\big)\chi_R-\sqrt{f^\varepsilon_R}*_v\rho^\delta\big\|_{L^2}+\big\|\sqrt{f^\varepsilon_R}*_v\rho^\delta-\sqrt{f^\varepsilon_R} \big\|_{L^2}.  
\end{align*}
Let $a\in(0,1)$. Since $\supp(\rho^\delta)\subset B_\delta(0)$ and $\|\rho^\delta\|_{L^1}=1$ for all $\delta\in(0,1)$, we have
\begin{align*}
&\big\|\big(\sqrt{f^\varepsilon}*_v\rho^\delta\big)\chi_R-\sqrt{f^\varepsilon_R}*_v\rho^\delta\big\|_{L^2(B^R_x\times \R^d_v)}\\
\le&{}\underbrace{\int_{B_a(0)}\rho^\delta}_{\le 1} \sup_{|h|\le a}\big\|\sqrt{f^\varepsilon}\big(\chi_R(\cdot+h)-\chi_R(\cdot)\big)\big\|_{L^2(B^R_x\times \R^d_v)}\\
&+C_R\underbrace{\sup_{|v|\ge a}\rho^\delta(v)}_{\to0 \text{ as }\delta\to0} \big\|\sqrt{f^\varepsilon}\big\|_{L^2(\Do)}\to0 
\end{align*}
as $\delta\to0$ uniformly in $\varepsilon\in(0,1)$, where $C_R>0$ is a constant independent of $\varepsilon$ and $\delta$. Indeed, 
the equi-integrability of $f^\varepsilon$ and $\sup_{|h|\le a}\big|\supp\big(\chi_R(v+h)-\chi_R(v)\big)\big|\to0$ ensures the uniform convergence of the first term on the right-hand side, and the uniform bound of $\sqrt{f^\varepsilon}$ in $L^2(\Do)$ ensures the uniform convergence of the second term.

\medskip

Hence, to show \eqref{uni-conv}, we only need to show 
 \begin{equation*}
    \label{uni-conv-R}
\sqrt{f^\varepsilon_R}*_v\rho^\delta\to \sqrt{f^\varepsilon_R}   \quad\text{in}\quad L^2(B^R_x\times \R^d_v)
\end{equation*}
 as $\delta\to0$ uniformly in $\varepsilon\in(0,1)$.  

We define the Fourier transform in $v\in\R^d$
\begin{align*}
\cF_v f(x,\xi):=(2\pi)^{-\frac{d}{2}}\int_{\R^d} f(x,v)e^{-iv\cdot\xi}\dd v.
\end{align*}
By the Plancherel theorem, we have 
\begin{equation}
    \label{uni-conv:2}
\|\sqrt{f^\varepsilon_R}*_v\rho^\delta-\sqrt{f^\varepsilon_R}\|_{L^2(B^R_x\times\R^d_v)}^2=\int_{B^R_x\times\R^d}  |\cF_v{\sqrt{f^\varepsilon_R}}(x,\xi)(1-\cF_v{\rho^\delta})|^2\dd \xi\dd x. 
\end{equation}
We recall $\rho^\delta\in C^\infty_c(\Do)$ and $\|\rho^\delta\|_{L^1}=1$, for all $\delta\in(0,1)$. Thus, to show the uniform convergence of \eqref{uni-conv:2}, we only need to show that
\begin{equation*}
\label{uni-vanish}
\int_{B^R_x\times\R^d}|\cF_v{\sqrt{f^\varepsilon_R}}(x,\xi)|^2\mathbb{1}_{\{|\xi|\ge M\}}\dd \xi\dd x\to0
\end{equation*} 
as $M\to+\infty$ uniformly in $\varepsilon$. More precisely, we show that
\begin{equation}
    \label{bdd:sqrt-f}
\int_{B^R_x\times\R^d}|\cF_v{\sqrt{f^\varepsilon_R}}(x,\xi)|^2\min(|\xi|^2,|\xi|^\nu)\dd\xi\dd x\le C_R\big(1+\cD^\varepsilon_B(f^\varepsilon)\big).
\end{equation}
The constant $C_R>0$ is independent of $\varepsilon$. 
The constant $\nu\in(0,2)$ is given in Assumption \ref{ass:ang}, which characterises  the angular singularity that  $ \beta(\theta)\gtrsim \theta^{-1-\nu}$.

\medskip

The rest of this subsection is devoted to showing \eqref{bdd:sqrt-f}.

\noindent\textbf{Proof of \eqref{bdd:sqrt-f}:} For the spatial-homogeneous cases, the estimate \eqref{bdd:sqrt-f} has been shown in \cite[Appendix B]{carrillo2022boltzmann} and \cite[Section 4-5]{ADVW00} by using the idea of truncation and the Fourier representation (in $v$). In the following, we follow the proof of the homogeneous case to outline the main ideas of the proof. When the spatial variables $x$ and $x_*$
 can be regarded as parameters, we directly apply the technical results established for the spatial-homogeneous equations. We provide detailed proofs only when the spatial variables play a nontrivial role.

We remove the kinetic singularity of the collision kernel $B^\varepsilon$ by defining
\begin{equation}
\label{bar-C}
    \bar A_0(|z|):=\min\big(\bar C, A_0(|z|)\big),\quad \bar B^\varepsilon(z,\sigma)=\bar A_0(|z|)b^\varepsilon (\theta)
\end{equation}
for some $\bar C>0$. In the following subsection, we replace
$B^\varepsilon$ by $\bar B^\varepsilon$  in the entropy dissipation $\cD_B^\varepsilon(f)$.

We define
$$\mathbb{f}^\varepsilon_*=\hf^\varepsilon(x,v_*)=f^\varepsilon*_x\kappa^\varepsilon=\int_{\R^d}\kappa^\varepsilon(x-x_*)f^\varepsilon(x_*,v_*)\dd x_*.$$
By Assumption \ref{ASS:kappa}, we have  
\begin{align*}
    \|\hf^\varepsilon\|_{L^\infty_x L^1_v}\le \|\kappa^\varepsilon\|_{L^\infty}\|f^\varepsilon\|_{L^1(\Do)}\le C_\kappa \|f^\varepsilon\|_{L^1(\Do)}.
\end{align*}

Our goal is to search for an appropriate lower bound of $\cD^\varepsilon_{\aB}(f^\varepsilon)$. Notice that  $\cD_B^\varepsilon(f^\varepsilon_R)$ can be written as
\begin{equation}
    \label{diss:cpt}
\begin{aligned}    &\cD^\varepsilon_B(f^\varepsilon_R)\\
=&{}-\int_{\R^{3d}\times S^{d-1}} \bar B^\varepsilon \big((f^\varepsilon_R)'(\hf^\varepsilon_R)_*'-f^\varepsilon_R(\hf^\varepsilon_R)_*\big)\log f^\varepsilon_R\dd\sigma \dd v_*\dd v\dd x\\
    =&{}\int_{\R^{3d}\times S^{d-1}} \bar B^\varepsilon f^\varepsilon_R(\hf^\varepsilon_R)_*\log \frac{f^\varepsilon_R}{(f^\varepsilon_R)'} \dd\sigma \dd v_*\dd v\dd x\\
    =&{}\int_{\R^{3d}\times S^{d-1}}\bar B^\varepsilon (\hf^\varepsilon_R)_*\big(f^\varepsilon_R-(f^\varepsilon_R)'\big)\dd x\dd v\dd v_*\dd\sigma\\
    &+\int_{\R^{3d}\times S^{d-1}}\bar B^\varepsilon (\hf^\varepsilon_R)_*\Big(f^\varepsilon_R\log \frac{f^\varepsilon_R}{(f^\varepsilon_R)'}-f^\varepsilon_R+(f^\varepsilon_R)'\Big)\dd x\dd v\dd v_*\dd\sigma\\
    :=&{}I_1+I_2.
\end{aligned}
\end{equation}

In the rest of this subsection, we drop $\varepsilon$  and $R$ of $f^\varepsilon_R$ for notational simplicity. 

By cancellation Lemma \ref{app-lem:cancel},  $I_1$ is bounded by
\begin{align*}
|I_1|=\Big|\int_{\Do}\hf_*[f*_v S](x,v_*)\dd x\dd v_*\Big|\lesssim \|S\|_{L^\infty}\|\kappa^\varepsilon\|_{L^\infty}\|f\|_{L^1(\Do)}^2=:C_1,
\end{align*}
where the boundedness of $\bar B^\varepsilon$ implies \begin{align*}
    | S(|z|)|\lesssim \bar A_0(|z|)\le \bar C.
\end{align*}

In the rest of this subsection, we use the notation of constants
\begin{align*}
  C_i=C_i(d,R,E,H)>0,\quad i\in\N_+,
\end{align*}
where $E=\sup_\varepsilon\|f^\varepsilon\|_{L^1_{2,2}(\Do)}<+\infty$ and $H=\sup_\varepsilon\|f^\varepsilon\log f^\varepsilon\|_{L^1(\Do)}<+\infty$  by \eqref{uni-bdd:D}. In particular, $C_i$ is independent of $\varepsilon$.

We define $g=\sqrt f$. By using the elementary inequality $a\log\frac{a}{b}-a+b\ge |\sqrt{a}-\sqrt b|^2$ for all $a,b>0$, we have 
    \begin{align*}
        I_2\ge \int_{\R^{3d}\times S^{d-1}} \bar B^\varepsilon\hf_*|g-g'|^2\dd v_*\dd x\dd v\dd\sigma.
    \end{align*}
Substituting the estimates of $I_1$ and $I_2$ to \eqref{diss:cpt}, we obtain
\begin{equation}
\label{goal:RHS}
\cD^\varepsilon_B(f)+C_1\ge\int_{\R^{3d}\times S^{d-1}} \bar A_0 b^\varepsilon(\theta)\hf_*|g-g'|^2\dd v_*\dd x\dd v\dd\sigma.
\end{equation}
We show an appropriate lower bound of the right-hand side of \eqref{goal:RHS}.
\medskip

Notice that, in the hard potential cases $A_0(|v-v_*|)=|v-v_*|^{\gamma}$, $\gamma\in(0,1]$, the kinetic kernel $\bar A_0(|v-v_*|)$ vanishes when $v\sim v_*$. 
We follow \cite{ADVW00} to treat this case. For any $r_0\in (0,1)$ and $v_j\in  B^R_v$, we define $A_j$ and $B_j$ as follows
\begin{align*}
    A_j=\{v\in B^R_v\mid |v-v_j|<r_0/4\}\quad\text{and}\quad  B_j=\{v\in B^R_v\mid |v-v_j|>r_0\}.
\end{align*}
Then there exist $r_0$ and finite $\{v_j\}_{i=1,\dots,N_j}$ such that 
\begin{align*}
B^{R+1}_v\subset\cup_{i=1}^{N_j} A_j.   
\end{align*}
Let $\chi_{A_j}\in C^\infty_c(\R^d;\R_+)$ such that $\chi_{A_j}=1$ for $v\in A_j$, $0\le \chi_{A_j}\le 1$, $\supp(\chi_{A_j})\subset  \{|v-v_j|<3r_0/8\}$ and $\Lip(\chi_{A_j})\le C$ for all $j=1,\dots,N_j$. Similarly, we define $\chi_{B_j}\in C^\infty_c(\R^d;\R_+)$ such that $B^R_v\cap \{|v-v_j|>7r_0/8\}\subset \supp(\chi_{B_j})$, $0\le \chi_{B_j}\le 1$ and $\Lip(\chi_{B_j})\le C$ for all $j=1,\dots,N_j$. Notice that $\operatorname{dist}(A_j,B_j)\ge \frac{r_0}{2}$. 

We note that, in the soft potential case and the hard potential with $A_0(|v-v_*|)\sim \langle v-v_*\rangle^\gamma$, $\gamma\in[0,1]$, the kinetic kernel $A_0(|v-v_*|)$  has a locally positive lower bound
\begin{align*}
    A_0(|v-v_*|)\mathbb{1}_{B^{R+1}_v\times B^{R+1}_{v_*}}\ge C_0>0.
\end{align*}
In this case, we can simply take $A_j=B_j=B^{R+1}_v$ and $N_j=1$.

For each $j=1,\dots,N_j$, we will follow the truncation argument \cite[Lemma 2]{ADVW00} to show a lower bound of the right-hand side of \eqref{goal:RHS}. 
\begin{lemma}\label{lemma-5}
For any $j=1,\dots,N_j$, we have 
\begin{equation}
\label{app:bdd:B-2}
    \begin{aligned}
&\int_{\R^{3d}\times S^{d-1}} \bar A_0(|v-v_*|)b(\theta) \hf_*|g-g'|^2\dd x\dd v_*\dd v \dd\sigma+C_2\\
\ge &{}\min_{r_0/2\le |z|\le 2\sqrt2 R}\bar A_0(|z|)\int_{\R^{3d}\times S^{d-1}}b(\theta) (\hf_{B_j})_*|g_{A_j}-g_{A_j}'|^2\dd x\dd v_*\dd v\dd\sigma.
\end{aligned}
 \end{equation}
\end{lemma}
Then by directly applying the Fourier representation results in \cite[Proposition 1 and Corollary 3]{ADVW00} and integrating over $x\in \R^d$, we have 
  \begin{equation}
      \label{trun-2}
  \begin{gathered}
\int_{\R^{3d}\times S^{d-1}}b(\theta)(\hf_{B_j})_*|g_{A_j}-g_{A_j}'|^2\dd x\dd v_*\dd v\dd\sigma
\ge\frac{1}{(2\pi)^d}\int_{\Do}|\cF_v g_{A_j}(x,\xi)|^2\\
\times\Big\{\int_{S^{d-1}}b(\theta_\xi)\big(\cF_v \hf_{B_j}(x,0)-|\cF_v \hf_{B_j}(x,\xi^-)|\big)\dd\sigma\Big\}\dd\xi\dd x,
\end{gathered}
\end{equation}
where $\xi^-:=\frac{\xi-|\xi|\sigma}{2}$ and $\theta_\xi:=\arccos\big(\sigma\cdot \xi/|\xi|\big)$.

We shortly bring $\varepsilon$ back in the following lemma. Following \cite[Lemma 3]{ADVW00}, the equi-integrability of $f^\varepsilon$ and the uniform lower bound of $\kappa^\varepsilon$ imply the following pointwise lower bounds.
\begin{lemma}\label{claim-2}
 Let $\varepsilon\in(0,1)$. Let $f^\varepsilon$ be given as in Lemma \ref{weak:conv} and $\hf^\varepsilon=f*_x \kappa^\varepsilon$. For almost all $\xi\in\R^d$ and $x\in B^R_x$, we have 
\begin{equation}
\label{stp-2:goal}
\cF_v\hf^\varepsilon(x,0)-|\cF_v\hf^\varepsilon(x,\xi^-)|\ge C_3\min(1,|\xi^-|^2).
\end{equation}
\end{lemma}
Then the identity $|\xi^-|^2=\frac{|\xi|^2}{2}\big(1-\sigma\cdot \xi/|\xi|\big)$ and the angular assumption \eqref{nu} that $\beta(\theta)\gtrsim \theta^{-1-\nu}$ for some $\nu\in(0,2)$, we have the pointwise lower bound
\begin{equation}
\label{trun-4}
    \int_{S^{d-1}}b^\varepsilon(\theta) \min (1,|\xi^-|^2)\dd\sigma\ge C_4\min (|\xi^2|,|\xi|^\nu)\quad\forall \varepsilon\in(0,1).
\end{equation}
The detailed proof of \eqref{trun-4} can be found in 
\cite[Lemma 4]{ADVW00}  and \cite[Lemma B.4]{carrillo2022boltzmann}.

Now we combine the estimates \eqref{app:bdd:B-2}, \eqref{trun-2}, \eqref{stp-2:goal} and \eqref{trun-4}. For any $j=1,\dots,N_j$, we have 
\begin{equation}
    \label{goal:j}
\begin{aligned}
&\int_{\R^{3d}\times S^{d-1}}\bar B^\varepsilon(|v-v_*|)b(\theta) \hf_*|g-g'|^2\dd x\dd v_*\dd v \dd\sigma+C_2\\
\ge&{} C_5\int_{B^R_x\times\R^d}|\cF_v g_{A_j}(x,\xi)|^2\min(|\xi|^2,|\xi|^\nu)\dd\xi\dd x.
\end{aligned}
\end{equation}
Notice that, one can take the sum over $j=1,\dots,N_j$ in \eqref{goal:j} by using
\begin{align*}
\sum_{i=1}^{N_j}|\cF_v g_{A_j}(x,\xi)|^2\gtrsim \Big|\sum_{i=1}^{N_j}\cF_v g_{A_j}(x,\xi)\Big|^2 \ge|\cF_v g(x,\xi)|^2. 
\end{align*}
We recall the definition $g=\sqrt{f^\varepsilon_R}$. Substituting the sum of \eqref{goal:j} to \eqref{goal:RHS}, we obtain \eqref{bdd:sqrt-f}
\begin{align*}
\int_{B^R_x\times \R^d}|\cF_v {\sqrt{f^\varepsilon_R}}(x,\xi)|^2\min(|\xi|^2,|\xi|^\nu)\dd x\dd\xi\le C_R \big(1+\cD^\varepsilon_B(f^\varepsilon)\big).
\end{align*}

\medskip

We prove Lemma \ref{lemma-5} and Lemma \ref{claim-2}.

\begin{proof}[Proof of Lemma \ref{lemma-5}]
We follow 
\cite[Lemma 2]{ADVW00} with appropriate modification for the fuzzy case. 
By the definition of $\chi_{A_j}$ and $\chi_{B_j}$, we have 
\begin{align*}
\hf_*|g-g'|^2&\ge  \hf_*|g-g'|^2(\chi_{B_j})_*\chi_{A_j}^2= (\hf_{B_j})_* |g-g'|^2\chi_{A_j}^2\\
&\ge (\hf_{B_j})_*\big( |g_{A_j}-g_{A_j}'|^2-|g'(\chi_{A_j}-\chi_{A_j}')|^2\big),
\end{align*}
where we use 
\begin{align*}
|g_{A_j}-g_{A_j}'|^2&=|g\chi_{A_j}-g'\chi_{A_j}+g'\chi_{A_j}-g'\chi_{A_j}'|^2\\
&\le|g-g'|^2\chi_{A_j}^2+|g'(\chi_{A_j}-\chi_{A_j}')|^2.
\end{align*}

Hence, for any $i\in\{1,\dots,N_j\}$, we have 
\begin{equation*}
\label{smooth:bdd-2}
\begin{aligned}
&\int_{\R^{3d}\times S^{d-1}}\bar A_0 b(\theta) \hf_*|g-g'|^2\dd x\dd v_*\dd v\dd\sigma\\
\ge&{}\min_{r_0/2\le |z|\le 2\sqrt2 R}\bar A_0(|z|)\int_{\R^{3d}\times S^{d-1}}b(\theta) (\hf_{B_j})_*|g_{A_j}-g_{A_j}'|^2\dd x\dd v_*\dd v\dd\sigma\\
&-\bar C\underbrace{\int_{\R^{3d}\times S^{d-1}}b(\theta) (\hf_{B_j})_*f'|\chi_{A_j}-\chi_{A_j}'|^2\dd x\dd v_*\dd v\dd\sigma}_{=:I_3},
\end{aligned}
\end{equation*}
where we have $|\bar A_0|\le \bar C$ by definition \eqref{bar-C}.

We are left to show the upper bound of $|I_3|$.
By Lemma \ref{lem:phi:conv}, we have
\begin{align*}
|\chi_{A_j}-\chi_{A_j}'|^2\lesssim \Lip(\chi_{A_j})\theta^2|v-v_*|^2.   
\end{align*}
By definition $v'=\frac{v+v_*}{2}+\frac{|v-v_*|}{2}\sigma$, we have  $\Big|\frac{\dd v}{\dd v'}\Big|\lesssim 1$, since
\begin{align*}
    \Big|\frac{\dd v'}{\dd v}\Big|=\frac12 | I_d+k\otimes \sigma|=\frac{1+k\cdot \sigma}{2^d}\ge 2^{-d}.
\end{align*} 
Then by changing of variable $v\mapsto v'$, we have 

\begin{align*}
|I_3|&\lesssim \int_{\R^{3d}\times S^{d-1}}\theta^2b^\varepsilon(\theta)  (\hf_{B_j})_*f'|v-v_*|^2\dd v'\dd v_*\dd x\dd \sigma\\
&= \int_{\R^{3d}\times S^{d-1}}\theta^2b^\varepsilon(\theta)  (\hf_{B_j})_*f|v'-v_*|^2\dd \theta \dd v\dd v_*\dd x\dd \sigma\\
&\le 2\Big(\int_0^{\frac{\pi}{2}}\theta^2\beta^\varepsilon(\theta)\dd\theta\Big)\int_{\R^{\G}}  f_*f|v-v_*|^2 \dd \eta\lesssim \|f\|_{L^1_{0,2}(\Do)},
\end{align*}
where we use the finite angular momentum assumption \eqref{def:beta-int} and the bound $|v-v'|^2\le |v-v_*|^2$ in the last inequality.

\end{proof}

\begin{proof}[Proof of Lemma \ref{claim-2}] We follow 
\cite[Lemma 3]{ADVW00} with appropriate modification for the fuzzy case. 
It is sufficient to show that 
\begin{equation*}
\label{stp-2:goal-2}
\cF_v\hf^\varepsilon(x,0)-|\cF_v\hf^\varepsilon(x,\xi^-)|\ge C_3\min(1,|\xi^-|^2)\quad \text{for all $\xi\in\R^d$ and $x\in B^R_x$}. 
\end{equation*}

 Let $\tau\in[0,2\pi)$ denote the angle between the real and imaginary parts of $\cF_v \hf^\varepsilon(x,\xi)$. 
Then we have
\begin{align*}
 |\cF_v \hf^\varepsilon(x,\xi)|&=(2\pi)^{-\frac{d}{2}}\Big|\int_{\R^d}\hf^\varepsilon(x,v)e^{-iv\cdot\xi}\dd v\Big|\\
 &=(2\pi)^{-\frac{d}{2}}\int_{\R^d}\hf^\varepsilon(x,v)\big(\cos(\xi\cdot v)\cos\tau-\sin(\xi\cdot v)\sin\tau \big)\dd v\\
&=(2\pi)^{-\frac{d}{2}}\int_{\R^d}\hf^\varepsilon(x,v)\cos(\xi\cdot v+\tau)\dd v,
\end{align*}
and 
\begin{align*}
\cF_v \hf^\varepsilon(x,0)- |\cF_v \hf^\varepsilon(x,\xi)|
&=2(2\pi)^{-\frac{d}{2}}\int_{\R^d}\hf^\varepsilon(x,v)\sin^2\Big(\frac{\xi\cdot v+\tau}{2}\Big) \dd v.
\end{align*}
For some $\delta>0$ small to be determined later, we define 
\begin{align*}
    A_{\xi,\delta}=\{v\in \R^d\mid |v\cdot \xi+\tau-2n\pi|\ge 2\delta \quad \forall n\in\mathbb{Z}\}.
\end{align*}
For some $L>0$ large to be determined later, we have 
\begin{align*}
    \R^d_v=(B^L_v)^c\cup (B^L_v\cap (\R^d\setminus A_{\xi,\delta}))\cup(B^R_v\cap A_{\xi,\delta}).
\end{align*}
Using this, we can write
\begin{equation}
    \label{coe}
\begin{aligned}
&\cF_v \hf^\varepsilon(x,0)- |\cF_v \hf^\varepsilon(x,\xi)|\\
\ge &{} 2(2\pi)^{-\frac{d}{2}}\sin^2\delta \int_{A_{\xi,\delta}}\hf^\varepsilon(x,v)\dd v\\
\ge&{}2(2\pi)^{-\frac{d}{2}}\sin^2\delta \Big(\int_{\R^d\setminus \big(B^L_v\cap(\R^d\setminus A_{\xi,\delta})\big)}\hf^\varepsilon(x,v)\dd v-\int_{(B^R_v)^c}\hf^\varepsilon(x,v)\dd v\Big).
\end{aligned}
\end{equation}
Notice that $v\in B^L_v\cap(\R^d\setminus A_{\xi,\delta})$ if and only if there exists $n_0\in\Z$ such that
\begin{equation}
\label{A:delta}
    \Big|v\cdot \frac{\xi}{|\xi|}+\frac{\tau}{|\xi|}-2\pi\frac{n_0}{|\xi|}\Big|\le \frac{2\delta}{|\xi|}.
\end{equation}
In the case of $|\xi|\ge 1$, \eqref{A:delta} is bounded by $2\delta$, and we have \begin{align*}
|B^L_v\cap(\R^d\setminus A_{\xi,\delta})|\le (2L)^{d-1}\frac{4\delta}{|\xi|}(1+L|\xi|/\pi)\lesssim \delta.   
\end{align*}

The uniform bounds \eqref{uni-bdd:D} ensures that $f^\varepsilon$ is equi-integrable. We choose $L$ large enough and $\delta$ small enough such that
\begin{gather*}
\int_{B^L_x\times B^L_v} f^\varepsilon\dd x\dd v\ge \frac12\quad \text{and}\quad \int_{B^L_x\times \big(B^L_v\cap (\R^d\setminus A_{\xi,\delta})\big)} f^\varepsilon\dd x\dd v\le \frac18
\end{gather*}
for all $\varepsilon\in(0,1)$, and hence, 
\begin{gather*}
\int_{B^L_x\times (B^L_v \cap A_{\xi,\delta})} f^\varepsilon\dd x\dd v\ge\frac38.
\end{gather*}
Combining with the assumption \eqref{kappa-2} that $\kappa^\varepsilon(x-x_*)\ge C_0>0$ for all $x\in B^R_x$ and $x_*\in B^L_x$, we have
\begin{align*}
\int_{\R^d\setminus \big(B^L_v\cap(\R^d\setminus A_{\xi,\delta})\big)}\hf^{\varepsilon}(x,v)\dd v\ge \int_{B^L_x\times (B^L_v \cap A_{\xi,\delta})} \kappa^{\varepsilon}(x-x_*)f^\varepsilon(x_*,v)\dd x_*\dd v \ge \frac{3C_0}{8} 
\end{align*}
for all $x\in B^R_x$ and $\varepsilon\in(0,1)$. Since $\||v|^2f^\varepsilon\|_{L^1(\Do)}$ is uniformly bounded, we choose $L$ large enough such that 
\begin{align*}
\sup_{x\in \R^d}\frac{\||v|^2\hf^\varepsilon\|_{L^1_v}}{L^2}=\frac{\||v|^2f^\varepsilon\|_{L^1(\Do)}}{L^2}\le \frac{C_0}{8}.
\end{align*}

Hence, the inequality \eqref{coe} has the following lower bound
\begin{equation}
\label{ineq-C4}
\begin{aligned}
\cF_v \hf^\varepsilon(x,0)- |\cF_v \hf^\varepsilon(x,\xi)|
\ge\frac{C_0}{2(2\pi)^{\frac{d}{2}}}\sin^2\delta
\end{aligned}
\end{equation}
and \eqref{stp-2:goal} holds by choosing $C_4=\frac{C_0}{2(2\pi)^{\frac{d}{2}}}\sin^2\delta$.

In the case of $|\xi|<1$, we choose $\delta= \delta_0 |\xi|$ for some $\delta_0>0$ small. Then \eqref{A:delta} implies that 
\begin{align*}
|B^L_v\cap(\R^d\setminus A_{\xi,\delta})|\le 4\delta_0(2L)^{d-1}(1+L/\pi)\lesssim \delta_0.   
\end{align*}
We repeat the arguments in the case of $|\xi|\ge 1$ by choosing $\delta_0>0$ small to arrive at \eqref{ineq-C4}. We note that
\begin{align*}
\sin^2\delta \ge \delta_0^2 \inf_{|\xi|\le 1}\Big|\frac{\sin(\delta_0|\xi|)}{\delta_0|\xi|}\Big|^2\ge \sin^2\delta_0. 
\end{align*}

We conclude that the inequality \eqref{stp-2:goal} holds by choosing $C_4=\frac{C_0}{2(2\pi)^{\frac{d}{2}}}\sin^2\delta_0$.
\end{proof}

\section{Convergence of dissipation}\label{sec:dissipation}
We recall the definition of the $\cosh$-dissipation \eqref{D:cosh}
\begin{align*}
\cD_{\cosh}^\varepsilon(f)=\frac12 \int_{\G\times S^{d-1}} \big|\sqrt{f'f_*'}-\sqrt{ff_*}\big|^2B^\varepsilon\kappa^\varepsilon  \dd\eta\dd\sigma.
\end{align*}
In this section, we show the following theorem.
\begin{theorem}
\label{main-thm-dissipation}
Let $\varepsilon\in(0,1)$. Let $B^\varepsilon$ and $\kappa^\varepsilon$ satisfy Assumption \ref{ASS:kappa}. Let $f^\varepsilon$ be $\cH$-solutions of \eqref{def:scaling} with initial values satisfying Assumption \ref{ass:curve}. Let $f$ be the weak limit of $f^\varepsilon$ given in Lemma \ref{weak:conv}. Then we have 
\begin{equation}
   \label{dissipation:goal}
   \frac12\DL(f) \le \liminf_{\varepsilon\to0} \cD^\varepsilon_{\cosh}(f^\varepsilon).
\end{equation}
\end{theorem}
As a direct consequence of Remark \ref{lem:psi:low}, we have the following corollary.
\begin{corollary}
Let $(\Psi^*,\Theta)$ satisfy Assumption \ref{ass-pair} and \ref{ass:Psi}. Let $f^\varepsilon$ and $f$ be given as in Theorem \ref{main-thm-dissipation}.  Then we have   
\begin{equation*}
   \frac12\DL(f) \le \liminf_{\varepsilon\to0} \cD_{\Psi^*}(f^\varepsilon).
\end{equation*}
\end{corollary}

\begin{remark}[Affine representation]
\label{rmk:affine}
In Section \ref{sec:dissipation} and  Section \ref{sec:action}, we show the limit of entropy dissipation and curve action by using the idea of affine representations. Let $\Psi,\Psi^*:\R\to\R$ be a smooth convex dual pair. Then we have
\begin{align*}
    \int\Psi^*(\xi)=\sup_{\phi\in C^\infty_c}\Big\{\int \xi\cdot\phi-\int\Psi(\phi)\Big\}.
\end{align*}
We follow the affine-representation approach developed for the homogeneous case \cite{carrillo2022boltzmann}, adapting it appropriately to incorporate the spatial variable. The use of affine representations based on Legendre transformations has appeared previously in the literature, for example, in \cite{leonard1995large, Ott01}.
\end{remark}

First, we extend the test functions from $\phi=\phi(x,v)$ to 
$\Phi=\Phi(x,x_*,v,v_*)$ in Section \ref{subsec:extension}. Then, Section \ref{subsec:proof-dissipation} contains the detailed proof of Theorem \ref{main-thm-dissipation} 

\subsection{Extension of notations}\label{subsec:extension}
To simplify the proof of Theorem \ref{main-thm-dissipation}, we follow \cite[Section 3]{carrillo2022boltzmann} with appropriate modifications concerning the spatial variables $x$. More precisely, we extend the notation of $\widetilde \nabla$ to functions $\Phi\in \G$, for which the elliptic equation $\widetilde\nabla\cdot\widetilde \nabla\Phi$ is solved in Lemma \ref{lem:elliptic}. \\ \\
In the rest of this section, we write
\begin{align*}
    F=ff_*,\quad F'=f'f_*'.
\end{align*}
By Lemma \ref{lemma:cpt:sqrt}, we have
\begin{align*}
    \sqrt{F^\varepsilon}\to \sqrt{F}\quad\text{in}\quad L^2_{\loc}(\G)\quad \text{as $\varepsilon\to0$}.
\end{align*}
We recall that for a fixed $\sigma\in S^{d-1}$ the fuzzy Boltzmann gradient is defined via
\begin{equation*}
\overline\nabla\phi(x,v) = \phi(x,v')+\phi(x_*,v'_*) - \phi(x,v) - \phi(x_*,v_*).
\end{equation*}
We recall the weak formulation of the fuzzy Boltzmann collision operator
\begin{equation*}
\label{boltzmann:weak}
    \begin{aligned}
   \int_{\Do} \QB(f,f)\phi\dd x\dd v =-\frac14\int_{\G\times S^{d-1}} \kappa B (f'f_*'-ff_*) \overline\nabla\phi \dd\eta\dd\sigma 
    \end{aligned}
\end{equation*}
for all test functions $\phi:\Do\to\R$.
Notice that, one can extend the weak formulation to all test functions 
\begin{align*}
\Phi:\G\to\R    
\end{align*} by redefining $\overline\nabla_\sigma$ as
\begin{align*}
\overline\nabla_\sigma\Phi&=\Phi'_*+\Phi'-\Phi_*-\Phi,
\end{align*}
where we define
\begin{align*}
\Phi'=\Phi(x,x_*,v',v'_*),\quad \Phi_*'=\Phi(x_*,x,v_*',v')\quad\text{and}\quad \Phi_*=\Phi(x_*,x,v_*,v).
\end{align*}
For simplification, we write $\overline\nabla_\sigma=\overline\nabla$. 
We abuse the notation of $\langle Q(f,f), \Phi\rangle$. By changing of variables, the weak Boltzmann collision operators can be written as
\begin{align*}
    \langle Q_{\aB}(f,f), \Phi\rangle 
    &= -\frac14\int_{\G\times S^{d-1}}\kappa B( F'-F)\ \overline\nabla\Phi \dd \eta\dd \sigma.
\end{align*}
Similarly, we can extend the definition of the fuzzy Landau gradient
as follows 
\begin{equation}
    \label{def:big:Lgrad}
    \widetilde\nabla\Phi=\frac{\sqrt A}{2}\Pi_{(v-v_*)^\perp}\nabla_0(\Phi+\Phi_*),
\end{equation}
where we define $\nabla_0=(\nabla_v-\nabla_{v_*})$, and $\Phi_*=\Phi(x_*,x,v_*,v)$. 
Correspondingly, we can define $\widetilde\nabla\cdot$ via integration by parts. Notice that 
\begin{equation*}
    \widetilde \nabla \cdot \widetilde \nabla\Phi = \frac {A}{2} \nabla_0\cdot\big(\Pi_{(v-v_*)^\perp}\nabla_0(\Phi+\Phi_*)\big).
\end{equation*}
By the extended definition, we have
\begin{align*}
\widetilde \nabla F = \sqrt{A} \ \Pi_{(v-v_*)^\perp}\nabla_0F.
\end{align*}
We abuse the notation of $\langle Q(f,f), \Phi\rangle$. The weak Landau collision operators can be written as
\begin{align*}
    \langle \QL(f,f),\Phi\rangle &=-\int_{\GL}\kappa \widetilde\nabla F\cdot \widetilde \nabla \Phi \dd \eta = \int_{\GL} \kappa F\ \widetilde \nabla \cdot \widetilde \nabla\Phi\dd \eta.
\end{align*}

\begin{remark}[Comparison with the homogeneous Landau gradient]\label{rmk:gradient}
Compared with the homogeneous case, due to the independent spatial variables $x$ and $x_*$, we always treat $f$ and $f_*$ independently. 
The Landau gradient in the homogeneous case is defined as (see also \eqref{def:homo:Lgrad} below)
\begin{equation*}
\widetilde\nabla_{\sf homo}\Phi=\sqrt A\Pi_{(v-v_*)^\perp}\nabla_0\Phi.
\end{equation*}
In comparison, in the fuzzy case we define
\begin{align*}
 \widetilde\nabla_{\sf fuzz}\Phi=\sqrt A\ \Pi_{(v-v_*)^\perp}\nabla_0\frac{\Phi+\Phi^*}{2}. 
\end{align*}
Notice that in the symmetric case ($\Phi=\Phi_*$), for instance when $\Phi=ff_*$, we have 
\begin{align*}
    \Phi=\frac{\Phi+\Phi_*}{2}\quad\text{and}\quad \widetilde\nabla_{\sf homo}\Phi=\widetilde\nabla_{\sf fuzzy}\Phi.
\end{align*}
\end{remark}

As a direct consequence of Lemma \ref{lem:phi:conv}, the following holds.
\begin{corollary}
\label{corr:big:phi}
Let $\Phi\in \cS(\G;\R)$. The following pointwise bounds hold 
    \begin{gather}
    |\overline\nabla\Phi|\le C_1\theta |v-v_*|,\label{big:bd1}\\
  \left|\frac{1}{|S^{d-2}|}\int_{S^{d-2}_{k^\perp}}\overline\nabla\Phi\dd p\right|\le C_2 \theta^2\big(|v-v_*|+|v-v_*|^2\big), \notag
\end{gather}
where $C_1=C_1(\|\nabla_0 \phi\|_{L^\infty(\G)})$ and $C_2=C_2(\|\nabla_0 \Phi\|_{L^\infty(\G)},\|D^2_0 \Phi\|_{L^\infty(\G)})$ are positive constants independent of $\theta$.
We have the following pointwise limits  
\begin{gather}
\lim_{\theta\to0}\frac{1}{\theta}\overline\nabla\Phi=\frac{|v-v_*|}{2} p\cdot\nabla_0\big(\Phi+\Phi_*\big), \label{big:conv1}\\
\lim_{\theta\to0}\frac{1}{\theta^2}\int_{S^{d-2}_{k^\perp}}\overline\nabla\Phi\dd p=\frac{|S^{d-2}|}{8(d-1)}\nabla_0\cdot\big(|v-v_*|^2\Pi_{k^\perp}\nabla_0(\Phi+\Phi_*)\big). \label{big:conv2}
\end{gather}
\end{corollary}

\subsection{Proof of Theorem \ref{main-thm-dissipation}}\label{subsec:proof-dissipation}

{\bf Strategy of the proof:}
   We only consider the case of $\liminf_{\eps\da0}\cD_{\cosh}(f^\varepsilon)< +\infty$, otherwise there is nothing to show. We show \eqref{dissipation:goal} in the following three steps:
    \begin{enumerate}[Step $1$:]
        \item We show the affine representation inequality of the Landau dissipation
        \begin{align*}
\DL(f)\le&\sup_{\Phi\in \mathbf{DS}^{\infty}_c}\Big\{-4\int_{\R^{4d}}\sqrt{\kappa F}\ \widetilde\nabla\cdot\widetilde\nabla\Phi\deta-2\int_{\R^{4d}} |\widetilde\nabla \Phi|^2\deta \Big\}.
        \end{align*}
        The functional space $\mathbf{DS}^{\infty}_c$ is defined as in \eqref{def:Sym} below.
  \item We show the affine representation inequality of the reduced Bltzmann dissipation 
  \begin{align*}
      \cD_{\cosh}(f^\varepsilon) \ge& \sup_{\Phi\in \mathbf{DS}^{\infty}_c}\Big\{-\int_{\R^{4d}}\sqrt{\kappa^\eps F^\eps}\big(\int_{ S^{d-1}}\Bgrad \Phi B^\eps\dsigma\big)\deta -\frac18\int_{\R^{4d}\times S^{d-1}}|\Bgrad \Phi|^2 B^\eps\dsigma\deta\Big\}.
  \end{align*}
  \item We conclude \eqref{dissipation:goal} by showing
  \begin{gather*}
        \int_{\R^{4d}}\sqrt{\kappa^\eps F^\eps} \ \Big(\int_{ S^{d-1}}\overline\nabla\Phi B^\eps\dd\sigma\Big) \dd\eta \limit{\eps \da 0} 2 \int_{\R^{4d}}\sqrt{\kappa F}\widetilde\nabla \cdot \widetilde\nabla\Phi \dd\eta\\
        \text{and} \quad\frac18\int_{\R^{4d}\times S^{d-1}} |\overline\nabla\Phi|^2 B^\eps \dd\eta\dd\sigma \limit{\eps \da 0} \int_{\R^{4d}} |\widetilde \nabla \Phi|^2\dd\eta.
\end{gather*}
    \end{enumerate}

\medskip

\noindent\textbf{Step 1: Affine representation of the fuzzy Landau dissipation}\\ 

The following result extends \cite[Proposition 6.1]{carrillo2022boltzmann} to the fuzzy setting.
\begin{proposition}\label{prop:aff}
We have the following preliminary expression for the fuzzy Landau dissipation

\begin{equation}
    \label{AR:D-landau}
\begin{aligned}
    \DL(f)=\sup_{\xi\in C^\infty_c(\R^{4d};\R^d)}\Big\{&-2\int_{\R^{4d}}|\xi|^2\deta \\
    &\quad{}-4\int_{\R^{4d}}\sqrt{\kappa A  F}\ \nabla_{0}\cdot(\Pi_{(v-v_*)^\perp}\xi)\dd \eta\Big\}.   
\end{aligned}
\end{equation}
\end{proposition}

\begin{proof}
Let us denote the right-hand side of \eqref{AR:D-landau} by $\cI_\mtt{L}$ and notice that
\begin{align*}
    \DL(f)=2\int_{\R^{4d}}\big|\sqrt\kappa\ \widetilde\nabla \sqrt{F}\big|^2\deta
\end{align*}
We want to show equality $\DL(f) = \cI_\mtt{L}$.
\begin{itemize}
    \item $\cI_\mtt{L} \le \DL$: We only need to consider the case of $\DL < +\infty$.
    For any fixed $\xi \in \Ccinfty(\G:\R^d)$, we have
    \begin{equation}
        \label{I:D-1}
\begin{aligned}
    &-4\int_{\R^{4d}}\sqrt{\kappa A  F}\ \nabla_{0}\cdot(\Pi_{(v-v_*)^\perp}\xi)\dd \eta-2\int_{\R^{4d}} |\xi|^2\dd\eta\\
    =&{}4\int_{\R^{4d}}\big(\sqrt{\kappa}\ \Lgrad\sqrt{F}\ \big)\cdot \xi\dd \eta -2\int_{\R^{4d}}|\xi|^2\dd\eta\\
    \le&{} 4\Big(\int_{\R^{4d}}|\sqrt\kappa \ \Lgrad\sqrt{F}|^2\dd \eta\Big)^\frac12\cdot\Big(\int_{\R^{4d}}|\xi|^2\dd\eta\Big)^\frac12-2\int_{\R^{4d}}|\xi|^2\dd\eta\\
    \le&{} 2\int_{\R^{4d}} |\sqrt \kappa \ \widetilde\nabla \sqrt{F}|^2\dd\eta=\DL(f).
\end{aligned}
\end{equation}
In the first line, the integration by parts of $\sqrt{A}\ \nabla_0\cdot \Pi_{(v-v_*)^\perp}$ onto $\sqrt{F}$ is justified, since finite dissipation implies $\sqrt \kappa \widetilde\nabla\sqrt{F} \in L^2(\G)$. Moreover, we used the fact that $\Pi_{(v-v_*)^\perp} \nabla_0\sqrt{A}=0$. The last lines follow from the Cauchy-Schwarz and Young's inequalities.
    \item $\DL \le \cI_\mtt{L}$: We only need to consider the case of $\cI_{\mtt{L}} < +\infty$. We define the following linear functional 
    \begin{equation}
    \label{def:F}
        T(\xi)= -4\int_{\R^{4d}}\sqrt{\kappa A  F}\ \grad_0\cdot(\Pi_{(v-v_*)^\perp}\xi)\dd\eta,\quad \forall\,\xi\in C^\infty_c(\G;\R^d).
    \end{equation}
The boundedness of $\cI_\mtt{L}$ implies 
    \begin{equation*}
        \sup_{\xi \in \Ccinfty, \norm{ \xi}_{\Lp{}{2}}=1} T(\xi)\le \cI_\mtt{L}(f) +2< +\infty.
    \end{equation*}
Therefore, one can uniquely extend $F$ to $\Lp{}{2}(\R^{4d};\R^d)$. 
    Moreover, by the Riesz representation theorem, there exists a unique $u \in \Lp{}{2}(\R^{4d};\R^d)$ such that 
    \begin{equation*}
       T(\xi)=4\langle u,\xi\rangle,\quad \forall \xi \in\Lp{}{2}(\R^{4d};\R^d),
    \end{equation*}
    where $\langle \cdot,\cdot\rangle$ denote the $L^2$-inner product.
    Since $\widetilde\nabla \sqrt{\kappa F}\in L^2(\G;\R^d)$, the integration by parts of the right-hand side of \eqref{def:F} implies that  $u=\widetilde\nabla \sqrt{\kappa F}$.
Moreover, by \eqref{I:D-1}, we have 
\begin{equation*}
\label{ob-L-1}
T(\xi)-2\langle \xi,\xi\rangle\le \cD_{\aL}(f)=T(u)-2\langle u,u\rangle,\quad \forall \xi
\in\Lp{}{2}(\R^{4d};\R^d).
\end{equation*}
Since $\Ccinfty$ is dense in $\Lp{}{2}$, we have 
\begin{equation*}
\label{ob-L-2}
\cD_{\aL}(f)=\sup_{\xi \in \Ccinfty(\R^{4d};\R^d)}\{T(\xi)-2\langle \xi,\xi\rangle\}=\cI_\mtt{L}(f).
\end{equation*}
\end{itemize}
\end{proof}

By Proposition \ref{prop:aff}, the optimal $\xi$ in $L^2(\R^{4d};\R^d)$ such that $\cI_\mtt{L}(f)=T(\xi)-2\langle \xi,\xi\rangle$ is given by $u=\widetilde \nabla \sqrt{\kappa F}$. Since u is anti-symmetric, so that $u=-u_*$, and $u\perp (v-v_*)$, we can characterise $\DL(f)$ by
\begin{equation}
    \label{var-AS}
\begin{aligned}
     \DL(f)=\sup_{V\in \mathbf{AS}^{\infty}_c}\Big\{&-2\int_{\R^{4d}}|\Pi_{(v-v_*)^\perp}V|^2\deta-4\int_{\R^{4d}}\sqrt{\kappa F}\ \widetilde\nabla\cdot (\Pi_{(v-v_*)^\perp}V)\deta\Big\}.   
\end{aligned}
\end{equation}
In the above, the space $\mathbf{AS}^{\infty}_c$ is defined as
\begin{equation*}
    \label{def:as}
\begin{gathered}
    \mathbf{AS} = \{V\in L^2(\G;\R^d)\mid  V = - V_*\ \text{a.e.}\},\\
     \mathbf{AS}^{\infty}_c = \{V\in \Ccinfty(\R^{4d};\R^d)\mid V\in \mathbf{AS}\}.
\end{gathered}
\end{equation*}
Notice that $\big(\mathbf{AS},\langle\cdot,\cdot\rangle_{L^2}\big)$ is a closed subspace of $L^2(\R^{4d};\R^d)$, and $\mathbf{AS}^{\infty}_c$ is dense in $\mathbf{AS}$ with respect to $L^2$-norm.\\ \\
We define the functional spaces
\begin{equation}
\label{def:Sym}
    \mathbf{DS}^{\infty}_c = \{\Phi \in \Ccinfty(\R^{4d};\R)\mid \Phi = \Phi_* ,\,\exists \delta_\Phi>0\text{ s.t. } \Phi=0\,\forall |v-v_*|\le \delta_{\Phi}\}.
\end{equation}
The vector field $V\in \mathbf{AS}^{\infty}_c$ in \eqref{var-AS} can be replaced by $\widetilde\nabla \Phi \in \mathbf{DS}^{\infty}_c $ by directly applying the following solvability lemma of the Landau elliptic equation.
\begin{lemma}
\label{lem:elliptic}
Let $V\in \mathbf{AS}^{\infty}_c(\G;\R^d)$, there exists a unique solution $\Phi\in \mathbf{DS}^{\infty}_c(\G;\R)$ to the following equation
\begin{equation}
 \label{elliptic}   \widetilde\nabla\cdot\widetilde\nabla \Phi=\widetilde\nabla\cdot(\Pi_{(v-v_*)^\perp} V).
\end{equation}
Moreover, $\Phi$ satisfies
\begin{equation}
\label{L2}
\| \widetilde\nabla\Phi\|_{L^2(\G)}^2\le\| \Pi_{(v-v_*)^\perp} V\|_{L^2(\G)}^2.   
\end{equation}
\end{lemma}
\begin{proof}
We define the homogeneous Landau gradient $\widetilde\nabla_{\sf homo}$ as follows
\begin{equation}
    \label{def:homo:Lgrad}
    \widetilde\nabla_{\sf homo}\Phi=\sqrt A\Pi_{(v-v_*)^\perp}\nabla_0\Phi.
\end{equation}
We fix $x$ and $x_*$ as parameters,  by directly applying the solvability of the homogeneous (Landau)-elliptic equation given in \cite[Lemma 6.5]{carrillo2022boltzmann}, there exist unique $\Phi_1$ and $\Phi_2\in \mathbf{DS}^{\infty}_c(\Do;\R)$ that solve the following equations
\begin{align}
\widetilde\nabla_{\sf homo}\cdot\widetilde\nabla_{\sf homo} \Phi_1&=\widetilde\nabla_{\sf homo}\cdot(\Pi_{(v-v_*)^\perp} V),\label{two-eq-1}\\ \widetilde\nabla_{\sf homo}\cdot\widetilde\nabla_{\sf homo} \Phi_2&=\widetilde\nabla_{\sf homo}\cdot(\Pi_{(v-v_*)^\perp} V_*)\label{two-eq-2}    
\end{align}
for any fixed $(x,x_*)\in\Do$. We comment that the solvability result in \cite{carrillo2022boltzmann} works for both $A_0$ taking the form \eqref{A-1} and \eqref{A-3}.
Moreover, we have the following $L^2$-decomposition equalities for all $(x,x_*)\in\Do$
\begin{equation}
\label{L2:hmo}
\begin{gathered}
\| \Pi_{(v-v_*)^\perp} V\|_{L^2_{v,v_*}}^2=\| \widetilde\nabla_{\sf homo} \Phi_1\|_{L^2_{v,v_*}}^2+\| \Pi_{(v-v_*)^\perp} V-\widetilde\nabla_{\sf homo} \Phi_1\|_{L^2_{v,v_*}}^2,\\
\| \Pi_{(v-v_*)^\perp} V_*\|_{L^2_{v,v_*}}^2=\| \widetilde\nabla_{\sf homo} \Phi_2\|_{L^2_{v,v_*}}^2+\| \Pi_{(v-v_*)^\perp} V_*-\widetilde\nabla_{\sf homo} \Phi_2\|_{L^2_{v,v_*}}^2.
\end{gathered}
\end{equation}
Notice that \eqref{two-eq-2} holds if and only if 
\begin{align*}
  -\widetilde\nabla_{\sf homo}\cdot\widetilde\nabla_{\sf homo} (\Phi_2)_*&=\widetilde\nabla_{\sf homo}\cdot(\Pi_{(v-v_*)^\perp} V).
\end{align*}
The uniqueness of the homogeneous elliptic equations ensures that
\begin{align*}
\Phi_1 = -(\Phi_2)_*\quad \text{for a.e.}\quad  (x,x_*,v,v_*)\in \G.    
\end{align*}
Let $\Phi=2\Phi_1$. 
We have $\Phi\in \mathbf{DS}^{\infty}_c(\G;\R^d)$: The compacted support of $\Phi$ in $x$ and $x_*$ is ensured by the uniqueness of \eqref{two-eq-1} compact support of $V\in \mathbf{AS}^{\infty}_c(\G;\R^d)$; We have the smoothness in $x$ and $x_*$ by taking $\nabla_x$ and $\nabla_{x_*}$ of the both sides of \eqref{two-eq-1}. Thus, $\Phi$ is a desired solution of \eqref{elliptic}.
The $L^2$-inequality \eqref{L2} holds by integrating \eqref{L2:hmo} over $(x,x_*)\in \Do$.

The uniqueness of \eqref{elliptic} is ensured by the uniqueness of the homogeneous equation \eqref{two-eq-1}, since 
\begin{align*}
\widetilde\nabla\cdot\widetilde\nabla\Phi=0\quad\text{if and only if}\quad \nabla_0\cdot \big(\Pi_{(v-v_*)^\perp} \nabla_0\Phi\big)=0.
\end{align*}

\end{proof}
Let $\Phi\in \mathbf{DS}^{\infty}_c(\G;\R^d)$ be a solution of \eqref{elliptic}. 
We substitute $V=\Lgrad\Phi\in \mathbf{AS}^{\infty}_c(\Do;\R^d)$ to \eqref{var-AS}, the $L^2$-inequality \eqref{L2} implies that
    \begin{equation}
    \label{UARB:D-landau}
        \DL(f)\le\sup_{\Phi\in \mathbf{DS}^{\infty}_c}\Big\{-2\int_{\R^{4d}}|\widetilde\nabla \Phi|^2\deta -4\int_{\R^{4d}} \sqrt{\kappa F}\ \widetilde\nabla\cdot\widetilde\nabla\Phi\deta\Big\}.
    \end{equation}

\medskip

\noindent\textbf{Step 2: Affine representation of the fuzzy Boltzmann dissipation}\\ 

The following result extends \cite[Proposition 6.6]{carrillo2022boltzmann} to the fuzzy setting.

\begin{proposition}
We have the following preliminary expression for the reduced fuzzy Boltzmann dissipation
    \begin{equation}
    \label{AR:D-Boltzmann}
    \begin{gathered}
        \cD_{\cosh}(f) = \sup_{\xi \in L^2(B\dd \sigma\dd\eta)}\Big\{\int_{\R^{4d}\times S^{d-1}}\sqrt\kappa(\sqrt{F'}-\sqrt{F} ) \xi B \dsigma \dd \eta\\
        -\frac12\int_{\R^{4d}\times S^{d-1}}\abs{\xi}^2 B\dsigma\dd \eta 
        \Big\}.
    \end{gathered}
    \end{equation}
\end{proposition}
\begin{proof}
For any fixed $\xi \in L^2(B\dd \sigma\dd\eta)$, by the Cauchy--Schwarz inequality, we have
\begin{equation*}
        \label{I:D-2}
    \begin{aligned}        &\int_{\R^{4d}\times S^{d-1}}\sqrt\kappa\ (\sqrt{F'}-\sqrt{F} ) \xi B \dsigma \dd \eta - \frac12\int_{\R^{4d}\times S^{d-1}}\abs{\xi}^2 B\dsigma\dd \eta \\
         \le&{} \frac12\int_{\R^{4d}\times S^{d-1}} |\sqrt\kappa\ (\sqrt{F'}-\sqrt{F})|^2B\dd\sigma\dd \eta=\cD_{\cosh}(f).
    \end{aligned}
    \end{equation*}
The left-hand side attains its maximum value, equal to $\cD_{\cosh}(f)$, when $\xi = \tfrac{1}{2} \overline{\nabla}\sqrt{\kappa F}\in L^2(B\dd \sigma\dd\eta)$.
\end{proof}

In comparison to \eqref{UARB:D-landau}, we want to replace $\xi \in L^2(B\dd \sigma\dd\eta)$ in \eqref{AR:D-Boltzmann} by $\overline\nabla \Phi$ for $\Phi\in 
\mathbf{DS}^{\infty}_c(\G)$.
By Corollary \ref{corr:big:phi}, we have 
    \begin{equation*}
        |\overline\nabla\Phi|^2 \lesssim \mathbbm{1}_{\supp(\overline\nabla \Phi)}\theta^2|v-v_*|^2.
    \end{equation*}
The integrability of $\theta^2\beta(\theta)$  \eqref{def:beta-int} ensures that
\begin{align*}
 \|\overline\nabla \Phi\|_{L^2(B\dd\eta\dd\sigma)}^2
 \lesssim \big|\supp(\overline\nabla \Phi)\big|\int_0^{\frac{\pi}{2}}\theta^2\beta(\theta)\dd\theta<+\infty.
\end{align*}
 Moreover, by changing of variables, we have
\begin{align*}
\frac12\int_{\R^{4d}\times S^{d-1}}\sqrt\kappa(\sqrt{F'}-\sqrt{F} ) \overline\nabla\Phi B \dsigma \dd \eta
=-\int_{\R^{4d}\times S^{d-1}}\sqrt{\kappa F}  \overline\nabla\Phi B \dsigma \dd \eta.
    \end{align*}

We conclude that
\begin{equation*}
    \label{LARB:D-Boltzmann}
    \begin{aligned}
         \cD_{\cosh}(f) \ge  \sup_{\Phi \in \mathbf{DS}^{\infty}_c}\Big\{&-\frac18\int_{\R^{4d}\times S^{d-1}}|\Bgrad \Phi|^2 B\dsigma\deta\\
        & -\int_{\R^{4d}}\sqrt{\kappa F}\big(\int_{ S^{d-1}}\Bgrad \Phi B\dsigma\big)\deta \Big\}.
    \end{aligned}
    \end{equation*}

\medskip

\noindent\textbf{Step 3: Proof of Theorem \ref{main-thm-dissipation}.}

\begin{lemma}\label{lem:AR-conv}
    Let $\Phi \in \mathbf{DS}^{\infty}_c(\G)$. Let $B^\varepsilon$ and $\kappa^\varepsilon$  satisfy Assumption \ref{ASS:kappa}. Then, we have 
    \begin{align}
        \int_{\R^{4d}}\sqrt{\kappa^\eps F^\eps} \ \Big(\int_{ S^{d-1}}\overline\nabla\Phi B^\eps\dd\sigma\Big) \dd\eta&\limit{\eps \da 0} 2 \int_{\R^{4d}}\sqrt{\kappa F}\ \widetilde\nabla \cdot \widetilde\nabla\Phi \dd\eta,\label{AR:conv1}\\
        \int_{\R^{4d}\times S^{d-1}} |\overline\nabla\Phi|^2 B^\eps \dd\sigma\dd \eta&\limit{\eps \da 0} 8\int_{\R^{4d}}|\widetilde \nabla \Phi|^2 \dd\eta. \label{AR:conv2}
    \end{align}
\end{lemma}
\begin{proof}
We follow \cite[Section 6.3]{carrillo2022boltzmann} with appropriate modifications to the fuzzy cases.
Without loss of generality, we assume $|\overline\nabla \Phi|\le 1$ and $\supp(\overline\nabla \Phi)\subset \Omega\times S^{d-1}$, where $\Omega:=B_1(0)\setminus (\Do\times\{|v-v_*|\le \delta\})$ for some $\delta\in(0,1)$, where $B_1(0)\subset \G$ denotes the unit ball.   Notice that removing the region of $|v-v_*|\le \delta$ leads to the boundedness of $A_0(|v-v_*|)\mathbb{1}_\Omega\le C$. Let $\theta= \eps\chi/\pi$. We recall the finite angular momentum assumption \eqref{def:beta-int}.


We first show \eqref{AR:conv1} by the dominated convergence theorem.  The right-hand side of \eqref{AR:conv1} can be written as 
\begin{align*}
I_1:=\int_{\Omega}A_0\sqrt{\kappa^\eps F^\eps} \Big(\int_0^{\frac{\pi}{2}}\frac{\beta(\chi)}{\pi \varepsilon^2}\int_{ S^{d-2}_{k^\perp}}\overline\nabla\Phi \dd p\dd\chi \Big)\dd\eta.
\end{align*}
By using the upper bound of $\abs{\int_{ S^{d-2}_{k^\perp}}\overline\nabla \Phi \dd p}\lesssim \varepsilon^2\chi^2|v-v_*|$ in Corollary \ref{corr:big:phi}, the integrand of $I_1$ is dominated by 
\begin{align*}
 &A_0\sqrt{\kappa^\eps F^\eps} \Big(\int_0^{\frac{\pi}{2}}\frac{\beta(\chi)}{\pi \varepsilon^2}\int_{ S^{d-2}_{k^\perp}}\overline\nabla\Phi \dd p\dd\chi \Big)\\\lesssim_{\Phi,C_\kappa} &{}\sqrt{F^\varepsilon}\mathbb{1}_{\Omega}\int_0^{\frac{\pi}{2}}\chi^2\beta(\chi)\dd\chi\lesssim \sqrt{F^\varepsilon}\mathbb{1}_{\Omega},
\end{align*}
which is uniformly integrable since $\sqrt{F^\eps} \to \sqrt{F}$ in $L^2(\Omega)$
by Lemma \ref{lemma:cpt:sqrt}.

Then we are left to show the pointwise limit 
\begin{equation*}\label{Bgrad:conv}
   \int_{ S^{d-1}}\overline\nabla\Phi B^\eps\dd\sigma \limit{\eps\da0} 2\widetilde\nabla\cdot\widetilde\nabla\Phi,
\end{equation*}
which is a direct consequence of the pointwise limit \eqref{big:conv2} 
\begin{align*}
    &\int_{ S^{d-1}}\overline\nabla\Phi B^\eps\dd\sigma =  A_0(|v-v_*|)\int_0^{\frac{\pi}{2}}\frac{\beta(\chi)}{\pi\varepsilon^2}\int_{S^{d-2}_{k^\perp}}\overline\nabla\Phi\dd p\dd \chi\\
    \limit{\eps\da0}&{}A_0\frac{|S^{d-2}|}{8(d-1)}\int_0^{\frac\pi2}\chi^2\beta(\chi)\dd\chi \grad_0\cdot\big(|v-v_*|^2\Pi_{(v-v_*)^\perp}\grad_0(\Phi+\Phi_*)\big)\\ 
    =&{} 2 \widetilde\nabla\cdot\widetilde\nabla \Phi \quad \text{for a.e. } (x,x_*,v,v_*) \in \G.
\end{align*}
Then we obtain the limit \eqref{AR:conv1} by the dominated convergence theorem.

We show \eqref{AR:conv2} similarly. The right-hand side of \eqref{AR:conv2} can be written as 
\begin{align*}
I_2:=\int_{\Omega}A_0\Big(\int_0^{\frac{\pi}{2}}\frac{\beta(\chi)}{\pi \varepsilon^2}\int_{ S^{d-2}_{k^\perp}}|\overline\nabla\Phi|^2 \dd p\dd\chi \Big)\dd\eta,
\end{align*}
and the integrated of $I_2$ is dominated by 
\begin{align*}
 A_0\Big(\int_0^{\frac{\pi}{2}}\frac{\beta(\chi)}{\pi \varepsilon^2}\int_{ S^{d-2}_{k^\perp}}|\overline\nabla\Phi|^2 \dd p\dd\chi \Big)\lesssim_{\Phi,C_\kappa} \mathbb{1}_{\Omega}.
\end{align*}
On the other hand, the bound and  limit of $\overline\nabla\Phi$ given by \eqref{big:bd1} and \eqref{big:conv1} 
\begin{align*}
|\overline\nabla \Phi|^2\lesssim \varepsilon^2\chi^2|v-v_*|^2\quad\text{and} \quad \lim_{\varepsilon\to0}\frac{1}{\varepsilon }   \overline\nabla \Phi=\frac{\chi}{\pi}\frac{|v-v_*|}{2} p\cdot\nabla_0\big(\Phi+\Phi_*\big)
\end{align*}
yield the pointwise limit for a.e. $(x,x_*,v,v_*)\in\G$
\begin{equation}
\label{Bgard:8}
\lim_{\varepsilon\to0}\int_{ S^{d-1}}\abs{\Bgrad\Phi}^2B^\eps\dd\sigma=8|\Lgrad\Phi|^2.  
\end{equation}
Indeed, by the identity $\abs{p\cdot z}^2 = z^T (p\otimes p)z$ for all $z \in \R^d$, we have 
\begin{align*}
    &\int_{ S^{d-1}}\abs{\Bgrad\Phi}^2B^\eps\dd\sigma=A_0\int_0^\frac\pi2\beta(\chi)\int_{ S^{d-2}_{k^\perp}}\abs{\frac{\pi}{\varepsilon}\Bgrad\Phi}^2\dd p \dd \chi \\
    \limit{\eps \da 0}&{} \frac{A_0}{4}|v-v_*|^2\int_0^\frac\pi2 \beta(\chi)\chi^2\dd \chi\int_{ S^{d-2}_{k^\perp}}\abs{p\cdot\grad_0(\Phi+\Phi_*)}^2\dd p  \\
   =&{} \frac{2(d-1)A}{|S^{d-2}|}\big(\nabla_0(\Phi+\Phi_*)\big)^T\Big(\int_{S^{d-2}_{k^\perp}} p\otimes p\dd p \Big)\nabla_0(\Phi+\Phi_*) \\
    =&{}2A\big(\nabla_0(\Phi+\Phi_*)\big)^T\Pi_{(v-v_*)^\perp} \nabla_0(\Phi+\Phi_*)= 8|\Lgrad\Phi|^2,
\end{align*}
where we use the notation $A(|v-v_*|)=A_0(|v-v_*|)|v-v_*|^2$, and the second last equality is given by Lemma \ref{lem:pi}.

Hence, the limit \eqref{AR:conv2} holds by the dominated convergence theorem.
\end{proof}

\section{Convergence of curve action}\label{sec:action}

We recall the definition of $\aR$ and $\cA_{\aL}$
\begin{gather}
\aR(f,U)=\frac14\int_{\G\times S^{d-1}} \Psi\Big(\frac{U}{\Theta(f)B\kappa}\Big)\Theta(f)B\kappa\dd\eta\dd\sigma,  \label{sec-6:R}\\
\text{and}\quad \cA_{\aL}(f,U)=\frac12\int_{\G} \frac{|U|^2}{\kappa ff_*}\dd\eta.\notag
\end{gather}
In this Section, we show the following theorem.
\begin{theorem}\label{lem:sec-6}
    Let $\varepsilon\in(0,1)$. Let $B^\varepsilon$ and $\kappa^\varepsilon$ satisfy Assumption \ref{ASS:kappa}. Let $f^\varepsilon$ be $\cH$-solutions of \eqref{def:scaling} with initial values satisfying Assumption \ref{ass:curve}. Let $f$ be the weak limit of $f^\varepsilon$ given in Lemma \ref{weak:conv}. Let $(\Psi^*,\Theta)$ be an admissible pair satisfying Assumption \ref{ass-pair} and Assumption \ref{ass:Psi}. Let  \begin{align*}
U^\varepsilon_{\aB}=\kappa^\varepsilon B^\varepsilon\big((f^\varepsilon)'(f^\varepsilon_*)'-f^\varepsilon f^\varepsilon_*\big).
\end{align*} 
There exists $U:[0,T]\times \G\to\R^d$ such that
\begin{equation*}
\label{sec-5:ineq}
 (f,U)\in\TGRE_T\quad \text{and}\quad  \frac12\cA_{\aL}(f,U)\le \liminf_{\varepsilon\to0}  \aR(f^\varepsilon,U^\varepsilon_{\aB}).
    \end{equation*}
\end{theorem}

\medskip


We show Theorem \ref{lem:sec-6} in three steps:
\begin{enumerate}[Step $1$:]
    \item Let $\delta>0$ be a constant defined as in Lemma \ref{cpt:U} below. We show the weak precompactness of 
\begin{equation}
\label{def:U-epsilon-q-delta}
\big\{U^\varepsilon_{\delta}\big\}_{\varepsilon\in(0,1)},\quad  U^\varepsilon_{\delta}:=|v-v_*|(\langle v\rangle+\langle v_*\rangle)^{\delta}\theta U^\varepsilon_{\aB}
\end{equation}
in $L^1([0,T]\times \G\times S^{d-1})$. 
We denote (up to a subsequence) $U^{\delta}_{\aB} $ the weak limit of $U^\varepsilon_{\delta}$. Correspondingly, we define
\begin{equation}
\label{def:UL}
   U^{\delta}_{\aL}=\frac14|v-v_*|^{-1}A_0^{-\frac12}(\langle v\rangle+\langle v_*\rangle)^{-\delta}\int_{S^{d-1}} U^{\delta}_{\aB}p\dd\sigma.
\end{equation}
\item We verify that $(f,U^\delta_{\aL})\in\TGRE_T$ and $\int_0^T\cA_{\aL}(f,U^\delta_{\aL})\dd t<+\infty$. Then we search for  $\widetilde U^\delta_{\aL}$ such that
\begin{equation*}
\int_0^T\cA_{\aL}(f,\widetilde U^\delta_{\aL})\dd t=\min_{\widetilde\nabla\cdot W=0} \Big\{\int_0^T\cA_{\aL}(f,U^\delta_{\aL}+W)\dd t\Big\}. 
\end{equation*}
Moreover, we show $(f,\widetilde U^\delta_{\aL})\in\TGRE_T$.

\item We show that $\frac12\cA_{\aL}(f,\widetilde U^\delta_{\aL})\le \liminf_{\varepsilon\to0}  \aR(f^\varepsilon,U^\varepsilon_{\aB})$.
\end{enumerate}


\medskip

\noindent\textbf{Step $1$: The compactness of $U^\varepsilon_{\delta}$.}

\begin{lemma}
\label{cpt:U}

In the hard and Maxwellian potential case \eqref{A-hard}, we take $\delta\in [-1-|\gamma|/2,0)$. In the soft potential case \eqref{A-soft}, we take $\delta\in[0,|\gamma|/2)$.
Let $U^\varepsilon_\delta $ be defined as in \eqref{def:U-epsilon-q-delta}
\begin{equation*}
U^\varepsilon_{\delta}=|v-v_*|(\langle v\rangle+\langle v_*\rangle)^{\delta}\theta U^\varepsilon_{\aB}.
\end{equation*}
Under the same assumptions as Theorem \ref{lem:sec-6}, we have the weak precompactness of
\begin{equation*}
\label{cpt:U-v}
\big\{U^\varepsilon_\delta\big\}_{\varepsilon\in(0,1)}\quad\text{in}\quad L^1([0,T]\times \G\times S^{d-1}).
\end{equation*}
    
\end{lemma}

\begin{proof}
   We first show the uniform bounds of $U^\varepsilon_\delta$. 
For any $\phi\in L^\infty([0,T]\times \Do)$, by Cauchy--Schwarz inequality, 
we have
    \begin{gather*}
       \Big| \int_0^T\int_{\G\times S^{d-1}}\phi U^\varepsilon_\delta\dd\eta\dd\sigma\dd t\Big|^2
       \le \Big(\int_0^T\cD_{\aB}(f^\varepsilon)\dd t\Big)^{\frac12}\\
       \times\Big(\int_0^T\int_{\GB}|v-v_*|^{2} (\langle v\rangle+\langle v_*\rangle)^{2\delta}\theta^2\kappa^\varepsilon B^\varepsilon\Lambda(f^\varepsilon)\phi^2\dd\eta\dd\sigma\dd t\Big)^{\frac12}.
    \end{gather*}
   By using of the inequality  $\Lambda(s,t)\le\frac{s+t}{2}$ for all $s,t>0$, we have 
    \begin{align*}
        &\int_{\GB}|v-v_*|^{2} (\langle v\rangle+\langle v_*\rangle)^{2\delta}\theta^2\kappa^\varepsilon B^\varepsilon\Lambda(f^\varepsilon)\phi^2\dd\eta\dd\sigma\\
        \le&{}\int_{\G}\int^{\frac{\varepsilon}{2}}_0\int_{S^{d-2}_{k^\perp}}|v-v_*|^{2}(\langle v\rangle+\langle v_*\rangle)^{2\delta}A_0\theta^2\kappa^\varepsilon \beta^\varepsilon(\theta)\phi^2\frac{f^\varepsilon f^\varepsilon_*+(f^\varepsilon)'(f^\varepsilon_*)'}{2}\dd p\dd \theta\dd\eta\\
        \lesssim &{}\int_{\G}\int^{\frac{\varepsilon}{2}}_0\int_{S^{d-2}_{k^\perp}}|v-v_*|^{2}(\langle v\rangle+\langle v_*\rangle)^{2\delta}A_0 \theta^2\kappa^\varepsilon \beta^\varepsilon(\theta)\big(\phi^2+(\phi')^2\big) f^\varepsilon f^\varepsilon_*\dd p\dd \theta\dd\eta,
        \end{align*}
where we change the variable $(v,v_*)\mapsto (v',v_*')$ and use the energy conservation law $|v|^2+|v_*|^2=|v'|^2+|v_*'|^2$ in the last step. The finite angular momentum assumption \eqref{def:beta-epsilon} and uniform bound of $\kappa^\varepsilon$ imply that 
        \begin{align*}
 &\int_{\G}\int^{\frac{\varepsilon}{2}}_0\int_{S^{d-2}_{k^\perp}}|v-v_*|^{2}(\langle v\rangle+\langle v_*\rangle)^{2\delta}A_0 \kappa^\varepsilon \theta^2\beta^\varepsilon(\theta)\big(\phi^2+(\phi')^2\big) f^\varepsilon f^\varepsilon_*\dd p\dd \theta\dd\eta\\
  \lesssim&{}\|\phi\|_{L^\infty}^2\int_{\G}|v-v_*|^{2}(\langle v\rangle+\langle v_*\rangle)^{2\delta}A_0  f^\varepsilon f^\varepsilon_*\dd\eta.
    \end{align*}
In the hard potential case,
we have 
\begin{align*}
    \int_{\G}|v-v_*|^2(\langle v\rangle+\langle v_*\rangle)^{2\delta}A_0f^\varepsilon f^\varepsilon_*\dd\eta\lesssim \|f^\varepsilon\|_{L^1_{0,s_1}(\Do)}^2,
\end{align*}
where $s_1=2+\gamma_++2\delta\in [0,2+\gamma_+)$. 
By uniform bounds \eqref{uni-bdd:D}, the right-hand side is uniformly bounded. Moreover, the weak compactness of $\{f^\varepsilon f^\varepsilon_*\}$ in $L^1_\loc (\G)$ and $s_1<2+\gamma_+$ ensure the weak compactness of $\big\{\big(\langle v\rangle+\langle v_*\rangle\big)^{s_1}f^\varepsilon f^\varepsilon_*\big\}$ and $U^\varepsilon_\delta$ by Dunford--Petitts theorem.

Similarly, in the case of $\gamma\in[-2,0)$, we have the uniform boundedness and compactness of  
\begin{align*}
    \int_{\G}|v-v_*|^2A_0(\langle v\rangle+\langle v_*\rangle)^{2\delta}  f^\varepsilon f^\varepsilon_*\dd\eta\lesssim \|f^\varepsilon\|_{L^1_{0,s_2}(\Do)}^2,
\end{align*}
where $s_2=2+\gamma+2\delta\in[0,2)$.
\end{proof}

\medskip

\noindent\textbf{Step $2$: The verification of $(f,\widetilde U^\delta_{\aL})\in\TGRE_T$.}

In the following, we fix a $\delta$  given in Lemma \ref{cpt:U}.
We denote (up to a subsequence) $U^{\delta}_{\aB} $ the weak limit of $U^\varepsilon_{\delta}$ in $L^1([0,T]\times\G\times S^{d-1})$. 
We recall the definition \eqref{def:UL}
\begin{equation}
\label{def:U-delta-L}
    U^{\delta}_{\aL}=\frac14|v-v_*|^{-1}A_0^{-\frac12}(\langle v\rangle+\langle v_*\rangle)^{-\delta}\int_{S^{d-1}} U^{\delta}_{\aB}p\dd\sigma,
\end{equation}
where $p\in S^{d-2}_{k^\perp}$.

We first verify that $(f,U^\delta_{\aL})\in \TGRE_T$. The weak continuity in time of $f$ is given by Lemma \ref{weak:conv}. We show that $(f,U^\delta_{\aL})$ is a weak solution of the transport grazing rate equation
\begin{equation}
\label{sec-5:T-eq}
    \d_tf +v\cdot\nabla_x f+\frac12 \widetilde\nabla\cdot U_{\aL}=0.
\end{equation}
We recall the weak formulation for the fuzzy Boltzmann equation 
\begin{align*}
&\int_0^T\int_{\Do}(\d_t+v\cdot\nabla_x)\phi f^\varepsilon\dd x\dd v\dd t-\int_{\Do}\phi_0f_0\dd x\dd v\\
&+\frac14\int_0^T\int_{\G\times S^{d-1}}  \overline\nabla\phi U_{\aB}^\varepsilon\dd\eta\dd\sigma\dd t=0,
\end{align*}
and the transport grazing rate equation for the fuzzy Landau equation
\begin{equation*}
    \label{weak-sec-5-TG}
\begin{aligned}
&\int_0^T\int_{\Do}(\d_t+v\cdot\nabla_x)\phi f^\varepsilon\dd x\dd v-\int_{\Do}\phi_0f_0\dd x\dd v\\
&+\frac12\int_0^T\int_{\G}  \widetilde\nabla\phi\cdot U_{\aL}\dd\eta\dd t=0
\end{aligned}
\end{equation*}
for all $\phi\in C^\infty_c([0,T)\times\Do)$.
Since in Lemma \ref{weak:conv} we showed that $f^\varepsilon\rightharpoonup f$ in $L^1([0,T]\times \Do)$, to show  $\big(f,U^{\delta}_{\aL}\big)$ is a weak solution of \eqref{sec-5:T-eq}, we only need to show that 
\begin{equation}
\label{sec-5:goal-1}
\lim_{\varepsilon\to0} \int_0^T\int_{\G\times S^{d-1}}  \overline\nabla\phi U_{\aB}^\varepsilon\dd\eta \dd\sigma \dd t= 2\int_0^T\int_{\G}  \widetilde\nabla\phi\cdot U^\delta_{\aL}\dd\eta\dd t.
\end{equation}
By definition \eqref{def:U-epsilon-q-delta}, we have  
\begin{align*}
&\int_0^T\int_{\G\times S^{d-1}}  \overline\nabla\phi U_{\aB}^\varepsilon \dd\eta\dd\sigma
=\int_0^T\int_{\G\times S^{d-1}}   \frac{\overline\nabla\phi}{\theta |v-v_*| (\langle v\rangle+\langle v_*\rangle)^{\delta}} U_{\delta}^\varepsilon\dd\eta\dd\sigma.
\end{align*}

On the one hand, Lemma \ref{cpt:U} ensures $U_\delta^\varepsilon \rightharpoonup U^\delta_{\aB}$ in $L^1([0,T]\times \G\times S^{d-1})$.
On the other hand, we have $\frac{\overline\nabla\phi}{\theta |v-v_*| (\langle v\rangle+\langle v_*\rangle)^{\delta}}$ is uniformly bounded and convergent in measure as $\theta\to0$. Indeed, we have $(\langle v\rangle+\langle v_*\rangle)^{-\delta}\le C$ for all $v,v_*\in \supp(\overline\nabla\phi)$. Lemma \ref{lem:phi:conv} implies that 
\begin{align*}
   \frac{|\overline\nabla\phi|}{\theta |v-v_*| } \lesssim  \Lip(\phi),
\end{align*}
and the following pointwise convergence 
\begin{equation*}
\label{conv:pw}
\begin{aligned}
\frac{\overline\nabla\phi}{\theta |v-v_*| }
\to&\frac12 p\cdot (\nabla_v\phi-\nabla_{v_*}\phi_*)\\
=&\frac12 |v-v_*|^{-1}A_0^{-\frac12}p\cdot \widetilde\nabla\phi\quad\text{as }\theta\to0.
\end{aligned}   
\end{equation*}
In the last step,  we use the definition of $k=\frac{v-v_*}{|v-v_*|}$, $p\in S^{d-2}_{k^\perp}$ and $\widetilde\nabla\phi=|v-v_*|A_0(|v-v_*|) ^\frac12\Pi_{k^\perp}(\nabla_v\phi-\nabla_{v_*}\phi_*)$.

Then by weak-strong convergence, we obtain \eqref{sec-5:goal-1} that
\begin{align*}
&\int_0^T\int_{\G\times S^{d-1}}  \overline\nabla\phi U_{\aB}^\varepsilon \dd\eta\dd\sigma\\
\to &{}\frac12\int_0^T\int_{\G} \widetilde\nabla\phi\cdot    |v-v_*|^{-1}(\langle v\rangle+\langle v_*\rangle)^{-\delta}A_0^{-\frac12} \int_{S^{d-1}}  U^\delta_{\aB}p\dd\sigma\dd\eta\\
=&{}2\int_0^T\int_{\G} \widetilde\nabla\phi\cdot U^\delta_{\aL}\dd\eta\quad\text{as }\varepsilon\to0.
\end{align*}
Hence, the pair $\big(f,U^\delta_{\aL}\big)$ is a weak solution of \eqref{sec-5:T-eq}. 

\medskip

{
To show  $\big(f,U^\delta_{\aL}\big)\in\TGRE_T$, we left to show $U^\delta_{\aL}\in L^1(\G)$.
We will first show
\begin{equation}
\label{AL:bdd}
    \int_0^T\cA_{\aL}(f,U^\delta_{\aL})\dd t<+\infty.
\end{equation}
To this end, we use the elementary result that for functions $a_n\rightharpoonup a$ and $b_n\rightharpoonup b$ in $L^1$, we have
\begin{equation}
\label{a-b}
    \int\frac{|a|^2}{b}\le\sup_n\int\frac{|a_n|^2}{b_n}.
\end{equation}
Indeed, by affine representation $\frac{|a|^2}{b}=\sup_{\phi}\{2a\phi-b\phi^2\}$, we have 
\begin{align*}
 \int\frac{|a|^2}{b}&=\sup_{\phi\in C^\infty_c}\Big\{2\int a\phi-\int b\phi^2\Big\}= \sup_{\phi}\lim_n\Big\{2\int a_n\phi-\int b_n\phi^2\Big\}\\
 &\le \sup_{n,\phi}\Big\{2\int a_n\phi-\int b_n\phi^2\Big\}\le\sup_n\int\frac{|a_n|^2}{b_n}.
\end{align*}
By definition of $U^\delta_{\aL}$ in \eqref{def:U-delta-L}, the weak convergence of $U^\varepsilon_\delta$ in Lemma \ref{cpt:U} and the convergence $\kappa^\varepsilon f^\varepsilon f^\varepsilon_*\to \kappa ff_*$ in $L^1(\G)$, we apply \eqref{a-b} to derive
\begin{align*} 
&\int_0^T\cA_L(f,{U_{\aL}^\delta})\dd t
=\frac12\int_0^T\int_{\G}\frac{| U^{\delta}_{\aL}|^2}{ff_*\kappa}\dd\eta \dd t\\
\lesssim&{} \sup_{\varepsilon}\int_0^T\int_{\G}\frac{|\int_{ S^{d-1}}\theta U^\varepsilon_{\aB}|^2}{A_0 \kappa^\varepsilon f^\varepsilon f^\varepsilon_*}\dd\eta \dd t.
\end{align*}
By Cauchy–-Schwarz inequality, we have 
\begin{align*}
    &\Big|\int_{S^{d-1}} \theta U^\varepsilon_{\aB}\dd\sigma\Big|^2\\
    \le &{}|S^{d-2}|\Big(\int_0^{\frac{\varepsilon}{2}}\theta^2\beta^\varepsilon(\theta)\dd\theta\Big)\int_{S^{d-1}}  A_0 B^\varepsilon \big|\kappa^\varepsilon\big((f^\varepsilon)_*'(f^\varepsilon)'-f^\varepsilon_*f^\varepsilon\big)\big|^2\dd\sigma\\
     \lesssim &{}A_0\int_{S^{d-1}} B^\varepsilon \big|\kappa^\varepsilon\big((f^\varepsilon)_*'(f^\varepsilon)'-f^\varepsilon_*f^\varepsilon\big)\big|^2\dd\sigma,
\end{align*}
where we use the angular momentum assumption \eqref{def:beta-int} for the last step.
We aim to show the uniform bound of 
\begin{align*}
\sup_{\varepsilon}\int_0^TD^\varepsilon,\quad D^\varepsilon:=\int_{\G\times S^{d-1}}\frac{B^\varepsilon \big|\kappa^\varepsilon\big((f^\varepsilon)_*'(f^\varepsilon)'-f^\varepsilon_*f^\varepsilon\big)\big|^2}{\kappa^\varepsilon f^\varepsilon f^\varepsilon_*}\dd\sigma\dd\eta.
\end{align*}
By affine representation, we have 
\begin{align*}
D^\varepsilon=\sup_{\xi}\Big\{\int_{\G\times S^{d-1}} 2\sqrt{B^\varepsilon}\kappa^\varepsilon\big((f^\varepsilon)_*'(f^\varepsilon)'-f^\varepsilon_*f^\varepsilon\big)\xi- \int_{\G\times S^{d-1}} |\xi|^2 \kappa^\varepsilon f^\varepsilon f^\varepsilon_*\Big\},
\end{align*}
where we take $\xi\in L^2(\kappa^\varepsilon f^\varepsilon f^\varepsilon_*\dd\eta\dd\sigma)\cap L^\infty(\G\times S^{d-1})$ such that $\|\xi\|_{L^\infty}=1$.
Since in Section \ref{sec:conv} we showed that $\frac{(f^\varepsilon)_*'(f^\varepsilon)'+f^\varepsilon_*f^\varepsilon}{2}\to ff_*$ in $L^1(\G\times S^{d-1})$, we have the following uniform bound
\begin{align*}
   \sup_{\varepsilon} \int_{\G\times S^{d-1}} |\xi|^2 \kappa^\varepsilon \Big|f^\varepsilon f^\varepsilon_*-\frac{(f^\varepsilon)_*'(f^\varepsilon)'+f^\varepsilon_*f^\varepsilon}{2}\Big|\le C
\end{align*}
for some constant $C>0$ independent of $\varepsilon$. 
By using the bound of the logarithm mean $\Lambda(s,t)\le \frac{s+t}{2}$ for $s,t>0$, we have 
\begin{align*}
D^\varepsilon-C\le \sup_{\xi}\Big\{\int_{\G\times S^{d-1}} 2\sqrt{B^\varepsilon}\kappa^\varepsilon\big((f^\varepsilon)_*'(f^\varepsilon)'-f^\varepsilon_*f^\varepsilon\big)\xi- \int_{\G\times S^{d-1}} |\xi|^2 \kappa^\varepsilon \Lambda(f^\varepsilon)\Big\}.
\end{align*}
Since $L^2(\kappa^\varepsilon f^\varepsilon_*f^\varepsilon)=L^2(\kappa^\varepsilon \Lambda(f^\varepsilon))$, the right-hand side of the above inequality can be seen as the affine representation of
\begin{align*}
     \cD_{\aB}^\varepsilon(f^\varepsilon)=\int_{\G\times S^{d-1}}\frac{B^\varepsilon\big|\kappa^\varepsilon\big((f^\varepsilon)_*'(f^\varepsilon)'-f^\varepsilon_*f^\varepsilon\big)\big|^2}{\kappa^\varepsilon\Lambda(f^\varepsilon)}\dd\sigma\dd\eta.
\end{align*}
The uniform bound of dissipation \eqref{uni-bdd:D} implies that
\begin{align*}
\sup_{\varepsilon}\int_0^T D^\varepsilon \le \int_0^T \cD_{\aB}^\varepsilon(f^\varepsilon)+CT <+\infty. 
\end{align*}
Thus, \eqref{AL:bdd} holds.
We show $U^\delta_{\aL}\in L^1([0,T]\times \G)$. By the Cauchy--Schwarz inequality, we have 
\begin{equation}
    \label{C-S:U}
\begin{aligned}
\int_0^T\int_{\G} |U| &\le  \sqrt{\int_0^T\int_{\G} ff_*\kappa}
\sqrt{\int_0^T\cA_{\aL}(f,U^\delta_{\aL})} \\
&\le \sqrt{T\|\kappa\|_{L^\infty}}\|f\|_{L^1(\Do)} \sqrt{\int_0^T\cA_{\aL}(f,U)}<+\infty.
\end{aligned}
\end{equation}
We conclude that $\big(f,U^\delta_{\aL}\big)\in\TGRE_T$ and $ \int_0^T\cA_{\aL}(f,U^\delta_{\aL})\dd t<+\infty$.
}

\medskip

For any $W\in L^1(\G;\R^d)$ such that $\widetilde\nabla\cdot W=0$ in the weak sense, i.e.
\begin{equation}
\label{div-free:tilde}
    \int_{\G} \widetilde\nabla \phi \cdot W\dd\eta=0,\quad \forall\, \phi\in C^\infty_c(\Do),
\end{equation}
we have $\big(f,U^\delta_{\aL}+W\big)\in \TGRE_T$ if $\cA_{\aL}(f,U^\delta_{\aL}+W)\in L^1([0,T])$. To verify that $\big(f,U^\delta_{\aL}+W\big)$ is a weak solution of \eqref{sec-5:T-eq}, we use 
\begin{equation*}
  \int_{\G} U^\delta_{\aL}\cdot \widetilde\nabla\phi\dd\eta=\int_{\G} (U^\delta_{\aL}+W)\cdot \widetilde\nabla\phi\dd\eta.
\end{equation*}

The rest of Step 2 is devoted to searching for the unique $\widetilde U^\delta_{\aL}$ such that
\begin{equation*}
\int_0^T\cA_{\aL}(f,\widetilde U^\delta_{\aL})\dd t=\min_{\widetilde\nabla\cdot W=0} \Big\{\int_0^T\cA_{\aL}(f,U^\delta_{\aL}+W)\dd t\Big\}. 
\end{equation*}
\begin{lemma}
\label{def:bbL}
Let $\big(f,U\big)\in \TGRE_T$ such that $\int_0^T\cA_{\aL}(f,U)\dd t<+\infty$.
Let $W\in L^1(\G;\R^d)$ satisfying \eqref{div-free:tilde}. We have (up to a subsequence)
\begin{align*}
U+W_n\rightharpoonup \widetilde U\quad\text{in } L^1([0,T]\times \G;\R^d)    
\end{align*}
such that
\begin{equation}
\label{opt-condition}
\int_0^T\cA_{\aL}(f,\widetilde U)\dd t=\min_{W\in L^1(\G;\R^d)} \Big\{\int_0^T\cA_{\aL}(f,U+W)\dd t \Bigl\vert \widetilde\nabla\cdot W=0 \text{ weakly}\Big\}. 
\end{equation}
Moreover, $\widetilde U$ is unique.
\end{lemma}
\begin{proof}
We define $\cA_0=\min \big\{\|\cA_{\aL}(f,U+W)\|_{L^1([0,T])}\big\}$.
There exists a sequence of divergence-free $\{W_n\}\subset L^1(\G;\R^d)$ such that $\int_0^T\cA_{\aL}(f,U+W_n)\dd t\le \int_0^T\cA_{\aL}(f,U)\dd t$ and 
\begin{align*}
\lim_{n\to\infty}\int_0^T\cA_{\aL}(f,U+W_n)\dd t=\cA_0.
\end{align*}
The sequence $\{U+W_n\}$ is weakly compact in $L^1([0,T]\times\G)$, for which we use the uniform bound \eqref{uni-bdd:D} and \eqref{C-S:U} that
\begin{align*}
\int_0^T\int_{\G} |U+W_n| \dd\eta\dd t\lesssim \|f\|_{L^1(\Do)} \Big(\int_0^T\cA_{\aL}(f,U)\dd t\Big)^{\frac12}.
\end{align*}

The lower semi-continuity and convexity of $\cA_{\aL}$ ensures that 
\begin{align*}
\cA_{\aL}(f,\widetilde U)\le \liminf_{n\to\infty} \cA_{\aL}(f,U+W_n)=\cA_0,   
\end{align*}
and the uniqueness of $\widetilde U$.
\end{proof}


We have the following lemma to characterise the optimal $\widetilde U$.
\begin{lemma}\label{lem:decompo}
  Let $\big(f,U\big)\in \TGRE_T$ such that $
    \int_0^T\cA_{\aL}(f,U^\delta)\dd t<+\infty$. Let $\widetilde U$ satisfy \eqref{opt-condition}. There exists $M\in L^1([0,T]\times\G;\R^d)$ such that
    \begin{align*}
    \widetilde U=\widetilde Mff_*\kappa\quad\text{and}\quad  
\widetilde M\in\overline{\{ \widetilde \nabla\phi\mid \phi\in C^\infty_c(\Do)\}}^{L^2(ff_* \kappa\dd\eta)}=:T_f.    
    \end{align*}

   
\end{lemma}
\begin{proof}
We follow \cite[Proposition A.11]{erbar2023gradient} with appropriate adaptations concerning the fuzzy Landau gradient.
By \cite[Lemma 3.3]{DH25}, there exists $\widetilde M\in L^1_\loc(\G)$ such that
\begin{align*}
    \widetilde U=ff_*\kappa \widetilde M.
\end{align*}
By definition of $\widetilde U$, we have, for all $s\in \R$,
\begin{align*}
\|\widetilde M\|_{L^2(ff_* \kappa\dd \eta)}\le \|\widetilde M+s\widetilde N\|_{L^2(ff_* \kappa\dd\eta)},\quad\forall  \widetilde N\in N_f.
\end{align*}
The space $N_f$ is defined as
\begin{align*}
N_f=\Big\{\widetilde N\in L^2(ff_* \kappa\dd \eta)\mid \int_{\G}\widetilde\nabla \phi ff_*\kappa \widetilde N\dd\eta=0,\,\forall \phi\in C^\infty_c(\Do) \Big\}.   
\end{align*}
Hence, $N_f$ is the orthogonal complement of $T_f$ in $L^2(ff_* \kappa\dd \eta)$, and we have $\widetilde M\in T_f$.

\end{proof}

We recall that $(f,U^\delta_{\aL})\in\TGRE_T$, where $U^\delta_{\aL}$ is defined as in \eqref{def:U-delta-L}. {By \eqref{AL:bdd}, we have  $
    \int_0^T\cA_{\aL}(f,U^\delta)\dd t<+\infty$.} By Lemma \ref{def:bbL}, we define $\widetilde{U^\delta_{\aL}}$ such that
    \begin{equation}
    \label{def:U-tilde-delta}
\int_0^T\cA_{\aL}(f,\widetilde{U^\delta_{\aL}})\dd t=\min_{\widetilde\nabla\cdot W=0} \Big\{\int_0^T\cA_{\aL}(f,U^\delta_{\aL}+W)\dd t \Big\}.
\end{equation}
Notice that we have $\big(f,\widetilde{U^\delta_{\aL}}\big)\in \TGRE_T$. By definition of $\widetilde U$ in Lemma \ref{def:bbL}, there exists $\{W_n\}\subset L^1(\G;\R^d)$ such that $\widetilde\nabla\cdot W_n=0$ and 
\begin{align*}
 U^\delta_{\aL}+W_n\rightharpoonup\widetilde U_{\aL}^\delta \quad \text{in}\quad L^1(\G;\R^d).   
 \end{align*}
The weak formulation \eqref{sec-5:T-eq} holds, since, for all $\phi\in C^\infty_c(\Do)$, we have 
\begin{equation}
\label{tilde:Tg}
  \int_{\G} U^\delta_{\aL}\cdot \widetilde\nabla\phi\dd\eta=\int_{\G} (U^\delta_{\aL}+W_n)\cdot \widetilde\nabla\phi\dd\eta\to\int_{\G} \widetilde{U^\delta_{\aL}}\cdot \widetilde\nabla\phi\dd\eta.
\end{equation}
Hence, we have $\big(f,\widetilde{U^\delta_{\aL}}\big)\in\TGRE_T$.

\medskip

\noindent\textbf{Step $3$: Verification of the curve action limit.}

In Step $3$, we show 
\begin{equation}
\label{gobal:step-3}
 \frac12\cA_{\aL}(f,\widetilde{U_{\aL}^\delta})\le \liminf_{\varepsilon\to0}  \aR(f^\varepsilon,U^\varepsilon_{\aB}).
\end{equation}

Similar to Proposition \ref{prop:aff}, we have the following affine representation of $\aR$ and $\cA_{\aL}$.
\begin{proposition}
    [Affine representation]\label{sec-5:prop}
   Let $U^\varepsilon_{\aB}=\kappa^\varepsilon B^\varepsilon\big((f^\varepsilon)'(f^\varepsilon_*)'-f^\varepsilon f^\varepsilon_*\big)$.  Let $\widetilde{U_{\aL}^\delta}$ satisfy \eqref{def:U-tilde-delta}.
Under the same assumptions as in Theorem \ref{lem:sec-6}, we have the following affine representation for the Boltzmann and Landau curve actions
    \begin{gather}
        \aR^\varepsilon(f^\varepsilon,U^\varepsilon_{\aB})=\frac14\sup_{\phi\in C^\infty_c(\Do)}\Big\{\int_{\GB}U^\varepsilon_{\aB}\overline\nabla \phi-\int_{\GB}\Psi^*(\overline\nabla\phi)B^\varepsilon \Theta(f^\varepsilon)\kappa^\varepsilon\Big\},\label{sec-5:aff-1}
   \\\cA_L(f,\widetilde{U_{\aL}^\delta})=\sup_{\phi\in C^\infty_c(\Do)}\Big\{  \int_{\Do}\widetilde{U_{\aL}^\delta}\cdot \widetilde\nabla \phi-\frac12\int_{\G}|\widetilde\nabla\phi|^2 ff_* \kappa\Big\}.\label{sec-5:aff-2}
    \end{gather}
\end{proposition}
\begin{proof}
By convex duality, we have \begin{align*}
\aR(f^\varepsilon,U^\varepsilon_{\aB})=\frac14\sup_{\xi}\Big\{\int_{\GB}U_{\aB}^\varepsilon\xi-\int_{\GB}\Psi^*(\xi)B^\varepsilon \Theta(f^\varepsilon)\kappa^\varepsilon \Big\},
\end{align*}
where $\xi\in L^2(\Theta B \kappa\dd\eta\dd\sigma;\R)$.

On the other hand, since $U_{\aB}^\varepsilon=\kappa^\varepsilon B^\varepsilon (\Psi^*)'(\overline\nabla \log f^\varepsilon)\Theta(f^\varepsilon)=\kappa^\varepsilon B^\varepsilon \big((f^\varepsilon)'(f^\varepsilon)_*'-f^\varepsilon f^\varepsilon_*\big)$, the duality inequality becomes an equality, i.e.
\begin{align*}
 \Psi\Big(\frac{U^\varepsilon_{\aB}}{\Theta(f^\varepsilon)B^\varepsilon\kappa^\varepsilon}\Big)\Theta(f)B\kappa=U_{\aB}^\varepsilon\xi-\Psi^*(\xi)B^\varepsilon \Theta(f^\varepsilon)\kappa^\varepsilon,\quad \xi=\overline\nabla \log f^\varepsilon.
\end{align*} 
The entropy dissipation bound \eqref{bdd:D-Psi} ensures the affine representation \eqref{sec-5:aff-1} holds.

Concerning the Landau affine representation, we have
\begin{align*} \cA_L(f,\widetilde{U_{\aL}^\delta})=\sup_{\zeta\in L^2(ff_*\kappa\dd\eta;\R^d)}\Big\{  \int_{\Do}\widetilde{U_{\aL}^\delta}\cdot \zeta -\int_{\G}\frac12|\zeta|^2 ff_* \kappa\Big\}.
\end{align*}
By Lemma \ref{lem:decompo}, we have
\begin{align*}
\widetilde{U_{\aL}^\delta}=Mff_*\kappa\quad\text{and}\quad  
M\in\overline{\{ \widetilde \nabla\phi\mid \phi\in C^\infty_c(\Do)\}}^{L^2(ff_* \kappa\dd\eta)}.    
    \end{align*}
    Hence, there exists $\{\phi_n\}\subset C^\infty_c(\Do)$, such that
\begin{align*}
\widetilde \nabla \phi_nff_*\kappa\cdot \widetilde \nabla \phi_n-\frac12 |\widetilde \nabla \phi_n|^2ff_*\kappa \to \frac12 \frac{|\widetilde{U_{\aL}^\delta}|^2}{ff_*\kappa} \quad\text{in}\quad  L^1(\G).
\end{align*}
Hence, we have the Landau affine representation \eqref{sec-5:aff-2}.
\end{proof}


To show \eqref{gobal:step-3}, we show the following lemma.
\begin{lemma}
\label{lem:6-6}
Let $\phi\in C^\infty_c(\Do)$. 
Let $U^\varepsilon_{\aB}=\kappa^\varepsilon B^\varepsilon\big((f^\varepsilon)'(f^\varepsilon_*)'-f^\varepsilon f^\varepsilon_*\big)$.  Let $\widetilde{U_{\aL}^\delta}$ satisfy \eqref{def:U-tilde-delta}.
Under the same assumptions as in Theorem \ref{lem:sec-6}, we have 
\begin{gather}
\frac14 \int_{\G\times S^{d-1}}U^\varepsilon_{\aB}\overline\nabla\phi\dd\eta\dd \sigma \to   \frac12\int_{\G} \widetilde{U^\delta_{\aL}}\cdot \widetilde\nabla\phi\dd\eta,\label{li-1}\\
\int_{\G\times S^{d-1}}\Psi^*(\overline\nabla\phi)B^\varepsilon \Theta(f^\varepsilon)\kappa^\varepsilon \dd\eta\dd\sigma\to \int_{\G} |\widetilde\nabla \phi|^2ff_*\kappa\dd\eta  \label{li-2}
\end{gather}
as $\varepsilon\to0$.
\end{lemma}

\begin{proof}
The limit \eqref{li-1} holds as a consequence of \eqref{sec-5:goal-1} and \eqref{tilde:Tg}
\begin{align*}
   \int_{\G\times S^{d-1}} U^\varepsilon_{\aB}\overline\nabla\phi\dd \eta\dd\sigma\to   2 \int_{\G} U^\delta_{\aL}\cdot \widetilde\nabla\phi\dd\eta=2 \int_{\G} \widetilde{U^\delta_{\aL}}\cdot \widetilde\nabla\phi\dd\eta.
\end{align*}

We use the notations $F^\varepsilon=f^\varepsilon f^\varepsilon_*$ and $\Theta(F^\varepsilon)=\Theta\big(F^\varepsilon,(F^\varepsilon)'\big)$. To show \eqref{li-2}, we only need to show, as $\varepsilon\to0$, 
\begin{gather}
    I_1:=\int_{\G\times S^{d-1}}\Psi^*(\overline\nabla\phi)B^\varepsilon \kappa^\varepsilon\Theta(F^\varepsilon) \dd\eta\dd\sigma\to \int_{\G\times S^{d-1}}\Psi^*(\overline\nabla\phi)B^\varepsilon \kappa^\varepsilon F^\varepsilon \dd\eta\dd\sigma,\label{DL-1}\\
    I_2:=\int_{\G\times S^{d-1}}\Psi^*(\overline\nabla\phi) F^\varepsilon B^\varepsilon\kappa^\varepsilon \dd\eta\dd\sigma\to \int_{\G} |\widetilde\nabla \phi|^2F\kappa\dd\eta.\label{DL-2}
\end{gather}
 Without loss of generality, we assume $|\overline\nabla \phi|\le 1$ and $\supp(\overline\nabla \phi)\subset B_1(0)\times S^{d-1}$, where $B_1(0)$ denotes the unit ball in $\G$.

We show the limit of $I_i$, $i=1,2$.
\begin{itemize}
    \item \textbf{(The limit of $I_1$).} We use the dominated convergence theorem. Let $\theta= \eps\chi/\pi$. We write
\begin{align*}
I_1= \int_{\G}\kappa^\varepsilon
A_0(|v-v_*|)\Theta(F^\varepsilon)\Big(\frac{1}{\pi \varepsilon^2}\int_0^{\frac{\pi}{2}}\beta(\chi)\Big(\int_{S^{d-2}_{k^\perp}}\Psi^*(\overline\nabla\phi) \dd p\Big)\dd\chi\Big)\dd\eta.
    \end{align*}

By Lemma \ref{lem:mean}, the following pointwise bounds hold
    \begin{align*}
    \Theta(F^\varepsilon)\le C_{\Psi^*} \big((F^\varepsilon)'+F^\varepsilon \big),\quad C_{\Psi^*} >0.
    \end{align*}
    Concerning the majorant of $\Psi^*(\overline\nabla\phi)$, $\Psi^*$ satisfies $\Psi^*(0)=(\Psi^*)'(0)=0$ and $(\Psi^*)''(0)=1$ by Assumption \ref{ass-pair}, and, without loss of generality, we assume $|\overline\nabla \phi|\le 1$. By Taylor expansion, we have 
\begin{equation}
\label{do-1}
\Psi^*(\overline\nabla\phi)\le \frac{\|(\Psi^*)''\|_{L^\infty([-1,1])}}{2}|\overline\nabla\phi|^2.   
    \end{equation}
By Lemma \ref{lem:phi:conv}, we have 
    \begin{equation}
    \label{do-2}
        |\overline\nabla \phi|\lesssim  \|\phi\|_{W^{1,\infty}}|v-v_*|\varepsilon^2\chi^2.
    \end{equation}
Since $\supp(\overline\nabla \phi)\subset B_1(0)\times S^{d-1}$ and  $\Psi^*(0)=0$, up to a constant independent of $\varepsilon$, the integrand of $I_1$ is dominated by 
    \begin{align*}
       &\kappa^\varepsilon
A_0(|v-v_*|)\Theta(F^\varepsilon)\Big(\frac{1}{\pi \varepsilon^2}\int_0^{\frac{\pi}{2}}\beta(\chi)\Big(\int_{S^{d-2}_{k^\perp}}\Psi^*(\overline\nabla\phi) \dd p\Big)\dd\chi\Big)\\
       \lesssim &{} 
       |v-v_*|^2A_0\chi^2\beta(\chi)\big((F^\varepsilon)'+F^\varepsilon\big) \mathbb{1}_{B_1(0)}.
    \end{align*}

    By changing of variables $(v',v_*')\mapsto (v,v_*)$, we have 
    \begin{equation}
        \label{majorant:conv}
     \begin{aligned}
         &\int_{B_1(0)\times S^{d-1}}|v-v_*|^2A_0\chi^2\beta(\chi)\big((F^\varepsilon)'+F^\varepsilon\big) \dd\eta\dd\sigma\\
         =&{}2|S^{d-2}|\Big(\int_0^{\frac{\pi}{2}}\chi^2 \beta(\chi)\dd\chi\Big)\int_{B_1(0)}|v-v_*|^2A_0(|v-v_*|)F^\varepsilon \dd\eta\\
         \lesssim &{}\int_{B_1(0)}|v-v_*|^2A_0(|v-v_*|)F^\varepsilon \dd\eta\lesssim \int_{\G}F^\varepsilon\dd\eta.
    \end{aligned}
    \end{equation}
    In the last inequality, we use $|v-v_*|^2A_0(|v-v_*|)\lesssim  |v-v_*|^{2+\gamma}\le C$, where the bounded region of integration ensures the boundedness of $|v-v_*|$, and in the soft potential case, the singularity of $A_0$ near $0$ is controlled. By Lemma \ref{weak:conv}, $F^\varepsilon$ is weakly precompact in $L^1(\G)$.  

   On the other hand, Lemma \ref{lemma:cpt:sqrt} implies the pointwise convergence
   \begin{align*}
(F_{\varepsilon})',\,F_{\varepsilon}\to F\quad\text{as}\quad \varepsilon\to0.   \end{align*}
Combining with the homogeneity assumption of $\Theta$ in Assumption \ref{ass-pair}, we have  the pointwise convergence
\begin{align*}
\Theta(F_{\varepsilon}) \to F_*\quad\text{as}\quad \varepsilon\to0.   
\end{align*}
By dominated convergence, the limit \eqref{DL-1} holds.

\item \textbf{(The limit of $I_2$).}
Similar to the $I_1$ term, we use the dominated convergence theorem. $I_2$ can be written as
    \begin{align*}
I_2= \int_{\G}\kappa^\varepsilon
A_0(|v-v_*|)F^\varepsilon\Big(\int_0^{\frac{\pi}{2}}\frac{\beta(\chi)}{\pi\varepsilon^2}\Big(\int_{S^{d-2}_{k^\perp}}\Psi^*(\overline\nabla\phi) \dd p\Big)\dd\chi\Big)\dd\eta.
    \end{align*}
    By using of \eqref{do-1} and \eqref{do-2}, 
similar to the $I_1$ term, the integrand of $I_2$ is dominated by 
    \begin{align*}
 &\kappa^\varepsilon
A_0(|v-v_*|)F^\varepsilon\Big(\int_0^{\frac{\pi}{2}}\frac{\beta(\chi)}{\pi\varepsilon^2}\Big(\int_{S^{d-2}_{k^\perp}}\Psi^*(\overline\nabla\phi) \dd p\Big)\dd\chi\Big)\\
\lesssim&{} |v-v_*|^2A_0F^\varepsilon \mathbb{1}_{B_1(0)}\int_{0}^{\frac{\pi}{2}}\chi^2 \beta(\chi)\dd\chi\lesssim |v-v_*|^2A_0F^\varepsilon \mathbb{1}_{B_1(0)},
\end{align*}
which is uniformly integrable by a similar argument as \eqref{majorant:conv}.

We are left to show the pointwise limit
the pointwise limit
  \begin{equation}
  \label{pw:limit}
\lim_{\varepsilon\to0}\int_{S^1}B^\varepsilon \Psi^*(\overline\nabla\phi)\dd \sigma =2|\widetilde \nabla\phi|^2\quad\text{for a.e. }(x,x_*,v,v_*)\in\G.
  \end{equation}
 By Taylor expansion, we have the pointwise limit $\lim_{r\to0}\frac{\Psi^*(r)}{r^2}=\frac12$. Combing the pointwise limit of bound \eqref{do-2} of $\overline\nabla\phi$ in Lemma \ref{lem:phi:conv},  we have 
 \begin{align*}
  &\lim_{\varepsilon\to0}\frac{1}{\varepsilon^2}\int_{S^{d-2}_{k^\perp}} \Psi^*(\overline\nabla\phi) \dd p =\lim_{\varepsilon\to0}\int_{S^{d-2}_{k^\perp}} \frac{|\overline\nabla\phi|^2}{2\varepsilon^2} \dd p.
  \end{align*}
  for almost every $(x,x_*,v,v_*)\in\G$. By repeating the arguments as in \eqref{Bgard:8}, we have the pointwise limit
  \begin{align*}
  &\lim_{\varepsilon\to0}\int_{S^1}B^\varepsilon |\overline\nabla\phi|^2 \dd \sigma =2|\widetilde \nabla\phi|^2.
  \end{align*}
  Hence, the limit \eqref{pw:limit} holds.
  Combining the pointwise limit of $F^\varepsilon$, the dominated convergence theorem implies the limit \eqref{DL-2}.

   \end{itemize} 
\end{proof}

\printbibliography

\Addresses

\end{document}